\documentclass[a4paper,oneside, reqno]{amsart}
\usepackage{mathrsfs}
\usepackage{amsfonts}
\usepackage{amssymb}
\usepackage{amsxtra}
\usepackage{dsfont}
\usepackage{amsthm}
\usepackage[colorlinks=true,linkcolor=blue]{hyperref}%
\usepackage[T1]{fontenc}
\usepackage[utf8]{inputenc}
\usepackage[english]{babel}
\usepackage{amsmath} 
\usepackage{bm}
\usepackage{graphicx}
\usepackage{hyperref}
\usepackage{color}
\usepackage{relsize}

\usepackage{hyperref}
\usepackage{stmaryrd}
\usepackage{listings}

\usepackage{enumitem}

\numberwithin{equation}{section}
\usepackage{url}
\newtheorem{Thm}{Theorem}[section]

\newtheorem{Cor}[Thm]{Corollary}
\newtheorem{Lem}[Thm]{Lemma}
\newtheorem{Conj}[Thm]{Conjecture}
\newtheorem{Que}[Thm]{Question}
\newtheorem{Prop}[Thm]{Proposition}
\theoremstyle{definition}
\newtheorem{Rem}[Thm]{Remark}      
\theoremstyle{definition}
\newtheorem{Defn}[Thm]{Definition}
\newtheorem{Ex}[Thm]{Example}
\newtheorem{Fact}[Thm]{Fact}
\theoremstyle{definition}
\newtheorem*{Prob*}{Problem}
      
\theoremstyle{definition}

\newcommand{\E}{\mathop{\mathbb{E}}}
\newcommand{\Pro}{\mathop{\mathbb{P}}}
\newcommand{\Z}{\mathbb{Z}}
\newcommand{\R}{\mathbb{R}}
\newcommand{\C}{\mathbb{C}}
\newcommand{\D}{\mathbb{D}}

\newcommand{\N}{\mathbb{N}}
\newcommand{\T}{\mathbb{T}}
\DeclareMathOperator{\Sym}{Sym}
\DeclareMathOperator{\sgn}{sgn}
\DeclareMathOperator{\Vol}{Vol}
\DeclareMathOperator{\supp}{supp}

\DeclareMathOperator{\conv}{conv}
\DeclareMathOperator{\mult}{mult}
\newcommand{\se}[1]{\underset{_{\substack{#1}}}{\mathlarger{\mathlarger{\E}}}}

\newcommand{\spr}[1]{\underset{_{\substack{#1}}}{\mathlarger{\mathlarger{\Pro}}}}

\title[High-dimensional discrete 1-symmetric convex bodies]{High-dimensional discrete 1-symmetric convex bodies and dimension-free estimates for maximal functions}
\author{Jakub Niksi\'nski}
\address{\scriptsize Jakub Niksiński (\textnormal{jakub.niksinski.math@gmail.com})
\newline  Department of Mathematics, Rutgers University, Piscataway, NJ 08854, USA
     \vspace{-0.25em}
}

\begin{document}
\begin{abstract}
We investigate the behavior of lattice points in 1-symmetric convex bodies—those invariant under both coordinate permutations and sign changes. In this setting we introduce a discrete analogue of the isotropic constant and establish concentration of mass properties for lattice points that parallel classical mass concentration results for convex bodies in isotropic position. These geometric estimates are applied to prove dimension-free $\ell^p$ bounds ($p \ge 2$) for discrete dyadic maximal operators and for maximal functions in the small-scale regime. Furthermore, we demonstrate that coordinate permutation invariance is a necessary condition by constructing 1-unconditional convex bodies in isotropic position for which these dimension-free estimates fail. Our results provide a framework that generalizes several recent works on dimension-free estimates for discrete maximal functions, including those by Bourgain, Mirek, Stein, and Wróbel.

\end{abstract}
\thanks{Jakub Niksiński was supported by the NSF CAREER grant DMS-2236493. This research was funded in whole or in part by National Science Centre, Poland, grant Sonata Bis 2022/46/E/ST1/00036.}

\maketitle
\tableofcontents
\section{Introduction}
\subsection{Statements of the results}
\begin{Defn}
We call $G \subseteq \R^d$ a \textit{convex body} if it is a closed, bounded, convex set with non-empty interior. We categorize such a body based on the following properties.
\begin{itemize}
    \item \textbf{Symmetry}: $G$ is \textit{symmetric} if it is centrally symmetric.
    \item \textbf{Isotropic Position}: $G$ is in \textit{isotropic position} if there exists an \textit{isotropic constant} $L(G) > 0$ such that for all $i, j \in \{1, \ldots, d\}$:$$\frac{1}{|G|} \int_G x_i x_j dx = \delta^i_j L(G)^2.$$
    \item \textbf{1-symmetry}: $G$ is \textit{1-symmetric} if it is invariant under both coordinate permutations and sign changes:
     \begin{enumerate}
        \item For any $\sigma \in \Sym(d)$, $x \in G \implies (x_{\sigma(1)}, \ldots, x_{\sigma(d)}) \in G$.
        \item For any $\epsilon \in \{-1, 1\}^d$, $x \in G \implies (\epsilon_1 x_1, \ldots, \epsilon_d x_d) \in G$.
    \end{enumerate} A body satisfying the first condition is called \textit{permutation invariant}, whereas a body satisfying the second condition (sign-invariance) is called \textit{1-unconditional}.
\end{itemize}
 We denote the family of $d$-dimensional 1-symmetric convex bodies by:$$\mathcal{F}_d = \{ G \subseteq \mathbb{R}^d : G \text{ is a 1-symmetric convex body} \}.$$
Note that any $G \in \mathcal{F}_d$ is in isotropic position. 
\end{Defn}
Natural examples of 1-symmetric convex bodies are $\ell^q$ unit balls for $q \ge 1$:
\[
B^{q,(d)}:= \Big\{x \in \R^d: |x|_q:= \Big( \sum_{i=1}^d |x_i|^q \Big)^{1/q} \le 1 \Big\}, \quad
B^{\infty,(d)}=[-1,1]^d.
\]
For any symmetric convex body $G \subseteq \R^d$ we define two constants
\[
\kappa(G):= \max_{x \in G} \frac{|x|_1}{d}, \quad \text{and} \quad\widetilde{\kappa}(G):=\max_{x \in G \cap \Z^d} \frac{|x|_1}{d},
\]
where $|x|_1= \sum_{i=1}^d |x_i|$.
For $G \in \mathcal{F}_d$ these constants will be crucial in the analysis of lattice points inside $G$, especially the second one. The constant $\kappa(G)$ is comparable to $L(G)$, as will be seen in later sections.
For $\widetilde{\kappa}(G) \gtrsim 1$ (large-scale regime), lattice points in $G$ behave similarly to the continuous mass of the body, whereas the case $\widetilde{\kappa}(G) \lesssim 1$ (small-scale regime) presents a distinct discrete nature. Thus, $\widetilde{\kappa}(G)$ effectively parametrizes the transition from continuous to discrete behavior.
We establish the following results.
\begin{Thm} \label{iso const}
For all $d \in \N$, $G \in \mathcal{F}_d$ and $p \in [1, \infty)$ we have 
\[
\frac{1}{|G \cap \Z^d|} \sum_{x \in G \cap \Z^d} |x_1|^p=\frac{1}{|G \cap \Z^d|} \sum_{x \in G \cap \Z^d} \frac{|x|_p^p}{d} \approx_p \max\Big(\widetilde{\kappa}(G),\widetilde{\kappa}(G)^p \Big),
\]
where $|x|_p=\Big(\sum_{i=1}^d |x_i|^p\Big)^{1/p}$ and $\approx_p$ means comparability from above and below with implied constants depending only on $p$.
In particular for $p=2$ we obtain a discrete analogue of the isotropic constant.
\end{Thm}
Theorem below is a discrete analogue of \cite[Theorem 1.2]{bob-naz} of Bobkov, Nazarov and \cite[Theorem 1.1]{Pa} of Paouris, but with an additional assumption of 1-symmetry.
\begin{Thm} \label{Bobkov-Nazarov main thm}
    There exists $t_0>0$, $c>0$ such that for all $d \in \N$ and $G \in \mathcal{F}_d$ the following holds.
    \begin{itemize}
        \item If $\widetilde{\kappa}(G) \ge 1/2$ and $t \ge t_0$ we have 
    \[
    \Big|\Big\{ x \in G \cap \Z^d: \frac{|x|_2}{\sqrt{d}} \ge t \widetilde{\kappa}(G) \Big\} \Big|  \le e^{-ct \sqrt{d}} |G \cap \Z^d|,
    \]
    \item If $\widetilde{\kappa}(G) \in (0, 1/2]$ and $t \ge t_0$ we have 
    \[
    \Big|\Big\{ x \in G \cap \Z^d: \frac{|x|_2}{\sqrt{d}} \ge t \sqrt{\widetilde{\kappa}(G) } \Big\} \Big|  \le  \widetilde{\kappa}(G)^{ ct\sqrt{d\widetilde{\kappa}(G)}} |G \cap \Z^d|,
    \]
    \end{itemize}
    where $|x|_2=\Big(\sum_{i=1}^d x_i^2\Big)^{1/2}.$
\end{Thm}
If $\widetilde{\kappa}(G)$ is small, we expect a typical point in $G \cap \Z^d$ to have close to $d \widetilde{\kappa}(G)$ coordinates $\pm 1$ and that its euclidean norm should not be too influenced by larger coordinates. Hence we predict that $|x|_2/ \sqrt{d} \approx \sqrt{\widetilde{\kappa}(G) }$ for most $x \in G \cap \Z^d$. The second part of Theorem \ref{Bobkov-Nazarov main thm}, up to the value of the constant $c>0$, is optimal in "highly degenerate range", that is when $\widetilde{\kappa}(G) \le d^{-1/2}$. This is explained in Remark \ref{optimality of small scale Bobkov-Nazarov}.
Note that by replacing $G$ by $N\cdot G$ and taking $N \to \infty$ one can recover \cite[Theorem 1.2]{bob-naz} and \cite[Theorem 1.1]{Pa} but under extra assumption $G \in \mathcal{F}_d$. In this fashion one can recover various known continuous results from the discrete ones. 
\par Our main application of these concentration estimates lies in the study of discrete maximal functions.
For a symmetric convex body $G \subseteq \R^d$, $f: \Z^d \to \C$, $x \in \Z^d$, $t>0$ we define the discrete averaging operator over $G$ by the formula
\[
\mathcal{M}_t^G f(x)= \frac{1}{|tG \cap \Z^d|} \sum_{y \in tG \cap \Z^d} f(x-y).
\]
We will establish three results regarding dimension-free estimates for discrete maximal operators over 1-symmetric convex bodies.
\begin{Thm} \label{dyadic dim-free}
 We have 
 \[
 \sup_{p \in [2,\infty]} \sup_{d \in \N} \sup_{G \in \mathcal{F}_d} \Big\| \sup_{n \in \Z} |\mathcal{M}_{2^n}^G| \Big\|_{\ell^p(\Z^d) \to \ell^p(\Z^d)}< \infty.
 \]
\end{Thm}
The above result vastly generalizes the works of Bourgain, Mirek, Stein, Wróbel \cite{balls} and Kosz, Mirek, Plewa, Wróbel \cite[Theorem 3]{KMPW} and Niksiński \cite{Ni2}.
\begin{Thm} \label{small scale dim-free}
 For every $\varepsilon>0$, we have 
 \[
 \sup_{p \in [2,\infty]} \sup_{d \in \N} \sup_{G \in \mathcal{F}_d} \Big\| \sup_{\substack{t>0, \\
 \widetilde{\kappa}(tG) \le d^{-\varepsilon}}} |\mathcal{M}_{t}^G| \Big\|_{\ell^p(\Z^d) \to \ell^p(\Z^d)}< \infty.
 \]
\end{Thm} 
For instance if $G=B^{q,(d)}$, then the innermost supremum is taken over $t$'s such that $t \lesssim d^{1/q-\varepsilon/q}$, which shows that the above inequality generalizes \cite[Theorem 1.1]{NW} and part of \cite[Theorem 1.15]{KNW}.
By interpolation with $p=\infty$ it suffices to only consider $p=2$ in Theorems \ref{dyadic dim-free}, \ref{small scale dim-free}. Taking $p=2$ allows us to make great use of Fourier methods.
By modifying the counterexample from \cite[Theorem 2]{cubes} we will establish that Theorems \ref{dyadic dim-free} and \ref{small scale dim-free} fail simultaneously when we consider supremum over closed 1-unconditional convex bodies in isotropic position instead of 1-symmetric convex bodies.
\begin{Prop} \label{counter-example}
    For any $K >0$ there exists $d \in \N$ and $G \subseteq \R^d$ a closed 1-unconditional convex body in isotropic position and $f: \Z^d \to \C$ such that for all $p \in (1, \infty)$ we have
    \[
    \Big\| \sup_{\substack{n \in \Z, \\ \widetilde{\kappa}(2^n G) \le d^{-1/2}}} |\mathcal{M}_{2^n}^G f| \Big\|_{\ell^p(\Z^d)} \ge \frac{1}{6} K^{1/p} \|f \|_{\ell^p(\Z^d)}.
    \]
\end{Prop}

\par The assumption that $G$ is a 1-symmetric convex body might seem incredibly restrictive, nevertheless it seems that in the discrete setting available tools are heavily limited. Examples constructed in Proposition \ref{counter-example} highlight that one needs to be careful when trying to establish discrete analogues of known continuous results.

\subsection{Known results}
\subsubsection{Asymptotic convex geometry}
Convex sets are fundamental objects in mathematics. Asymptotic convex geometry is concerned with properties of high-dimensional convex bodies with an emphasis on dependence on the dimension. In 1980 Hensley in his work \cite{He} already considered convex bodies in isotropic position, the relation of isotropic position to slicing problems and potential applications to Diophantine equations. A few years later in an influential work \cite{B1} of Bourgain on dimension-free estimates, the concept of isotropic position was introduced and rediscovered.
This notion has had a profound impact on the field of convex geometry of finite-dimensional Banach spaces. Various concentration of continuous mass properties can be established for convex bodies in isotropic position. 
\par In mid-1980s it was conjectured by Bourgain that if $|G|=1$ and $G$ is in isotropic position, then $L(G) \approx 1$ uniformly, which is equivalent to saying that there is a $c>0$ such that for all $G \subseteq \R^d$ with $|G|=1$ in isotropic position there is a hyperplane $H \subseteq \R^d$ such that 
\[
\Vol_{d-1}(H \cap G) \ge c.
\]
Around that time it was observed that the conjecture holds in the case of 1-unconditional convex bodies, see the work of Milman and Pajor \cite[b. Proposition, p. 85]{MP}.
This conjecture was fully resolved not so long ago in the breakthrough work of Klartag and Lehec \cite{iso conj}, building upon years of significant advancements in the field. See for instance \cite{He}, \cite[Section 3]{MP}, \cite[Theorem 3.1.2]{conv book} for relation between slicing by hyperplanes and isotropic constant. See for instance \cite[Lemma 4.1]{MP}, \cite[Proposition 3.3.1]{conv book} for a simple lower bound for $L(G)$.  
One of the main inspirations for this paper is the work of Bobkov and Nazarov \cite{bob-naz}, where it was shown that there are universal constants $t_0,c>0$ such that for all $d \in \N$, $G \subseteq\R^d$ a 1-unconditional convex body in isotropic position with $\Vol_d(G)=1$ and $t \ge t_0$ we have
\begin{equation} \label{eq:conc of mes bob-naz}
    \Vol_d\Big( \Big\{ x \in G: |x|_2 \ge t \sqrt{d} \Big\} \Big) \le e^{-ct\sqrt{d}}.
\end{equation}
Later this result was generalized by Paouris \cite{Pa}, who established that for any $G \subseteq \R^d$ in isotropic position with $\Vol_d(G)=1$ and any $t \ge 1$ one has
\[
\Vol_d\Big( \Big\{x \in G: |x|_2 \ge ct L(G) \sqrt{d} \Big\}\Big) \le e^{-t\sqrt{d}},
\]
where $c>0$ is a universal constant.
Prior to the above works, Schechtman and Zinn \cite{SZ} established that for $1 \le p<q< \infty$ there exists $c=c(p,q)>0$ and $T=T(p,q)>0$ such that for all $d\in \N$ and $t \ge T$ we have 
\[
 \Vol_d\Big(\Big\{ u \in B^{p,(d)}: \frac{|u|_q}{d^{1/q}} \ge t \frac{|u|_p}{d^{1/p}}\Big\} \Big) \le  e^{-ct^pd^{p/q}}\Vol_d(B^{p,(d)}),
\]
moreover they obtained inequality in the opposite direction for $t \in [2, d^{1/p-1/q}/2]$ but with different constant $c>0$. In \cite{SZ2} the same authors generalized this type of deviation inequality to Lipschitz functions instead of $\ell^q$-norms.
We refer readers interested in asymptotic convex geometry to the monograph \cite{conv book} of Brazitikos, Giannopoulos, Valettas and Vritsiou. The term \textit{1-symmetric} comes from the theory of Banach spaces.
\par Lattice points in convex sets have been extensively studied; see, for instance, \cite{matusek}. Related applications of isotropic position to Diophantine equations can be found in \cite{He}, while discrete variants of the Brunn--Minkowski inequality are studied in \cite{INZ,HKS}, in \cite{HSX} relation between packing minima and lattice points in convex bodies was considered.
Nevertheless the author is unaware of works that study mass concentration properties of lattice points in convex bodies and establish results analogous to the continuous ones that arise for convex bodies in isotropic position.
\par 
Apart from motivation, our diverge from problems considered in asymptotic geometry of convex isotropic bodies. Assumption that $G$ is 1-symmetric vastly simplifies many problems considered in the aforementioned part of mathematics. On the other hand considering lattice points seems to heavily restrict the applicability methods of analysis and probability, which are fundamental in the study of convex bodies.

\subsubsection{Dimension-free estimates for maximal functions}
Let $G \subseteq \R^d$ be a symmetric convex body. For such set we define the corresponding averaging operator
\begin{equation} \label{cont avr op defn}
   M_t^Gf(x)= \frac{1}{\Vol_d(tG)} \int_{tG} f(x-y)dy, \quad f \in L^1_{loc}(\R^d) 
\end{equation}
where $tG=\{tx: x \in G\}$. Let $C(p,G)>0$ be the best constant such that the following inequality holds 
\[
\Big\| \sup_{t>0} |M_t^Gf| \Big\|_{L^p(\R^d)} \le C(p,G) \|f \|_{L^p(\R^d)}.
\]
Classical arguments show that for every $p \in (1,\infty]$ and symmetric convex body $G$ we have $C(p,G)<\infty$. In 1982 Stein \cite{SteinMax} (see also \cite{StStr}) discovered that for all $p>1$ the constant $C(p,B^{2,(d)})$ can be bounded uniformly in $d \in \N$. Later Bourgain in \cite{B1} showed that one can bound $C(2,G)$ uniformly with respect to all symmetric convex bodies. Consequently Bourgain \cite{B2} and independently Carbery \cite{Car1} extended this result to $C(p,G) \le C_p$ for all $p>3/2$. Later Müller \cite{Mul1} for all $p>1$ and $q \in [1, \infty)$ established existence of $C_{p,q}>0$ such that $C(p,B^{q,(d)}) \le C_{p,q}$ holds, this result was completed by Bourgain \cite{B3} who proved the same conclusion for $q=\infty$.
For a broader perspective on dimension-free estimates for
continuous maximal functions, we refer the reader to the survey articles \cite{BMSW4} and \cite{HL-cont}.
\par The study of dimension-free estimates for discrete maximal functions was initiated by Bourgain, Mirek, Stein, Wróbel in \cite{cubes} and \cite{balls}.
Results in the discrete setting are not as fruitful.
For a symmetric convex body $G \subseteq \R^d$ and $f: \Z^d \to \C$ let 
\[
\mathcal{M}_t^G f(x)= \frac{1}{|tG \cap \Z^d|} \sum_{y \in tG \cap \Z^d} f(x-y).
\]
Let us give a brief overview of almost everything that is currently known in this area. Let $\mathbb{D}= \{2^n:n \in \N_0\}$.
\begin{itemize}
    \item In \cite{cubes} it was proved that for every $p \in ( \frac{3}{2}, \infty ]$ there exists a constant $C_p>0$ such that for every $d \in \N$ we have
    \[
    \Big\| \sup_{t>0} \big|\mathcal{M}_t^{B^{\infty,(d)}}\big| \Big\|_{\ell^p(\Z ^d) \to \ell^p(\Z ^d)} \leq C_p .
    \]
    This result is almost as strong as those in the continuous case. In the case of sets $B^q$ for $q \neq \infty$ authors of papers \cite{cubes}, \cite{balls}, \cite{KMPW}  could only obtain weaker conclusions. Moreover in the same work authors constructed a family of ellipsoids $E(d) \subseteq \R^d$, such that for some $f_d : \Z^d \to \C$ and all $p \in (1, \infty)$ one has 
    \[
     \Big\| \sup_{t>0} \big| \mathcal{M}_t^{E(d)} f_d \big| \Big\|_{\ell^p(\Z^d)} \ge C_p \log(d)^{1/p} \|f_d \|_{\ell^p(\Z^d)}.
    \]
    Moreover for any symmetric convex body $G \subseteq \R^d$ the following comparison principle was established
    \[
    \Big\| \sup_{t \ge c(G) d} \big| \mathcal{M}_t^G  \big| \Big\|_{\ell^p(\Z^d) \to \ell^p(\Z^d)} \le e^6 \Big\| \sup_{t>0} \big|M_t^G\big| \Big\|_{L^p(\R^d) \to L^p(\R^d) } \le C,
    \]
    where $c(G)= \inf\{t>0: [-1/2,1/2]^d \subseteq tG \}$.
    \item When it comes to Euclidean balls, in \cite{balls} it was proved that there exists $C>0$ such that for every $p \in [2, \infty)$, $d \in \N$ we have
    \[
    \Big\| \sup_{t \in \D} \big|\mathcal{M}_t^{B^{2,(d)}} \big| \Big\|_{\ell^p(\Z ^d) \to \ell^p(\Z ^d)} \leq C.
    \]
    \item In \cite{KMPW} by extending methods of \cite{balls}  it was proved that for any $q \in (2, \infty)$ there a exists constant $C(q)>0$ 
    such that for every $p \in [2, \infty)$, $d \in \N$ we have
    \[
    \Big\| \sup_{t \in \D, t \geq d^{1/q}} \big|\mathcal{M}_t^{B^{q,(d)}} \big| \Big\|_{\ell^p(\Z ^d) \to \ell^p(\Z ^d)} \leq C(q).
    \]
    Moreover it was shown that for all $q \in (2, \infty)$ and $p \in (1, \infty)$ one has 
    \[
    \Big\| \sup_{t \ge d} \big|\mathcal{M}_t^{B^{q,(d)}} \big| \Big\|_{\ell^p(\Z ^d) \to \ell^p(\Z ^d)} \leq C(p,q).
    \]
    The paper \cite{KMPW} did not cover the range $t<d^{1/q}$ nor $q<2$.
    It was also shown that for every symmetric convex body $G \subseteq \R^d$ one has
    \[
    \Big\| \sup_{t>0}\big|M_t^G\big| \Big\|_{L^p(\R^d) \to L^p(\R^d)} \le \Big\| \sup_{t>0}\big|\mathcal{M}_t^G\big| \Big\|_{\ell^p(\Z^d) \to \ell^p(\Z^d) }
    \]
    for all $p \in (1, \infty)$.
    \item  In \cite{NW}, the author in collaboration with B. Wróbel has shown that for any $\varepsilon>0$ there is a constant $C_{\varepsilon}>0$ such that for all $p \in [2,\infty)$, $d \in \N$ we have
    \[
    \Big\| \sup_{0 \le t \le d^{\frac{1}{2}-\varepsilon}} \big|\mathcal{M}_t^{B^{2,(d)}}\big| \Big\|_{\ell^p(\Z^d) \to \ell^p(\Z ^d)} \le C_{\varepsilon}.
    \]
    \item In \cite{Ni2} the author has established that there is $C>0$ such that for all $q \in [1, \infty)$, $p \in [2, \infty)$, $d \in \N$ one has
    \[
    \Big\| \sup_{t \in \D, t \le d^{1/q}} \big| \mathcal{M}_t^{B^{q,(d)}} \big| \Big\|_{\ell^p(\Z^d) \to \ell^p(\Z ^d)} \le C
    \]
    \item In \cite{KNW} the author in collaboration with D. Kosz and B. Wróbel has shown that for all $d \in \N$ we have 
    \[
    \Big\| \sup_{t \in \D} \big| \mathcal{M}_t^{B^{1,(d)}} \big| \Big\|_{\ell^p(\Z^d) \to \ell^p(\Z ^d)} \le C ,
    \]
    \[
    \Big\| \sup_{t \ge d^{3/2}} \big| \mathcal{M}_t^{B^{1,(d)}} \big| \Big\|_{\ell^p(\Z^d) \to \ell^p(\Z ^d)} \le C 
    \]
    and
    \[
    \Big\| \sup_{t \le d^{1- \varepsilon}} \big| \mathcal{M}_t^{B^{1,(d)}} \big| \Big\|_{\ell^p(\Z^d) \to \ell^p(\Z ^d)} \le C_\varepsilon .
    \]
\end{itemize}
\subsection{Questions and conjectures}
We state some questions and conjectures, which are relevant to the present work.
From the perspective of Theorems \ref{dyadic dim-free}, \ref{small scale dim-free} it seems natural to conjecture the following.
\begin{Conj}\label{dim free conjecture}
    We have 
 \[
 \sup_{p \in [2,\infty]} \sup_{d \in \N} \sup_{G \in \mathcal{F}_d} \Big\| \sup_{t>0} |\mathcal{M}_{t}^G| \Big\|_{\ell^p(\Z^d) \to \ell^p(\Z^d)}< \infty.
 \] 
\end{Conj}
In the mid 1990s Stein asked the following question.
\begin{Que}
    Is it true that 
    \[
    \sup_{d \in \N}  \Big\| \sup_{t>0} |\mathcal{M}_{t}^{B^{2,(d)}}| \Big\|_{\ell^2(\Z^d) \to \ell^2(\Z^d)}< \infty ?
    \]
\end{Que}
The problems above should require much deeper understanding of the sets $tG \cap \Z^d$ as $t$ varies. When obtaining dimension-free estimates for dyadic maximal function it suffices to establish certain bounds for $|\mathfrak{m}_{tG}(\xi)|$ and $|\mathfrak{m}_{tG}(\xi)-1|$, where 
\[
\mathfrak{m}_{tG}(\xi)= \frac{1}{|tG \cap \Z^d|} \sum_{x \in tG \cap \Z^d} e(x \cdot \xi)
\]
and $\xi \in \T^d$, $x \cdot \xi= \sum_{i=1}^d x_i \xi_i$. If one wants to apply the same strategy to establish dimension-free estimates for the full maximal function, then one needs to properly bound $|\mathfrak{m}_{t_1G}(\xi)-\mathfrak{m}_{t_2G}(\xi)|$ for $t_1 \neq t_2$. This seems difficult, because the structure of set $t G \cap \Z^d$ is unclear, let alone its size. In \cite{NW} this problem was bypassed by establishing technical bound $|m_{t_1}(\xi)-m_{t_2}(\xi)|$ for multiplier $\mathfrak{m}$ corresponding to different operator, which very loosely speaking could be thought as an average of various low-dimensional averaging operators on Hamming cube-like structure. Such a reduction was only possible in small-scale regime, where one can exploit the fact that lattice points should have almost all coordinates in some uniformly bounded set, see Lemma \ref{fun fact for sup sym conv body}.
\par We finish this part with the following question.
\begin{Que}
    Is it true that for all $K>0$ there exists $d \in \N$ and $G \subseteq \R^d$ a closed symmetric convex body, which is permutation invariant such that 
    \[
    \Big\| \sup_{t>0} |\mathcal{M}_t^G| \Big\|_{\ell^2(\Z^d) \to \ell^2(\Z^d)} \ge K.
    \]
\end{Que}
We conjecture, although we currently have no supporting evidence, that the question above has affirmative answer. Obtaining such $G$ for large $K$ would be interesting because it would highlight different degenerate behavior, which arises in discrete setting, than the one that is exploited in Proposition \ref{counter-example} and \cite[Theorem 2]{cubes}.
\subsection{Notation}
 We denote by $\N= \{1,2,... \}$ the set of positive integers and by $\N_0=\N \cup \{ 0\}$ the set of non-negative integers.
 For $N\in \N$ we abbreviate $[N]=\{1,\dots,N\}.$
$\Sym(d)$ will denote the set of permutations of $[d]$. For $x \in \R^d, \ \sigma \in \Sym(d), \epsilon \in \{-1,1\}^d$ we define 
\[
\sigma \cdot x= (x_{\sigma(1)}, \ldots x_{\sigma(d)}), \quad \epsilon \cdot x= (\epsilon_1 x_1, \ldots, \epsilon_d x_d).
\]
For $t \in \R$ let $e(t)=e^{2\pi i t}$.
For $x \in \R^d$ and $p \in [1,\infty)$ we define
\[
|x|_p= \Big(\sum_{i=1}^d |x_i|^p \Big)^{1/p}, \quad |x|_{\infty}=\max_{i \in [d]} |x_i|.
\]
 For a symmetric convex body $G \subseteq \R^d$, $x \in \R^d$ and $\theta \in S^{d-1}$ we define 
\[
\| x\|_G= \inf\Big\{ \lambda>0 : x \in \lambda G \Big\},
\]
\[
\rho_G(\theta)= \sup\Big\{t>0: t \theta \in G \Big\}=\|\theta\|_G^{-1}.
\]
Note for that any $f: \R^d \to \C$, the polar coordinates formula for integral over $G$ has the following form
\[
\int_{G} f(x) dx= \int_{S^{d-1}} \int_0^{\rho_G(\theta)} f(r \theta) r^{d-1} dr d\sigma(\theta),
\]
where $\sigma$ is the ordinary, unnormalized surface measure of $(d-1)$-dimensional sphere.
 For two sets $X,Y$, by $X^Y$ we will denote the set of functions from $Y$ to $X$.
 For $f\in \ell^1(\Z^d)$ we let 
\begin{equation*}
 \hat{f}(\xi) := \sum_{n\in\Z^d} f(n) e(n\cdot \xi),\qquad \xi\in\T^d,
\end{equation*}
be its Fourier series. $\mathcal{F}$ will denote Fourier transform on $\R^d$. $\mathcal{F}^{-1}$ will denote inverse Fourier transform on $\R^d$ or $\T^d$, depending on the context. 
 For two nonnegative quantities $X, Y$
we write $X \lesssim_{\delta} Y$ if there is an absolute constant
$C_{\delta}>0$ depending only on $\delta>0$ such that $X\le C_{\delta}Y$ .
We  write $X \approx_{\delta} Y$ when
$X \lesssim_{\delta} Y\lesssim_{\delta} X$. We will omit the subscript
$\delta$ if the implicit constants are universal.
For $x \in \mathbb{R}^d$ we define
\[
\supp(x)=\{i \in [d]: x_i \neq 0\}.
\]
For $\xi\in \R^d$ or $\xi \in \T^d$ we let $\| \xi \|^2:=\sum_{j=1}^d \sin^2(\pi \xi_j).$ Furthermore we abbreviate \[\xi+1/2=(\xi_1+1/2,\ldots,\xi_d+1/2),\qquad \|\xi +1/2 \|:=\|(\xi_1+1/2,\ldots,\xi_d+1/2)\|.\] We will frequently use the fact that
$\|\xi +1/2 \|^2= \sum_{j=1}^d \cos^2(\pi \xi_j)$. Note that for $\xi \in \R$ we have $\sin^2(\pi \xi) \approx \min_{k \in \Z} |\xi-k|^2$.
 Let $A$ be a finite set of $|A|$ elements and let $h\colon A\to \mathbb{C}.$ We abbreviate \[\se{x \in A}h(x):=\frac{1}{|A|}\sum_{x\in A} h(x).\]
 Moreover, if $h=\mathds{1}_B$ is indicator function of some set $B$ we define
 \[
 \spr{x \in A}(x \in B)= \se{x \in A}\mathds{1}_B(x),
 \]
 For a symmetric convex body $G \subseteq \R^d$ we define
\[
\mathfrak{m}_G(\xi)= \se{x \in G \cap \Z^d} e(x \cdot \xi),
\]
where $\xi \in \T^d$.


\section{Basic definitions and facts}
In this section we introduce basic definitions and properties of 1-symmetric convex bodies.
\begin{Lem} \label{lem:kappa_def}
    For $G \in \mathcal{F}_d$ we have
    \[
    \kappa(G)= \max\Big\{a \in \R : (a,a,...,a) \in G \Big\}=\max\Big\{l \in \R : [-l,l]^d \subseteq G \Big\}.
    \]
\end{Lem}
\begin{proof}
It is easy to see that 
\[
\max\Big\{a \in \R : (a,a,...,a) \in G \Big\} \le \max\Big\{l \in \R : [-l,l]^d \subseteq G \Big\} \le \kappa(G).
\]
On the other hand if $\overrightarrow{a}:=(a, \ldots ,a) \in G$, then for any $\epsilon \in \{-1,1\}^d$ we have $\epsilon \cdot \overrightarrow{a} \in G $, hence 
\[
[-a,a]^d \subseteq G,
\]
since vectors $\epsilon \cdot \overrightarrow{a}$ are the extremal points of the cube $[-a,a]^d$. Thus 
\[
\max\Big\{a \in \R : (a,a,...,a) \in G \Big\} = \max\Big\{l \in \R : [-l,l]^d \subseteq G \Big\}.
\]
Take $x \in G$ such that $\kappa(G)= |x|_1/d$. By 1-symmetry, replacing $x$ with $(|x_1|, \ldots, |x_d|)$, we may assume that $x_i \ge 0$ for all $i \in [d]$. Hence, by permutation invariance and convexity, we have
\[
\overrightarrow{\kappa(G)}= \se{\sigma \in \Sym(d)} \sigma \cdot x \in G.
\]
Thus $\kappa(G)= |x|_1/d \le \max\Big\{a \in \R : (a,a,...,a) \in G \Big\}$.
\end{proof}
\begin{Rem}
 Note that $\kappa(tG)=t\kappa(G)$ for $t>0$.
We have for example $\kappa(B_N^{q,(d)})= \frac{N}{d^{1/q}}.$ Moreover
\begin{equation} \label{cont_incl}
   [-\kappa(G),\kappa(G)]^d \subseteq G \subseteq (d \kappa(G) ) \cdot B^{1,(d)},
\end{equation}
from the above we see that $|G| = c_G^d \kappa(G)^d$ for some $c_G \in [2,2e]$. As inaccurate as it might be, it is actually sufficient for some purposes.
A version of \eqref{cont_incl} under less restrictive conditions on $G$ follows from \cite[Propositions 2.4, 2.5 and (2.2)]{bob-naz}, mainly we have 
\[
\Big[- \frac{1}{\sqrt{2}}L(G),\frac{1}{\sqrt{2}}L(G)\Big]^d \subseteq G \subseteq \big(\sqrt{3 \pi e}   L(G) d\big) B^{1,(d)}.
\]
Above inclusions combined with Lemma \ref{lem:kappa_def} show that
\begin{equation} \label{comparability of LG and kappaG}
    \frac{L(G)}{\sqrt{2}} \le \kappa(G) \le \sqrt{3 \pi e} L(G).
\end{equation}
Note that $L(G) \approx \kappa(G)$ can be recovered from Theorem \ref{iso const} by replacing $G$ with $NG$ for $N \in \N$ and taking $N \to \infty$.
In the future we shall use the fact that for $A,B \in \mathcal{F}_d$ we have $A \cap B \in \mathcal{F}_d$ and
\[
\kappa(A \cap B) = \min\big( \kappa(A), \kappa(B) \big).
\]
\end{Rem}
Now let us investigate the discrete counterpart of the constant $\kappa(G)$, that is $\widetilde{\kappa}(G)=\max_{x \in G \cap \Z^d} \frac{|x|_1}{d}$. One has the following simple, yet crucial observation, which explains why $\widetilde{\kappa}(G)$ properly captures degenerate behavior in small-scale regime.
\begin{Lem} \label{kappa tilde property}
    For $G \in \mathcal{F}_d$ we have
    \[
    \min\Big(d,d \cdot\widetilde{\kappa}(G) \Big)= \max_{n \in \N_0} \Big\{ (\underbrace{1,1,\ldots,1}_{n\text{-times}},0,...,0) \in G \Big\}.
    \]
\end{Lem}
\begin{proof}
    It is easy to see that the right-hand side is bounded by the left-hand side, so we will consider inequality in the opposite direction only. Let $a=d \widetilde{\kappa}(G)= |x|_1$ for some $x \in G \cap \Z^d$. \\
    \textbf{1) If $a<d$.} \\
    Since $x \in \Z^d$ we have that $| \supp(x)| \le a$, by 1-symmetry of $G$ without loss of generality we can assume that $\supp(x) \subseteq [a]$ and that $x_i \ge 0$ for all $i \in [d]$. Then consider $v \in \R^d$ given by 
    \[
    v= \se{\sigma \in \Sym(a)} \sigma \cdot x \in G.
    \]
    (Note every $\sigma \in \Sym(a)$ can be extended, say by identity to $\widetilde{\sigma} \in \Sym(d)$, this way the above notation makes sense).
    Then we have $v_i=0$ for $i>a$ and $v_i=1$ for $i \le a$, hence
    \[
    a=d\widetilde{\kappa}(G) \le  \max_{n \in \N_0} \Big\{ (\underbrace{1,1,\ldots,1}_{n\text{-times}},0,...,0) \in G \Big\}.
    \]
    \textbf{2) If $a \ge d$.}
    We define 
    \[
    v=\se{\sigma \in \Sym(d)} \sigma \cdot x \in G,
    \]
    then $v= \overrightarrow{\frac{a}{d}}$ by taking convex combination with $\overrightarrow{0}$, we get $\overrightarrow{1} \in G$.
    Hence
    \[
    d=\max_{n \in \N_0} \Big\{ (\underbrace{1,1,\ldots,1}_{n\text{-times}},0,...,0) \in G \Big\}.
    \]
\end{proof}
Moving forward, we shall say that $G$ is in the small-scale regime (or simply small) if $\widetilde{\kappa}(G) \lesssim 1$. Conversely, we refer to the case $\widetilde{\kappa}(G) \gtrsim 1$ as the large-scale regime. 
\begin{Rem}
    In the proof we used $x \in \Z^d$ only to make deduction $|x|_1 \le a \implies |\supp(x)| \le a$. \\
    Note that for Lemma \ref{kappa tilde property} we have 
    \begin{equation} \label{disc_incl}
        d \widetilde{\kappa}(G) B^{1,(d)} \cap \{-1,0,1 \}^d \subseteq G \cap \Z^d \subseteq d \widetilde{\kappa}(G) B^{1,(d)} \cap \Z^d,
    \end{equation}
    the above serves as a substitute of \eqref{cont_incl} in small-scale regime. Inclusions \eqref{disc_incl} combined with \cite[Corollary 1.8]{KNW} show that if $\widetilde{\kappa}(G) \le d^{-1/2}$, then 
    \[
    |G \cap \Z^d| \approx |d \widetilde{\kappa}(G) B^{1,(d)} \cap \{-1,0,1 \}^d|= 2^{d \widetilde{\kappa}(G)} \binom{d}{d \widetilde{\kappa}(G)}.
    \]
    The constant $\widetilde{\kappa}(G)$ does not behave well under arbitrary scaling, for instance if $G=[-0.99,0.99]^d$, then
    \[
    \widetilde{\kappa}(G)=0, \quad \text{but} \quad \widetilde{\kappa}\Big(\frac{100}{99}G \Big)=1.
    \]
    However we have the following substitute 
    \[
    \widetilde{\kappa}(tG) \ge \lfloor t \rfloor \widetilde{\kappa}(G),
    \]
    where $t>0$. This implies for instance that sequence $\widetilde{\kappa}(2^nG)$ for $n \in \Z$ is a lacunary sequence. Note that if $N \in \N$, $N\le d^{1/q}$, then
    \[
    \widetilde{\kappa}(B_N^{q,(d)})= \frac{\lfloor N^q \rfloor}{d}.
    \]
\end{Rem}
    In the large-scale regime we have $\widetilde{\kappa}(G) \approx\kappa(G)$ as is shown in the following lemma. 
\begin{Lem} \label{tilde kappa comparable to kappa in large scale}
 For any $G \in \mathcal{F}_d$ we have
 \[
 \widetilde{\kappa}(G) \le \kappa(G) \le \widetilde{\kappa}(G)+1.
 \]
\end{Lem}
\begin{proof}
First inequality is obvious, for the second one, let $x \in G$ be s.t. $\kappa(G)=|x|_1/d$ and let $y \in G \cap \Z^d$ be defined by 
\[
y_i= \lfloor |x_i| \rfloor.
\]
Then one obtains
\[
d\kappa(G)=|x|_1 = \sum_{i=1}^d |x_i| \le  \sum_{i=1}^d (y_i+1) \le d \big( \widetilde{\kappa}(G)+1 \big).
\]
\end{proof}
Let us recall a simple upper bound for number of lattice points in $\ell^q$ balls, which will be frequently used.
\begin{Lem}[{\cite[Corollary 2.2]{Ni2}}] \label{Cor 2.2 from my lq paper}
    For any $q \ge 1$, $d \in \N$, $N>0$ we have the following bound:
    \[
    |NB^{q,(d)} \cap \Z^d| \le 2 \Big( \frac{N}{d^{1/q}}+ \frac{1}{2} \Big)^d \cdot 8^d.
    \]
    In particular combining inequality above with \eqref{disc_incl}, we obtain that for any $G \in \mathcal{F}_d$ we have
    \[
    |G \cap \Z^d| \le 2 \Big( \widetilde{\kappa}(G)+ \frac{1}{2} \Big)^d \cdot 8^d.
    \]
\end{Lem}
\subsection{Projections and extensions of 1-symmetric convex bodies}
For $G \in \mathcal{F}_d$, $i \in [d]$ we define $\pi_i: \R^d \to \R^{d-1}$ to be the projection onto hyperplane orthogonal to $e_i$. Moreover for $i \in [d+1]$ we define $E_i: \R^{d} \to \R^{d+1}$ to be it's "adjoint" by the formula
\[
E_i(y_1,...,y_{d})=(y_1,...,y_{i-1},0,y_i,y_{i+1},...,y_{d}) \in \R^{d+1}.
\]
\begin{Defn}
For $G \in \mathcal{F}_d$ we define $E(G) \subseteq \R^{d+1}$ by
\[
E(G)= \conv\Big( \Big\lbrace E_ix : x \in G, i \in [d+1] \Big\} \Big).
\]
Moreover for $k \in \N_0$ we define 
\[
G^{(d+k)}=\underbrace{E \circ E \circ \ldots \circ E}_{k\text{-times}} (G).
\]
\end{Defn}
The extension operator $E$ creates higher dimensional object exactly the same way that $B^{1,(d+1)}$ is constructed from $B^{1,(d)}$, that is $E(B^{1,(d)})=B^{1,(d+1)}$. This is not the case for other $\ell^q$ balls. We shall use both extension and projections of convex bodies later, hence we present basic facts regarding them.
\begin{Fact} \label{fact about proj and ext}
 Let $G \in \mathcal{F}_d$, $i \in [d]$. We have
 \begin{itemize}
     \item 
     \[
     \pi_i(G) \in \mathcal{F}_{d-1}, \quad \text{and} \quad E(G) \in \mathcal{F}_{d+1},
     \]
     \item 
     \[
     E\big(\pi_i(G) \big) \subseteq G,
     \]
     \item 
     \[
     \kappa(G) \le \kappa\big(\pi_i(G)\big) \le \frac{d}{d-1} \kappa(G), \quad \kappa\big(E(G)\big)= \frac{d}{d+1} \kappa(G).
     \]
     \item \[
     \widetilde{\kappa}(G) \le \widetilde{\kappa}\big(\pi_i(G)\big) \le \frac{d}{d-1} \widetilde{\kappa}(G), \quad  \frac{d}{d+1} \widetilde{\kappa}(G) \le \widetilde{\kappa}\big(E(G)\big) \le \widetilde{\kappa}(G).
     \]
 \end{itemize}
\end{Fact}
Thee fact that $\kappa(G)$ is stable under projections/extensions will be used in later sections.

\section{Lattice points in large-scale regime}
In this section we will prove the first part of Theorem \ref{Bobkov-Nazarov main thm} and part of Theorem \ref{iso const}.
We start with partial result towards concentration in $G \cap \Z^d$, which is just a generalization of \cite[Lemma 2.4]{balls} and says that we see the expected behavior on large portion of coordinates.
\begin{Lem} \label{pos_prop_conc}
For every $d \in \N, G \in \mathcal{F}_d$ such that $\kappa(G) \ge 10$ we have
\[
\spr{x \in G \cap \Z^d} \Big(\big| \{ i \in [d]: |x_i| \ge  \kappa(G)/10 \}\big| \le  d/10\Big) \le 2 e^{-d/10} 
\]
\end{Lem}
\begin{proof}
Let $\varepsilon=1/10$, define
\[
E=\Big\{ x \in G \cap \Z^d: | \{ i \in [d]: |x_i| \ge \varepsilon \kappa(G) \}| \le \varepsilon d \Big\}
\]
and for $x \in G \cap \Z^d$ let
\[
I_x= \big\{ i \in [d]: |x_i| \ge \varepsilon \kappa(G) \big\}.
\]
Note that if $x \in E$, then $|I_x| \le \varepsilon d$ and hence we have 
\[
E \subseteq \bigcup_{\substack{I \subseteq [d], \\ |I| \le \varepsilon d}} \Big\{ x\in G \cap \Z^d: I_x=I \Big\}.
\]
Moreover by the definition of $\kappa(G)$ we have
\[
\Big|  \Big\{ x\in G \cap \Z^d: I_x=I \Big\} \Big| \le (2 \varepsilon \kappa(G)+1)^{d-|I|} |d \kappa(G) B^{1,(|I|)} \cap \Z^{|I|}|.
\]
Using the above and Lemma \ref{Cor 2.2 from my lq paper} for $q=1$ we get
\begin{equation}
    \begin{split}
    &|E| \le \sum_{\substack{I \subseteq [d], \\ |I| \le \varepsilon d}}  (2 \varepsilon \kappa(G)+1)^{d-|I|} |d \kappa(G) B^{1,(|I|)} \cap \Z^{|I|}|\\
    &= \sum_{m \le \varepsilon d} \binom{d}{m}(2 \varepsilon \kappa(G)+1)^{d-m} |d \kappa(G) B^{1,(m)} \cap \Z^{m}| \\
    &\le (2\varepsilon \kappa(G)+1)^d+ 2 \sum_{1 \le m \le \varepsilon d} \binom{d}{m}(2 \varepsilon \kappa(G)+1)^{d-m} \Big(\frac{d\kappa(G)}{m}+\frac{1}{2}\Big)^m 8^m \\
    &\le (3\varepsilon \kappa(G))^d+2 \sum_{1 \le m \le \varepsilon d}\frac{d^m}{m!}(2 \varepsilon \kappa(G)+1)^{d-m} \Big(\frac{d\kappa(G)}{m}+\frac{1}{2}\Big)^m 8^m \\
    &\le (3\varepsilon \kappa(G))^d+2 \sum_{1 \le m \le \varepsilon d}\frac{d^m}{m!}(3 \varepsilon \kappa(G))^{d-m} \Big(2\frac{d\kappa(G)}{m}\Big)^m 8^m \\
    &\le (3\varepsilon \kappa(G))^d+2 \cdot (3\varepsilon)^{d-\varepsilon d} \kappa(G)^d \sum_{1 \le m \le \varepsilon d}  \Big( \frac{4 e^{1/2}d}{m} \Big)^{2m} \le 2 \kappa(G)^d \Big( 3 \varepsilon \cdot \Big(\frac{16e}{3\varepsilon^3} \Big)^{\varepsilon} \Big)^d (1+ \varepsilon d) \\
    &\le 2 \kappa(G)^d (1+ \frac{d}{10}) e^{-d \cdot 0.2457...} \le 2 e^{-d/10}  \kappa(G)^d \le 2 e^{-d/10} |G \cap \Z^d|
    \end{split}
\end{equation}
We've used the fact that $\kappa(G) \ge 10$ and that function $(0,a/e] \ni t \mapsto (a/t)^t$ is increasing and calculation
\[
3 \varepsilon \cdot \Big(\frac{16e}{3\varepsilon^3} \Big)^{\varepsilon}=e^{- 0.2457...},
\]
and $[-\kappa(G), \kappa(G)]^d \subseteq G$.
\end{proof}
\begin{Lem}[Number of lattice points in projection of convex body] \label{dim_comp}
    Consider $d \in \N$ and $G \in \mathcal{F}_d$. Let $\pi_1 : \R^d \to \R^{d-1}$ be a projection onto hyperplane orthogonal to the vector $e_1=(1,0,0 \ldots, 0)$.
    \begin{itemize}
        \item If $d \ge 20$ and $\widetilde{\kappa}(G) \ge 1/6$, then we have
    \[
   \frac{1}{20} \widetilde{\kappa}(G)^{-1} |G \cap \Z^d| \le | \pi_1(G) \cap \Z^{d-1} | \le 144 \widetilde{\kappa}(G)^{-1} |G \cap \Z^d|.
    \]
    \item If $\widetilde{\kappa}(G) \le 1/6$, then we have
    \[
   \frac{1}{4} |G \cap \Z^d| \le | \pi_1(G) \cap \Z^{d-1} | \le  |G \cap \Z^d|.
    \]
    \end{itemize}
\end{Lem}
Lemma above is a discrete analogue of result that for any $K \subseteq \R^d$ symmetric convex body in isotropic position and $\theta \in S^{d-1}$ we have
    \[
    \Vol_{d-1}(K \cap \theta^\perp) \approx L(K)^{-1} \Vol_d(K),
    \]
    where $\theta^\perp$ is the hyperplane orthogonal to $\theta$.
Second part of the theorem should not be surprising, if $\widetilde{\kappa}(G) \le 1/6$, then most of the coordinates of any $x \in G \cap \Z^d$ are equal to $0$.
\begin{proof}[Proof of Lemma \ref{dim_comp}]
We will prove two inequalities separately.
\\
\textbf{1) $ |G \cap \Z^d| \le (1+4 \widetilde{\kappa}(G)) | \pi_1(G) \cap \Z^{d-1} |$.} \\
For $x \in G \cap \Z^d$ let 
\[
S_x=\Big\{ i \in [d]: |x_i| \le 2\widetilde{\kappa}(G) \Big\}, 
\]
then for every $x \in G \cap \Z^d$ we have
\[
d \widetilde{\kappa}(G) \ge |x|_1 > 2 \widetilde{\kappa}(G) \big(d-|S_x|\big) \implies |S_x| >d/2.
\]
By an invariance under permutations argument we get
\[
|G \cap \Z^d| \le 2 \big|\big\{ x \in G \cap \Z^d : |x_1| \le 2 \widetilde{\kappa}(G)\big\} \big| \le 2(1+4 \widetilde{\kappa}(G)) |\pi_1(G) \cap \Z^{d-1}|.
\]
This shows that if $\widetilde{\kappa}(G) \ge 1/6$, then we have 
\[
|G \cap \Z^d| \le 20\widetilde{\kappa}(G) | \pi_1(G) \cap \Z^{d-1} |
\]
and if $\widetilde{\kappa}(G) \le 1/6$, then we have
\[
|G \cap \Z^d| \le 4 | \pi_1(G) \cap \Z^{d-1} |.
\]

\textbf{2) $ \min\big(1,\frac{1}{144}\kappa(G)\big) |\pi_1(G) \cap \Z^{d-1}| \le  |G \cap \Z^d|$}. \\
\\
First of all we always trivially have
\[
|\pi_1(G) \cap \Z^{d-1}| \le  |G \cap \Z^d|,
\]
this proves inequality $\min\big(1,\frac{1}{144}\widetilde{\kappa}(G)\big)|\pi_1(G) \cap \Z^{d-1}| \le  |G \cap \Z^d|$ in the case $\widetilde{\kappa}(G) \le 144 $. \par
Let us now assume that $\widetilde{\kappa}(G) \ge 144$ and $d \ge 20$.
For $x \in \pi_1(G) \cap \Z^{d-1} $ let 
\[
Z_x= \big\{ i \in [d-1]: |x_i| \ge \kappa(G)/10 \big\}.
\]
By Lemma \ref{pos_prop_conc} and $\kappa(\pi_1(G)) \ge \kappa(G) \ge 10$ we have
\[
0.7007|\pi_1(G) \cap \Z^{d-1}| \le(1-2e^{-(d-1)/10})|\pi_1(G) \cap \Z^{d-1}| \le  \Big| \Big\{ x \in \pi_1(G) \cap \Z^{d-1}: |Z_x| \ge (d-1)/10\Big\} \Big|,
\]
let $E$ denote the set on the right-hand side. For $x \in E$, $k \in \N_0$, $0 \le k \le \kappa(G)/10$, $j \in Z_x$ we define $f_{k,j}(x) \in G$ by the formula
\[
f_{k,j}(x)=\Big(x_1,...,x_{j-1},x_j-k\sgn(x_j), x_{j+1},...,x_{d-1}, k\sgn(x_j) \Big),
\]
The fact that $f_{k,j}(x) \in G$ follows from convexity and  
\[
(x_1,...,x_{j-1},0,x_{j+1},...,x_{d-1},x_j), (x_1,...,x_{j-1},x_j,x_{j+1},...,x_{d-1},0) \in E(\pi_1(G)) \subseteq G.
\]
We will count size including multiplicity of the following multiset 
\[
A=\Big\{ f_{k,j}(x): x \in E, j \in Z_x, k \in \Z, 0 \le k \le \kappa(G)/10\Big\}.
\]
First note that for every $y  \in G \cap \Z^d$ we have
\[
\Big| \Big\{ (x,k,j): f_{k,j}(x)=y \Big\} \Big| \le d-1,
\]
this follows from the fact that data $y,j$ uniquely determines $x$ and $k$. Once in the above set $y$ is fixed, then for $j$ there are at most $d-1$ options. For $y \in A$, $\mult_A(y)$ will denote multiplicity of element $y$ in multiset $A$. Computing the size of $A$ we get 
\begin{align*}
    &(d-1)|G \cap \Z^d| \ge \sum_{y \in G \cap \Z^d} \mult_A(y) = |A|= \Big|\Big\{ (x,k,j): x \in E, j \in Z_x, k \in \Z, 0 \le k \le \kappa(G)/10 \Big\} \Big| \\
    &= \sum_{x \in E} |Z_x| (1+ \lfloor \kappa(G)/10 \rfloor) \ge\frac{d-1}{100} \kappa(G) |E| \ge \frac{d-1}{100}\cdot0.7007 \cdot \frac{143}{144}\widetilde{\kappa}(G) |\pi_1(G) \cap \Z^{d-1}|,
\end{align*}
above we have used $\kappa(G) \ge  \widetilde{\kappa}(G)-1$, comparing both sides finishes the proof.
\end{proof}
\begin{Lem} \label{big_coord}
    Let $d \in \N, G \in \mathcal{F}_d$ satisfy $d \ge 20$, $\widetilde{\kappa}(G) \ge 1/2$. Take $r \in \N$, $r \le d$ and distinct natural numbers $1 \le i_1,...,i_r \le d$. Then for all $j_1,...,j_r \in \N_0$ with $d\widetilde{\kappa}(G) \ge j_l \ge 200 \widetilde{\kappa}(G)$ we have
    \[
    \spr{x \in G \cap \Z^d} \Big( |x_{i_1}|=j_1, \ldots,|x_{i_r}|=j_r \Big) \le  10^{-r}\widetilde{\kappa}(G)^{-r}\exp\Big(- \frac{\sum_{l=1}^rj_l}{63\widetilde{\kappa}(G)} \Big) .
    \]
\end{Lem}
\begin{proof}
By invariance under permutation of coordinates, we can assume without loss of generality, that $i_r=d, i_{r-1}=d-1,\ldots, i_1=d-r+1$. We may also assume that $\sum_{l=1}^r j_l \le d \widetilde{\kappa}(G)$, since otherwise by definition of $\widetilde{\kappa}(G)$ the left hand-side of the lemma would be $0$.
For $l \in [r]$ define $k_l= \lfloor \frac{j_l}{100\widetilde{\kappa}(G)} \rfloor \ge 1$ . We will prove that for any $x=(x_1,...,x_d)  \in G$ with $x_d=j_r, x_{d-1}=j_{r-1},...,x_{d-r+1}=j_1$ and any $\overrightarrow{y}^{(l)}=(y_1^{(l)},...,y_{k_l+1}^{(l)}) \in j_l B^{1,(k_l+1)} \cap \Z^{k_l+1}$ we have
\begin{equation} \label{el in d+k}
(x_1,...,x_{d-r},\overrightarrow{y}^{(1)}, \ldots, \overrightarrow{y}^{(r)}) \in G^{(d+\sum_{l=1}^r k_l)} \cap \Z^{(d+\sum_{l=1}^r k_l)}.
\end{equation}
The above would show that there is an injection 
\[
\Big\{ x \in G \cap \Z^d: (\forall l \in [r]) x_{d-r+l}=j_l
\Big\} \times \Big(j_1 B^{1,(k_1+1)} \cap \Z^{k_1+1} \Big) \times \ldots \times \Big( j_r B^{1,(k_r+1)} \cap \Z^{k_r+1} \Big) \longrightarrow G^{(d+\sum_{l=1}^r k_l)} \cap \Z^{(d+\sum_{l=1}^r k_l)}
\]
It turns out that this suffices to prove Lemma \ref{big_coord}. Indeed, let $k=\sum_{l=1}^rk_l$ assuming \eqref{el in d+k} by Lemma \ref{dim_comp} we get
\begin{equation} \label{eq:dim_incr}
\begin{split}
&\Big| \Big\{ x \in G \cap \Z^d: (\forall l \in [r]) |x_{d-r+l}|=j_l
\Big\} \Big|=2^r \Big|\Big\{ x \in G \cap \Z^d: (\forall l \in [r]) x_{d-r+l}=j_l
\Big\} \Big| \\
&\le 2^r |G^{(d+k)}\cap \Z^{d+k}| \cdot \prod_{l=1}^r| j_l B^{1,(k_l+1)}|^{-1}  \le 2^r |G^{(d+k)}\cap \Z^{d+k}| \prod_{l=1}^r \binom{j_l+k_l+1}{k_l+1}^{-1} \\
&\le 2^r \prod_{l=1}^r\Big( \frac{k_l+1}{j_l} \Big)^{k_l+1} |G^{(d+k-1)}\cap \Z^{d+k-1}| \cdot 20 \widetilde{\kappa}(G^{(d+k)}) \\
&\le 2^r \prod_{l=1}^r\Big( \frac{k_l+1}{j_l} \Big)^{k_l+1} |G^{(d+k-1)}\cap \Z^{d+k-1}| \cdot 20 \widetilde{\kappa}(G) \le \ldots \\
&\le 2^r \prod_{l=1}^r\Big( \frac{k_l+1}{j_l} \Big)^{k_l+1}  20^k \widetilde{\kappa}(G)^k |G \cap \Z^d| = 10^{-r}\widetilde{\kappa}(G)^{-r} \prod_{l=1}^r\Big( \frac{20(k_l+1) \widetilde{\kappa}(G)}{j_l} \Big)^{k_l+1}  |G \cap \Z^d|.
\end{split}
\end{equation}
We were able to apply Lemma \ref{dim_comp}, since by Fact \ref{fact about proj and ext} we have
\begin{align*}
&\widetilde{\kappa}(G) \ge \widetilde{\kappa}(G^{(d+1)}) \ldots \ge\widetilde{\kappa}(G^{(d+k)}) \ge \frac{d+k-1}{d+k} \widetilde{\kappa}\big(G^{(d+k-1)}\big) \ge \ldots \\
& \ge \widetilde{\kappa}(G) \prod_{j=1}^k(1-\frac{1}{d+j}) 
\ge \widetilde{\kappa}(G) \Big(1- \frac{1}{d} \Big)^k \ge \widetilde{\kappa}(G) \Big(1- \frac{1}{d} \Big)^d \ge \frac{1}{3} \widetilde{\kappa}(G) \ge 1/6,
\end{align*}
above we have used the assumption $\sum_{l=1}^r j_l \le d \widetilde{\kappa}(G)$.
Now we investigate every individual term from the last expression in \eqref{eq:dim_incr}, fix $l \in [r]$
since $j_l \ge 200 \widetilde{\kappa}(G)$ we have
\[
k_l+1 \le \frac{j_l}{20e \widetilde{\kappa}(G)}.
\]
Using the decreasingness of the function $[0, a/e) \ni t \mapsto (t/a)^t$ we get
\begin{align*}
    \Big( \frac{20(k_l+1) \widetilde{\kappa}(G)}{j_l} \Big)^{k_l+1} \le \Big( \frac{1}{5} \Big)^{\frac{1}{100} \frac{j_l}{\widetilde{\kappa}(G)}}= \exp \Big( - \frac{1}{100} \log(5) \cdot \frac{j_l}{\widetilde{\kappa}(G)} \Big) \le \exp \Big( -  \frac{j_l}{63\widetilde{\kappa}(G)} \Big),
\end{align*}
plugging the above into \eqref{eq:dim_incr} we get the desired conclusion. 
\par Let us go back to the proof of \eqref{el in d+k}.
Let $\overrightarrow{z}^{(l)} \in \Z^{k_l+1}$ be given by $z_1=j_l$, $z_i=0$ for $i \neq 1$. Notice that by invariance under permutations, for any $x \in G$ with $x_{d-r+l}=j_l$ for all $l \in [r]$, we have
\[
(x_1,...,x_{d-r},\overrightarrow{z}^{(1)}, \ldots, \overrightarrow{z}^{(r)}) \in G^{(d+\sum_{l=1}^r k_l)} \cap \Z^{(d+\sum_{l=1}^r k_l)}.
\]
Since 
\[
j_lB^{1,(k_l+1)}= \conv \Big( \Big\{ \sigma \cdot \big(\epsilon \cdot\overrightarrow{z}^{(l)} \big):\sigma \in \Sym(k_l+1), \\ \epsilon \in \{-1,1\}^{k_l+1}\Big\}\Big),
\]
due to convexity we get the desired conclusion.
\end{proof}
Now we give two simple corollaries of Lemma \ref{big_coord}.
\begin{Cor} \label{big_coord tail bound}
   Let $d \in \N, G \in \mathcal{F}_d$ satisfy $d \ge 20$, $\widetilde{\kappa}(G) \ge 1/2$. Take $r \in \N$, $r \le d$ and distinct natural numbers $1 \le i_1,...,i_r \le d$. Then for all $\alpha_1,...,\alpha_r \in \N_0$ with $d\widetilde{\kappa}(G) \ge \alpha_l \ge 200 \widetilde{\kappa}(G)$ we have
    \[
    \spr{x \in G \cap \Z^d} \Big( |x_{i_1}|\ge \alpha_1, \ldots,|x_{i_r}| \ge \alpha_r \Big) \le  7^{r}\exp\Big(- \frac{\sum_{l=1}^r\alpha_l}{63\widetilde{\kappa}(G)} \Big) .
    \]  
\end{Cor}
\begin{Cor} \label{big_coord tail bound no ass on alpha}
   Let $d \in \N, G \in \mathcal{F}_d$ satisfy $d \ge 20$, $\widetilde{\kappa}(G) \ge 1/2$. Take $r \in \N$, $r \le d$ and distinct natural numbers $1 \le i_1,...,i_r \le d$. Then for all $\alpha_1,...,\alpha_r \in \N_0$  we have
    \[
    \spr{x \in G \cap \Z^d} \Big( |x_{i_1}|\ge \alpha_1, \ldots,|x_{i_r}| \ge \alpha_r \Big) \le 24^r\exp\Big(- \frac{\sum_{l=1}^r\alpha_l}{63\widetilde{\kappa}(G)} \Big) .
    \]  
\end{Cor}
\begin{proof}[Proof of Corollary \ref{big_coord tail bound}]
    By Lemma \ref{big_coord} we have 
    \begin{align*}
        &\spr{x \in G \cap \Z^d} \Big( x \in G \cap \Z^d: |x_{i_1}|\ge \alpha_1, \ldots,|x_{i_r}| \ge \alpha_r \Big)= \sum_{\substack{j_1,\ldots, j_r \\
        \alpha_l \le j_l \le d \widetilde{\kappa}(G)}} \spr{x \in G \cap \Z^d} \Big( |x_{i_1}|=j_1, \ldots,|x_{i_r}|=j_r \Big) \\
        &\le 10^{-r} \widetilde{\kappa}(G)^{-r}\prod_{l=1}^r \sum_{j_r= \alpha_r}^\infty e^{- \frac{j_r}{63 \widetilde{\kappa}(G)}} \le  10^{-r} \widetilde{\kappa}(G)^{-r}\prod_{l=1}^r \frac{\exp\Big(- \alpha_r/(63 \widetilde{\kappa}(G))\Big)}{1-\exp\Big(- 1/(63 \widetilde{\kappa}(G))\Big)}\\
        &\le \Big(\frac{1+63 \widetilde{\kappa}(G)}{10 \widetilde{\kappa}(G)} \Big)^r \exp\Big(- \frac{\sum_{l=1}^r\alpha_l}{63\widetilde{\kappa}(G)} \Big)  \le 7^r \exp\Big(- \frac{\sum_{l=1}^r\alpha_l}{63\widetilde{\kappa}(G)} \Big) .
    \end{align*}
\end{proof}
From Corollary \ref{big_coord tail bound} we can easily deduce Corollary \ref{big_coord tail bound no ass on alpha}.
\begin{proof}[Proof of Corollary \ref{big_coord tail bound no ass on alpha}]
We can assume that for all $l \in [r]$ we have $\alpha_l \le d \widetilde{\kappa}(G)$, since otherwise the left hand-side of the statement would equal $0$. Let 
\[
k= \Big| \Big\{ l \in [r]: \alpha_l \ge 200 \widetilde{\kappa}(G) \Big\}\Big|.
\]
Without loss of generality assume that $\alpha_1,..., \alpha_k \ge 200 \widetilde{\kappa}(G)$. Then by Corollary \ref{big_coord tail bound} we have 
\begin{align*}
    &\spr{x \in G \cap \Z^d} \Big( |x_{i_1}|\ge \alpha_1, \ldots,|x_{i_r}| \ge \alpha_r \Big)| \le \spr{x \in G \cap \Z^d} \Big( |x_{i_1}|\ge \alpha_1, \ldots,|x_{i_k}| \ge \alpha_k \Big) \\
    &\le 7^k \exp\Big(- \frac{\sum_{l=1}^k\alpha_l}{63\widetilde{\kappa}(G)} \Big)  \le 7^k e^{\frac{200}{63}(r-k)} \exp\Big(- \frac{\sum_{l=1}^r\alpha_l}{63\widetilde{\kappa}(G)} \Big)  \le 24^r \exp\Big(- \frac{\sum_{l=1}^r\alpha_l}{63\widetilde{\kappa}(G)} \Big).
\end{align*}
\end{proof}
Having proved tail bound for coordinates of lattice points in $G$ in order to prove the first part of Theorem \ref{Bobkov-Nazarov main thm} we can proceed the exact same way as in \cite[Sections 4,5]{bob-naz}.
For $x \in \R^d$ we write absolute values of its coordinates in decreasing order
\[
X_1 \ge X_2 \ge \ldots \ge X_d.
\]
\begin{Prop}[Version of {\cite[Proposition 4.1]{bob-naz}}] \label{ordered tail bound}
    For any $d \ge 20$, $G \in \mathcal{F}_d$ with $\widetilde{\kappa}(G) \ge 1/2$, any $\alpha \ge 0$ and $k \le d$ we have
    \[
    \spr{x \in G \cap \Z^d} \Big( X_k \ge \alpha \Big)\le 24^k \binom{d}{k} e^{-k \alpha/(63\widetilde{\kappa}(G))}.
    \]
\end{Prop}
\begin{proof}
    Since 
    \[
    \Big\{ x \in G \cap \Z^d: X_k \ge \alpha \Big\}= \bigcup_{d \ge i_1 > \ldots>i_k \ge 1} \Big\{x \in G \cap \Z^d:|x_{i_1}| \ge \lceil \alpha \rceil,\ldots |x_{i_k}| \ge \lceil \alpha \rceil \Big\},
    \]
by Corollary \ref{big_coord tail bound no ass on alpha} we get 
\begin{align*}
   \spr{x \in G \cap \Z^d} \Big(x \in G \cap \Z^d: X_k \ge \alpha \Big) &\le \sum_{d \ge i_1 > \ldots>i_k \ge 1} \spr{x \in G \cap \Z^d} \Big(|x_{i_1}| \ge \lceil \alpha \rceil,\ldots |x_{i_k}| \ge \lceil \alpha \rceil \Big) \\
    &\le 24^k \binom{d}{k} e^{-k\alpha/(63 \widetilde{\kappa}(G))}.
\end{align*}
\end{proof}
As a consequence of preceding estimates we can establish the first part of Theorem \ref{Bobkov-Nazarov main thm}.
\begin{proof}[Proof of the first part of Theorem \ref{Bobkov-Nazarov main thm}]
First consider the case $d \ge 20$.
Take $d \ge 20$ and $G \in \mathcal{F}_d$ with $\widetilde{\kappa}(G) \ge 1/2$. Pick $\alpha_1,\alpha_2,\ldots, \alpha_d >0$, which will be determined later. By Proposition \ref{ordered tail bound} we have 
\begin{align*}
    &\spr{x \in G \cap \Z^d} \Big( |x|_2^2 \ge \sum_{k=1}^d \alpha_k^2 \Big)=\spr{x \in G \cap \Z^d} \Big(\sum_{k=1}^d X_k^2 \ge \sum_{k=1}^d \alpha_k^2 \Big) \\
    &\le \sum_{k=1}^d \spr{x \in G \cap \Z^d} \Big(  X_k \ge  \alpha_k \Big) \le \sum_{k=1}^d 24^k \binom{d}{k} e^{-k \alpha_k/(63\widetilde{\kappa}(G))} \\
    &\le \sum_{k=1}^d  \exp\Big(-k \Big(\frac{\alpha_k}{63\widetilde{\kappa}(G)}-\log\big( \frac{ed}{k}\big)- \log(24) \Big)\Big).
\end{align*}
Now take $\alpha_k=63\widetilde{\kappa}(G) \cdot \Big(\log(48)+ \log( \frac{ed}{k}\big)+t \frac{\sqrt{d}}{k}\Big)$. Then for $t \ge 10^{9}$ we have
\begin{align*}
    &\sum_{k=1}^d \alpha_k^2 \le 63^2\widetilde{\kappa}(G)^2 \bigg(3 t^2 d\sum_{k=1}^d \frac{1}{k^2}+3 \sum_{k=1}^d \log^2\big( \frac{ed}{k} \big)+ 3d\Big(  \log(48)\Big)^2 \bigg) \le 4 \cdot 10^4 \cdot t^2d \widetilde{\kappa}(G)^2
\end{align*}
and
\begin{align*}
   \spr{x \in G \cap \Z^d} \Big( |x|_2^2 \ge \sum_{k=1}^d \alpha_k^2 \Big) &\le \sum_{k=1}^d  \exp\Big(-k \Big(\frac{\alpha_k}{63\widetilde{\kappa}(G)}-\log\big( \frac{ed}{k}\big)- \log(24) \Big)\Big) \\
    &\le \sum_{k=1}^d 2^{-k} e^{-t \sqrt{d}} \le e^{-t \sqrt{d}}.
\end{align*}
 In conclusion for $t \ge 10^9$ we get
\begin{align*}
    \spr{x \in G \cap \Z^d} \Big(\frac{|x|_2}{\sqrt{d}} \ge 200 t \widetilde{\kappa}(G) \Big) \le  e^{-t \sqrt{d}} .
\end{align*}
Hence for any $t \ge 2\cdot10^{11}$ one obtains
\begin{align*}
    \spr{x \in G \cap \Z^d} \Big( \frac{|x|_2}{\sqrt{d}} \ge t \widetilde{\kappa}(G) \Big) \le e^{-t \sqrt{d}/200}.
\end{align*}
Now let us go back to the case $1 \le d <20$. 
In this situation for any $x \in G \cap \Z^d$ with $\widetilde{\kappa}(G) \ge 1/2$ we have
\[
\frac{|x|_2}{\sqrt{d}} \le \frac{|x|_1}{\sqrt{d}} \le \sqrt{d} \widetilde{\kappa}(G)<t \widetilde{\kappa}(G),
\]
for $t \ge 2 \cdot 10^{11}$. Hence the inequality
\begin{align*}
    \spr{x \in G \cap \Z^d} \Big( \frac{|x|_2}{\sqrt{d}} \ge t \widetilde{\kappa}(G) \Big) \le e^{-t \sqrt{d}/200}
\end{align*}
trivially holds in this case.
\end{proof}
We end this section with partial result toward Theorem \ref{iso const}.
\begin{Prop} \label{concen medium}
     For all $p \in [1,\infty)$ and $d \in \N, G \in \mathcal{F}_d$, such that $\widetilde{\kappa}(G) \ge 1/2$ we have
    \[
    \se{x \in G \cap \Z^d} |x_1|^p= \se{x \in G \cap \Z^d} \frac{|x|_p^p}{d} \lesssim_{p} \widetilde{\kappa}(G)^p.
    \]
    Moreover if $d \ge 20$ and $\widetilde{\kappa}(G) \ge 10$, then
    \[
   \se{x \in G \cap \Z^d} |x_d|^p= \se{x \in G \cap \Z^d} \frac{|x|_p^p}{d} \gtrsim_{p} \widetilde{\kappa}(G)^p.
    \]
\end{Prop}
\begin{proof}
    In the first part we may assume that $d \ge 20$, since otherwise for $x \in G \cap \Z^d$ we trivially have
    \[
    |x|_p^p \le |x|_1^p \le d^p \widetilde{\kappa}(G)^p <20^p\widetilde{\kappa}(G)^p.
    \]
    By Lemma \ref{big_coord} we have 
    \begin{align*}
        &\se{x \in G \cap \Z^d} |x_1|^p \lesssim \widetilde{\kappa}(G)^p+  \frac{1}{|G \cap \Z^d|}\sum_{j \ge 200 \widetilde{\kappa}(G)} j^p \Big|\Big\{ x \in G \cap \Z^d: |x_1|=j \Big\}\Big| \\
        &\lesssim\widetilde{\kappa}(G)^p+\sum_{j \ge 200 \widetilde{\kappa}(G)} j^p \widetilde{\kappa}(G)^{-1} e^{-\frac{j}{63 \widetilde{\kappa}(G)}} \lesssim_p \widetilde{\kappa}(G)^p
    \end{align*}
    Second part of the proposition directly follows from Lemma \ref{pos_prop_conc}.
\end{proof}

\section{Lattice points in small-scale regime}

The next lemma is an adapted version of \cite[Lemma 3.2]{balls}, proof goes along the same lines.
\begin{Lem} \label{BMSW a lot of 1}
For all $d \in \mathbb{N}, G \in \mathcal{F}_d$, if $\widetilde{\kappa}(G) \leq  400^{-1}$ and $d \widetilde{\kappa}(G) \geq k \geq 400 d \widetilde{\kappa}(G)^2$, then we have
\begin{equation*} 
\spr{x \in G \cap \Z^d} \Big(|\lbrace i \in [d]:x_i= \pm 1 \rbrace | \leq d \widetilde{\kappa}(G)-k \Big) \leq 2^{2-k} .  
\end{equation*} 
\end{Lem}
\begin{proof}
 Let $n=d \widetilde{\kappa}(G)$. We have that
\begin{equation} \label{eq:3.1}
 \Big| \Big\lbrace x \in G \cap \Z^d: |\lbrace i \in [d] :x_i= \pm 1 \rbrace | \leq n-k \Big\rbrace\Big|= \sum_{m=k}^n |E_m|, 
 \end{equation}
where 
\[ E_m= \Big\lbrace x \in G \cap \mathbb{Z}^d : |\lbrace i \in [d] :x_i= \pm 1 \rbrace | = n-m \Big\rbrace. \]
To prove the lemma it is sufficient to show that $|E_m| \leq 2^{1-m} |G \cap \mathbb{Z}^d|$ holds for every $m \in \lbrace k,k+1,...,n \rbrace$. Notice that if $x \in E_m$ then
\[ | \lbrace i \in [d]: |x_i| \geq 2 \rbrace | \leq \Big\lfloor \frac{m}{2} \Big\rfloor. \]
In order to bound $|E_m|$ we will consider two cases separately.
\par \textbf{Case 1) $m< 2$.} \\
In this situation every $x \in E_m$ consists only of coordinates $-1,0,1$, hence we have
\begin{equation*}
    |E_m| = 2^{n-m} \binom{d}{n-m}.
\end{equation*}
Moreover 
\begin{equation*} 
    \begin{split}
        &\binom{d}{n-m}  \binom{d}{n}^{-1} \hspace{-0.3cm}= \frac{n!(d-n)!}{(n-m)!(d-n+m)!} \le \Big( \frac{n}{d-n} \Big)^m\hspace{-0.3cm} = \Big( \frac{\widetilde{\kappa}(G)}{1-\widetilde{\kappa}(G)}  \Big)^m \hspace{-0.3cm}\le \Big( 2\widetilde{\kappa}(G)\Big)^m,
    \end{split}
\end{equation*} 
above we have used the fact that $\widetilde{\kappa}(G)\le \frac{1}{2}$.
By \eqref{disc_incl} we also have 
\begin{equation} \label{eq:3.2}
 2^n \binom{d}{n} \leq |G \cap \mathbb{Z}^d |.
\end{equation}
Combining these three facts we obtain
\begin{align*}
    &|E_m| \le 2^{n-m} \binom{d}{n-m}= 2^{-m} \cdot 2^{n} \binom{d}{n} \binom{d}{n-m}  \binom{d}{n}^{-1} \\
    &\le 2^{-m} |G \cap \Z^d| \Big( 2\widetilde{\kappa}(G)\Big)^m \le 2^{-m} |G \cap \Z^d|.
\end{align*}
\par \textbf{Case 2) $m \ge 2$.} \\
Notice that we have the following upper bound for $|E_m|$ . 
\begin{equation} \label{eq:3.3}
|E_m| \leq 2^{n-m} \binom{d}{n-m} \binom{d-n+m}{\lfloor \frac{m}{2} \rfloor} |mB^{(1,\lfloor m/2 \rfloor)} \cap \mathbb{Z}^{\lfloor m/2\rfloor} |. 
\end{equation}
Indeed, in $\binom{d}{n-m}$ options we choose coordinates on which $x \in E_m$ will have values $\pm 1$, then we choose each sign in 2 ways, this explains the factor $2^{n-m} \binom{d}{n-m}$. Next we choose $\lfloor \frac{m}{2} \rfloor$ coordinates  in which the set $\lbrace i \in [d]: |x_i| \geq 2 \rbrace$ will be contained, for that we have $\binom{d-n+m}{\lfloor \frac{m}{2} \rfloor}$ options. Lastly we bound the number of ways of putting numbers on these coordinates such that the condition $| x |_{1} \leq n$ holds, this is bounded by $|mB^{(1,\lfloor m/2 \rfloor)} \cap \mathbb{Z}^{\lfloor m/2 \rfloor}|$. \\ \\
Using Lemma \ref{Cor 2.2 from my lq paper} we obtain 
\begin{equation} \label{eq:3.4}
    \begin{split}
        |mB^{(1,\lfloor m/2 \rfloor)} \cap \mathbb{Z}^{\lfloor m/2 \rfloor}| &\leq 2 \cdot \Big( \frac{m}{\lfloor m/2\rfloor}+\frac{1}{2} \Big)^{\lfloor m/2 \rfloor} 8^{\lfloor m/2 \rfloor} \leq 2 \cdot \Big( \Big( \frac{m}{m/4} \Big)+ \frac{1}{2} \Big)^{m/2} 8^{\lfloor m/2 \rfloor}  \\
 &= 2 \cdot\Big( 4+ \frac{1}{2} \Big)^{m/2} 8^{\lfloor m/2 \rfloor} \leq  2 \cdot\Big( \frac{9}{2} \Big)^{m/2} 8^{m/2}= 2 \cdot 6^{m},
    \end{split}
\end{equation}
above we have used the fact that $m  \geq 2$.
We also have that
\begin{equation} \label{eq:3.5}
    \begin{split}
        &\binom{d}{n-m} \binom{d-n+m}{\lfloor \frac{m}{2} \rfloor} \binom{d}{n}^{-1} = \frac{n! \cdot (d-n)!}{(n-m)! \lfloor m/2 \rfloor! \cdot (d-n+m-\lfloor m/2 \rfloor)!}  \\
        &\leq \frac{n^{\lfloor m/2 \rfloor}}{\lfloor m/2 \rfloor!} \Big( \frac{n}{d-n} \Big)^{m-\lfloor m/2 \rfloor} \stackrel{(*)}{\leq} \Big( \frac{en}{\lfloor m/2 \rfloor} \Big)^{\lfloor m/2 \rfloor} \Big( \frac{n}{d-n} \Big)^{m-\lfloor m/2 \rfloor}  \\ 
        &\stackrel{(**)}{\leq} \Big( \frac{2en}{m} \Big)^{ m/2 } \Big( \frac{n}{d-n} \Big)^{m-\lfloor m/2 \rfloor}= \Big( \frac{2en}{m} \Big)^{ m/2 } \Big( \frac{\widetilde{\kappa}(G)}{1-\widetilde{\kappa}(G)} \Big)^{m-\lfloor m/2 \rfloor}  \\
    &\leq \Big( \frac{2en}{m} \Big)^{ m/2 } \Big( 2\widetilde{\kappa}(G) \Big)^{m-\lfloor m/2 \rfloor} \leq \Big( \frac{2 en}{m} \Big)^{ m/2 } \Big( 2\widetilde{\kappa}(G) \Big)^{m/2}.
    \end{split}
\end{equation} 
$(*)$ holds, since $\lfloor m/2 \rfloor! \geq (\lfloor m/2 \rfloor/e)^{\lfloor m/2 \rfloor}$. In $(**)$ we used the fact that for any $a>0$ the function $(0,a/e] \ni t \mapsto (a/t)^t$ is increasing. Two last inequalities hold, since $m-\lfloor m/2 \rfloor \geq m/2$ and $\widetilde{\kappa}(G) \le \frac{1}{2}$. \\
Using inequalities \eqref{eq:3.2}, \eqref{eq:3.3}, \eqref{eq:3.4}, \eqref{eq:3.5} and our assumptions on $\widetilde{\kappa}(G)$ we finally obtain
\begin{equation*} 
\begin{split}    
|E_m| &\leq 2 \cdot2^n \binom{d}{n} 2^{-m}  \Big( \frac{2en}{m} \Big)^{ m/2 } \Big( 2\widetilde{\kappa}(G) \Big)^{m/2} 6^m \\
 &\leq  2^{1-m}  \cdot |G \cap \mathbb{Z}^d| \Big(\frac{n\widetilde{\kappa}(G)}{m}\Big)^{m/2} e^{m/2} \cdot 2^{m} \cdot  6^m\\
 &\leq  2^{1-m} |G \cap \mathbb{Z}^d| \Big(\frac{n\widetilde{\kappa}(G)}{m}\Big)^{m/2} 20^m \leq  2^{1-m} |G \cap \mathbb{Z}^d|.
 \end{split}
 \end{equation*}
Last inequality holds by the assumption $m \geq k \geq 400n \widetilde{\kappa}(G) $. Due to \eqref{eq:3.1} this concludes the proof of Lemma \ref{BMSW a lot of 1}
\end{proof}
Note that if $0<\widetilde{\kappa}(G) \le 10^{-3}$, then we can apply Lemma \ref{BMSW a lot of 1} with $k=\lfloor \frac{d \widetilde{\kappa}(G)}{2} \rfloor$.

Now we will establish analogue of Lemma \ref{big_coord} and Corollaries \ref{big_coord tail bound}, \ref{big_coord tail bound no ass on alpha} in small-scale regime.
\begin{Lem} \label{big_coord small scales}
    Let $d \in \N, G \in \mathcal{F}_d$ satisfy  $\frac{1}{d} \le \widetilde{\kappa}(G) \le 1/2 $. Take $r \in \N$, $r \le d$ and distinct natural numbers $1 \le i_1,...,i_r \le d$. Then for all $j_1,...,j_r \in \N$ we have
    \begin{equation} \label{eq:big cord small scale eq}
        \spr{x \in G \cap \Z^d} \Big( |x_{i_1}|=j_1, \ldots,|x_{i_r}|=j_r \Big) \le  2^r \widetilde{\kappa}(G)^{\sum_{l=1}^r j_l} ,
    \end{equation}
    and
    \begin{equation} \label{eq:big cord small scale geq}
       \spr{x \in G \cap \Z^d} \Big( |x_{i_1}| \ge j_1, \ldots,|x_{i_r}| \ge j_r \Big) \le  4^r \widetilde{\kappa}(G)^{\sum_{l=1}^r j_l} .
    \end{equation}
\end{Lem}
Up to $C^r$ we expect the above inequalities to be optimal for $G=nB^{1,(d)}$, when $n \le d/1000$.

\begin{proof}
    Note that it suffices to prove \eqref{eq:big cord small scale eq}. Without loss of generality we can assume that $i_l=l$ for all $l \in [r]$. We define $j=\sum_{l=1}^rj_l$, we can assume that $j \le d \widetilde{\kappa}(G)$, since otherwise the conclusion trivially holds. Let 
    \[
    A= \Big\{ x \in G \cap \Z^d: x_{1}=j_1, \ldots,x_{r}=j_r \Big\}.
    \]
    For any $x \in A$, $S \subseteq \big([d] \setminus \supp(x)\big) \cup [r]$ with $|S|=j$ and $\epsilon \in \{-1,1\}^S$ we define $f(x,S,\epsilon) \in G \cap \Z^d$ by the formula
    \[
    f(x,S,\epsilon)=x-\sum_{i=1}^r j_i e_i+\sum_{i \in S}\epsilon_i e_i,
    \]
    where $e_i \in \R^d$ is a vector with $1$ on $i$-th coordinate and $0$ on other coordinates.
    It is not difficult to see that $f(x,S,\epsilon)$ is a convex combination of points obtained by permutations and changing signs of coordinates of $x$, hence it lies in $G \cap \Z^d$.
    Pick $y \in G \cap \Z^d$, note that if for
     $x \in A$, $S \subseteq \big([d] \setminus \supp(x)\big) \cup [r]$ with $|S|=j$ and $\epsilon \in \{-1,1\}^S$ we have
     \[
     y=f(x,S,\epsilon),
     \]
     then the data $y,S$ uniquely determines the triple $(x,S,\epsilon)$. Indeed if $y=(y_1,...,y_d)$ and $S \subseteq \supp(y)$, we have
     \[
     x= \sum_{i \in \supp(y) \setminus S} y_ie_i + \sum_{l=1}^r j_l e_l, \quad \epsilon_i=y_i \text{ for } i \in S.
     \]
     In particular for fixed $y$ we get that 
     \[
     \Big|\Big\{(x,S,\epsilon): f(x,S,\epsilon)=y \Big\}\Big| \le \binom{|\supp(y)|}{j}\le \binom{d\widetilde{\kappa}(G)}{j}.
     \]
     Also note that for $x \in A$ we have
     \[
     |\supp(x)| \le r+d\widetilde{\kappa}(G)-j.
     \]
    This way we obtain
    \begin{align*}
        &\binom{d\widetilde{\kappa}(G)}{j} |G \cap \Z^d| \ge \Big|\Big\{(x,S,\epsilon): x \in A,S \subseteq [d] \setminus \supp(x) \cup [r],|S|=j,\epsilon \in \{-1,1\}^S\Big\}\Big|\\
        &= \sum_{x \in A} 2^j \binom{d+r-|\supp(x)|}{j} \ge 2^j \binom{d-d \widetilde{\kappa}(G)+j}{j} |A|.
    \end{align*}
    Thus we get
    \begin{align*}
       &|A| \le 2^{-j}  \binom{d\widetilde{\kappa}(G)}{j} \binom{d-d \widetilde{\kappa}(G)+j}{j}^{-1}|G \cap \Z^d| \le 2^{-j} \frac{(d \widetilde{\kappa}(G))^{j}}{j!} \cdot \frac{j!}{(d-d \widetilde{\kappa}(G))^{j}}|G \cap \Z^d| \\
       &\le 2^{-j} \Big(\frac{\widetilde{\kappa}(G)}{1-\widetilde{\kappa}(G)}\Big)^j |G \cap \Z^d| \le \widetilde{\kappa}(G)^j|G \cap \Z^d|.
    \end{align*}
    This way we obtain
    \[
     \Big| \Big\{ x \in G \cap \Z^d: |x_{1}|=j_1, \ldots,|x_{r}|=j_r \Big\}\Big| =2^r \Big| \Big\{ x \in G \cap \Z^d: x_{1}=j_1, \ldots,x_{r}=j_r \Big\}\Big| \le 2^r \widetilde{\kappa}(G)^j |G \cap \Z^d|.
    \]
\end{proof}
Let us give a simple Corollary of the previous lemma.
\begin{Cor} \label{small scales no big coord}
    For $d,K \in \N$, $G \in \mathcal{F}_d$, if $\widetilde{\kappa}(G) \le (8d)^{- \frac{1}{K+1}}$, then 
    \[
    \spr{x \in G \cap \Z^d} \Big(x \in \{-K, \ldots, K \}^d \Big) \ge \frac{1}{2} .
    \]
\end{Cor}
\begin{proof}
    By Lemma \ref{big_coord small scales} we have
    \begin{align*}
        &\spr{x \in G \cap \Z^d} \Big( (\exists i \in [d]) \ |x_i| \ge K+1 \Big) \le d\spr{x \in G \cap \Z^d} \Big( |x_1| \ge K+1 \Big)\le 4d \widetilde{\kappa}(G)^{K+1}  \le \frac{1}{2} .
    \end{align*}
\end{proof}
Having tail bound for coordinates of lattice points in $G$ in order to prove the second part of Theorem \ref{Bobkov-Nazarov main thm} we can proceed the exact same way as in \cite[Sections 4,5]{bob-naz}.
For $x \in \R^d$ we write absolute values of its coordinates in decreasing order
\[
X_1 \ge X_2 \ge \ldots \ge X_d.
\]
\begin{Prop}[Version of {\cite[Proposition 4.1]{bob-naz}} in small-scale regime] \label{small scale ordered tail bound}
    For any $d \in \N$, $G \in \mathcal{F}_d$ with $\widetilde{\kappa}(G) \le 1/2$, any $\alpha \ge 0$ and $k \le d$ we have
    \[
   \spr{x \in G \cap \Z^d} \Big( X_k \ge \alpha \Big)\le 4^k \binom{d}{k} \widetilde{\kappa}(G)^{k\alpha} .
    \]
\end{Prop}
\begin{proof}
    Since 
    \[
    \Big\{ x \in G \cap \Z^d: X_k \ge \alpha \Big\}= \bigcup_{d \ge i_1 > \ldots>i_k \ge 1} \Big\{x \in G \cap \Z^d:|x_{i_1}| \ge \lceil \alpha \rceil,\ldots |x_{i_k}| \ge \lceil \alpha \rceil \Big\},
    \]
by Lemma \ref{big_coord small scales} we get 
\begin{align*}
    \spr{x \in G \cap \Z^d} \Big( X_k \ge \alpha \Big) &\le \sum_{d \ge i_1 > \ldots>i_k \ge 1} \spr{x \in G \cap \Z^d} \Big(|x_{i_1}| \ge \lceil \alpha \rceil,\ldots |x_{i_k}| \ge \lceil \alpha \rceil \Big) \le 4^k \binom{d}{k} \widetilde{\kappa}(G)^{k\alpha}.
\end{align*}
\end{proof}
Equipped with these estimates we can establish the second part of Theorem \ref{Bobkov-Nazarov main thm} 
\begin{proof}[Proof of the second part of Theorem \ref{Bobkov-Nazarov main thm}] Take $d \in \N$, $G \in \mathcal{F}_d$ with $\widetilde{\kappa}(G) \le 1/2$. Pick $\alpha_1,\alpha_2,\ldots, \alpha_{d\widetilde{\kappa}(G)} \in \R_{>0}$, which will be determined later. Note that $X_l=0$ for all $l>d\widetilde{\kappa}(G)$. By Proposition \ref{small scale ordered tail bound} we have 
\begin{align*}
    &\spr{ x \in G \cap \Z^d}\Big(|x|_2^2 \ge \sum_{k=1}^{d\widetilde{\kappa}(G)} \alpha_k^2 \Big)=\spr{ x \in G \cap \Z^d}\Big(\sum_{k=1}^{d\widetilde{\kappa}(G)} X_k^2 \ge \sum_{k=1}^{d\widetilde{\kappa}(G)} \alpha_k^2 \Big) \\
    &\le \sum_{k=1}^{d\widetilde{\kappa}(G)} \spr{ x \in G \cap \Z^d}\Big(  X_k \ge  \alpha_k \Big) \le \sum_{k=1}^{d\widetilde{\kappa}(G)} 4^k \binom{d}{k} \widetilde{\kappa}(G)^{k\alpha_k} \\
    &\le \sum_{k=1}^{d\widetilde{\kappa}(G)} \exp \Big( k \Big( \alpha_k \log (\widetilde{\kappa}(G))+  \log \big( \frac{4ed}{k} \big) \Big) \Big).
\end{align*}
Now take 
\[
\alpha_k= \frac{t\sqrt{d\widetilde{\kappa}(G)}}{k}+\frac{\log \big( \frac{8ed}{k} \big)}{|\log(\widetilde{\kappa}(G))|}.
\]
Then for $t \ge 10^5$ we have
\begin{align*}
    &\sum_{k=1}^{d\widetilde{\kappa}(G)} \alpha_k^2 \le 2 t^2d\widetilde{\kappa}(G)\sum_{k=1}^d \frac{1}{k^2}+\frac{2}{|\log(\widetilde{\kappa}(G))|^2} \sum_{k=1}^{d \widetilde{\kappa}(G)} \log^2\Big(\frac{8ed}{k}\Big) \\
    &\le \frac{\pi^2}{3} t^2 d \widetilde{\kappa}(G)+ \frac{4}{|\log(\widetilde{\kappa}(G))|^2} \sum_{k=1}^{d \widetilde{\kappa}(G)}\Big(\log^2\Big( \frac{1}{\widetilde{\kappa}(G)}\Big)+ \log^2\Big( \frac{8e d \widetilde{\kappa}(G)}{k}\Big) \Big) \\
    &\le 4 t^2 d\widetilde{\kappa}(G),
\end{align*}
and
\begin{align*}
    \spr{x \in G \cap \Z^d} \Big(|x|_2^2 \ge \sum_{k=1}^{d\widetilde{\kappa}(G)} \alpha_k^2 \Big) &\le \sum_{k=1}^{d\widetilde{\kappa}(G)} \exp \Big( k \Big( \alpha_k \log (\widetilde{\kappa}(G))+  \log \big( \frac{4ed}{k} \big) \Big) \Big) \\
    & = \sum_{k=1}^{d \widetilde{\kappa}(G)} 2^{-k} \widetilde{\kappa}(G)^{t \sqrt{d\widetilde{\kappa}(G)}}  \le\widetilde{\kappa}(G)^{t \sqrt{d\widetilde{\kappa}(G)}}.
\end{align*}
 In conclusion for $t \ge 10^5$ we get
\begin{align*}
   \spr{x \in G \cap \Z^d} \Big( \frac{|x|_2}{\sqrt{d}} \ge 2t \sqrt{\widetilde{\kappa}(G) }\Big) \le  \widetilde{\kappa}(G)^{t \sqrt{d \widetilde{\kappa}(G)}}.
\end{align*}
Hence for any $t \ge 2 \cdot 10^5$ we obtain
\begin{align*}
    \spr{x \in G \cap \Z^d} \Big( \frac{|x|_2}{\sqrt{d}} \ge t  \sqrt{\widetilde{\kappa}(G) } \Big) \le  \widetilde{\kappa}(G)^{ t\sqrt{d\widetilde{\kappa}(G)}/2}.
\end{align*}
\end{proof}
\begin{Rem}[Optimality of the second part of Theorem \ref{Bobkov-Nazarov main thm}] \label{optimality of small scale Bobkov-Nazarov} We will establish optimality up to value of the constant $c>0$ of the second part of Theorem \ref{Bobkov-Nazarov main thm} in the range $\widetilde{\kappa}(G) \le d^{-1/2}$.  For any $d \in \N$ large and $k \in \N$ with $16 \le k \le \sqrt{d}$. Consider $G=kB^{1,(d)}$ and any $t \in \N$ with $2 \le t \le  \frac{1}{2} k^{1/2}$, let $t= \varepsilon k^{1/2} \ge 2$. Consider points $x \in G \cap \Z^d$, which have one coordinate equal to $\lceil \varepsilon k \rceil$ and the rest equal to $\pm 1$ or $0$, there are at least 
\[
d \cdot \binom{d-1}{k-\lceil \varepsilon k \rceil} 2^{k-\lceil \varepsilon k \rceil}
\]
such points in $G \cap \Z^d$. By \cite[Corollary 1.8]{KNW} and \eqref{disc_incl} we have $|G \cap \Z^d| \approx 2^k \binom{d}{k}$. This way we get 
\begin{align*}
    &\frac{\Big|\Big\{ x \in G \cap \Z^d: \frac{|x|_2}{\sqrt{d}} \ge t \sqrt{\widetilde{\kappa}(G) } \Big\} \Big|}{|G \cap \Z^d|} = \frac{\Big|\Big\{ x \in G \cap \Z^d: |x|_2 \ge \varepsilon k \Big\} \Big|}{|G \cap \Z^d|} \\
    & \gtrsim \frac{d \cdot \binom{d-1}{k-\lceil \varepsilon k \rceil} 2^{k-\lceil \varepsilon k \rceil}}{2^k \binom{d}{k}}= 2^{- \lceil \varepsilon k \rceil} \frac{ k!(d-k)!}{(k-\lceil \varepsilon k \rceil)! (d-1+ \lceil \varepsilon k \rceil-k)!} \\
    &\gtrsim 2^{-\varepsilon k}\Big( \frac{k-\lceil \varepsilon k \rceil}{d} \Big)^{\lceil \varepsilon k \rceil} d \ge \Big( \frac{k}{d} \Big)^{\varepsilon k} 2^{-\varepsilon k } \cdot \Big(1- \frac{\lceil \varepsilon k \rceil}{k} \Big)^{\lceil \varepsilon k \rceil} \\
    &\gtrsim  \Big(\frac{k}{d} \Big)^{\varepsilon k} e^{-10 \varepsilon^2 k} \ge \Big( \frac{k}{d} \Big)^{ 2\varepsilon k}= \widetilde{\kappa}(G)^{2t \sqrt{d \widetilde{\kappa}(G)}},
\end{align*}
    last inequality holds for sufficiently large $d$, since then $k/d \le d^{-1/2} \le e^{-10}$. This shows that in this range inequality from Theorem \ref{Bobkov-Nazarov main thm} is essentially the best possible. 
\end{Rem}
The following analogue of \cite[Theorem 3.1]{NW} extended to $G \in \mathcal{F}_d$ will be used later to deduce
Theorem \ref{small scale dim-free}. 
\begin{Lem} \label{fun fact for sup sym conv body}
For any $K \in \N$, $\varepsilon>0$ there exists $a \in \N$ such that 
for all $d \in \N$, $G \in \mathcal{F}_d$ satisfying $\widetilde{\kappa}(G) \le d^{-\frac{1+\varepsilon}{K+1}}$ we have 
\[
\spr{x \in G \cap \Z^d} \Big(\sum_{\substack{i=1, \\ |x_i| \ge K+1}}^d
|x_i| > a \Big) \le \frac{4}{d}.
\]
\end{Lem}
\begin{proof}
For $d < \max(2^{K+1},4^{\frac{K+2}{\varepsilon}})$ the result trivially holds if $a  \ge \max(2^{K+1},4^{\frac{K+2}{\varepsilon}})$. 
Let us consider $d \ge \max(2^{K+1},4^{\frac{K+2}{\varepsilon}})$, fix $a$ to be determined later then by Lemma \ref{big_coord small scales} we get  
\begin{align*}
&\spr{x \in G \cap \Z^d} \Big(\sum_{\substack{i=1, \\ |x_i| \ge K+1}}^d 
|x_i| \ge a \Big)  \le \sum_{r \le d } \sum_{S \subseteq [d], |S|=r} \sum_{\substack{j_i \ge K+1,\\i \in S, \\
\sum_{i\in S} j_i \ge a}}\spr{x \in G \cap \Z^d} \Big( (\forall i \in S) |x_i| =j_i\Big) \\
&\le \sum_{r \le d } \sum_{S \subseteq [d], |S|=r} \sum_{\substack{j_i \ge K+1,\\i \in S, \\
\sum_{i\in S} j_i \ge a}} 2^r \widetilde{\kappa}(G)^{\sum_{i \in S}j_i} = \sum_{m \ge a}\sum_{r \le \frac{m}{K+1} } \sum_{S \subseteq [d], |S|=r} \sum_{\substack{j_i \ge K+1,\\i \in S, \\
\sum_{i\in S} j_i =m}} 2^r \widetilde{\kappa}(G)^{\sum_{i \in S}j_i} \\
&\le \sum_{m \ge a}\sum_{r \le \frac{m}{K+1} } \sum_{S \subseteq [d], |S|=r}  2^r 2^{m} \widetilde{\kappa}(G)^m \le 2 \sum_{m \ge a} (2^{K+2}d)^{ \frac{m}{K+1}} \widetilde{\kappa}(G)^m \le 2 \sum_{m \ge a} \Big(\frac{2^{K+2}}{d^\varepsilon} \Big)^{\frac{m}{K+1}} \le 4d ^{-\frac{a\varepsilon}{2(K+1)}} \le \frac{4}{d},
\end{align*}
last inequality holds for $a \ge \frac{2(K+1)}{\varepsilon}$.
In conclusion the lemma holds with $a=\max(2^{K+1},4^{\frac{K+2}{\varepsilon}}, \frac{2(K+1)}{\varepsilon})$.
\end{proof}
\begin{Rem}
    It is not possible to remove $\varepsilon$ from the statement of Lemma \ref{fun fact for sup sym conv body}. Fix $a,K \in \N$, we briefly sketch without details, that for arbitrarily large $d$ and $G= (\frac{1}{100}d^{K/(K+1)}) B^{1,(d)}$ we have
    \[
    \spr{x \in G \cap \Z^d} \Big(\sum_{\substack{i=1, \\ |x_i| \ge K+1}}^d
|x_i| > a_0 \Big) \approx_{a_0,K} 1.
    \]
    Replacing $a_0$ by a bigger number $a$ we can assume that $b:=\frac{a}{K+1} \in \Z$ and consider weak inequality instead $\ge$, for simplicity we will also assume that $d^{K/(K+1)} \in \Z$ and $100 \mid d^{K/(K+1)}$. Using Corollary \ref{small scales no big coord} we get 
    \begin{align*}
        &\spr{x \in G \cap \Z^d} \Big(\sum_{\substack{i=1, \\ |x_i| \ge K+1}}^d
|x_i| \ge a \Big) \ge \binom{d}{b} \frac{ \bigg|(\frac{1}{100}d^{K/(K+1)}- a) B^{1,(d-b)} \cap \Big\{-K,\ldots, K\Big\}^{d-b}\bigg|}{|\frac{1}{100}d^{K/(K+1)}B^{1,(d)} \cap \Z^d|} \\
    &\gtrsim_{a,K} d^{b} \frac{ \bigg|(\frac{1}{100}d^{K/(K+1)}- a) B^{1,(d-b)} \cap \Z^{d-b}\bigg|}{|\frac{1}{100}d^{K/(K+1)}B^{1,(d)} \cap \Z^d|}. 
    \end{align*}
    Now using \cite[Theorem 1.4]{KNW} and performing some straightforward computations one can show that 
    \[
    \frac{ \bigg|(\frac{1}{100}d^{K/(K+1)}- a) B^{1,(d-b)} \cap \Z^{d-b}\bigg|}{|\frac{1}{100}d^{K/(K+1)}B^{1,(d)} \cap \Z^d|} \approx_{a,K} d^{-b}
    \]
    for all sufficiently large $d$.
\end{Rem}
We finish this section with the proof of Theorem \ref{iso const}. We start with the following proposition.
\begin{Prop} \label{concen small}
     For all $p \in [1,\infty)$ and $d \in \N, G \in \mathcal{F}_d$, such that $\frac{1}{d} \le \widetilde{\kappa}(G) \le 1/2$ we have
    \[
    \se{x \in G \cap \Z^d} |x_1|^p= \se{x \in G \cap \Z^d} \frac{|x|_p^p}{d} \lesssim_{p} \widetilde{\kappa}(G).
    \]
    Moreover if $\widetilde{\kappa}(G) \le 10^{-3}$, then
    \[
    \se{x \in G \cap \Z^d} |x_d|^p= \se{x \in G \cap \Z^d} \frac{|x|_p^p}{d} \gtrsim_{p} \widetilde{\kappa}(G).
    \]
\end{Prop}
\begin{proof}
    From Lemma \ref{big_coord small scales} we have
    \begin{align*}
        &\se{x \in G \cap \Z^d} |x_1|^p= \frac{1}{|G \cap \Z^d|} \sum_{j \ge 1} j^p \Big|\Big\{ x \in G \cap \Z^d: |x_1|=j \Big\}\Big| \le 2 \sum_{j \ge 1} j^p \widetilde{\kappa}(G)^j \lesssim_p \widetilde{\kappa}(G).
    \end{align*}
    Second part of the proposition follows from Lemma \ref{BMSW a lot of 1}.
\end{proof}
Combining previous results we can finally obtain Theorem \ref{iso const}.
\begin{proof}[Proof of Theorem \ref{iso const}]
    Due to Propositions \ref{concen medium} and \ref{concen small}. It suffices to show the following.
    \begin{itemize}
        \item If $d \ge 20$ and $\widetilde{\kappa}(G) \in [10^{-3},10]$, then
        \[
        \se{x \in G \cap \Z^d} |x_1|^p  \gtrsim_p 1.
        \]
        \item If $d<20$ and $\widetilde{\kappa}(G) \ge 10^{-3}$, then 
        \[
        \se{x \in G \cap \Z^d} |x_1|^p \gtrsim_p \widetilde{\kappa}(G)^p.
        \]
    \end{itemize}
    Let us focus first on the case $\widetilde{\kappa}(G) \in [10^{-3},10]$.
    We have the trivial inclusion
    \[
    d \widetilde{\kappa}(G) B^{1,(d)} \cap \{-1,0,1 \}^d \subseteq G \cap \Z^d.
    \]
    From the above we get 
    \begin{equation} \label{abc}
        |G \cap \Z^d| \ge 2^{ \min(d,d\widetilde{\kappa}(G))} \binom{d}{\min(d,d \widetilde{\kappa}(G))} \ge \max\Big(2^{d \widetilde{\kappa}(G)/10}, \widetilde{\kappa}(G)^{-d \widetilde{\kappa}(G)}\Big).
    \end{equation}
    In \cite[Theorem 1.4]{KNW} it was proved for all $n \le d/2$ we have
    \[
    |nB^{1,(d)} \cap \Z^d| \approx (2e)^n \Big( \frac{n}{d}\Big)^{-n}  \frac{1}{\sqrt{n}}\exp \big( nb(\frac{n}{d}) \big),
    \]
    where $b(z)=\sum_{k=1}^\infty b_k z^{2k}$ is certain power series converging uniformly in $|z| \le 1/2$.
    Take $\varepsilon>0$ small, which will be determined later, from the result above we get 
    \[
    |\big(\varepsilon d \widetilde{\kappa}(G)\big)B^{1,(d)} \cap \Z^d| \lesssim C^{\varepsilon d \widetilde{\kappa}(G)} \Big(\varepsilon \widetilde{\kappa}(G) \Big)^{-\varepsilon d \widetilde{\kappa}(G)}
    \]
    for some absolute constant $C>0$. Hence for $\varepsilon$ fixed, sufficiently small we get 
    \begin{equation} \label{xyz}
        |\big(\varepsilon d \widetilde{\kappa}(G)\big)B^{1,(d)} \cap \Z^d| \le C' 2^{d \widetilde{\kappa}(G)/1000} \cdot \widetilde{\kappa}(G)^{-d \widetilde{\kappa}(G)/100}.
    \end{equation}
    This way for $\widetilde{\kappa}(G) \in [10^{-3},10]$ from \eqref{abc} and \eqref{xyz} we obtain
    \[
    \spr{x \in G \cap \Z^d}\Big(|x|_1 \le \varepsilon d \widetilde{\kappa}(G)\Big) \le C 2^{-d/10^9} \le \frac{1}{2},
    \]
    where the last inequality holds for $d>C_2$, where $C_2$ is some absolute constant, thus we have
    \[
    \se{x \in G \cap \Z^d} \frac{|x|_p^p}{d} \ge \se{x \in G \cap \Z^d}  \Big( \frac{|x|_1}{d} \Big)^p \ge \frac{1}{2} \varepsilon^p \widetilde{\kappa}(G)^p \approx_p \widetilde{\kappa}(G)^p .
    \]
    It remains to treat $d \le C_2$ and $G \in \mathcal{F}_d$ with any value of $\widetilde{\kappa}(G)$. Note that if $\frac{1}{d} \le \widetilde{\kappa}(G) \le 100$, then by Lemma \ref{Cor 2.2 from my lq paper} and \eqref{disc_incl} we obtain
    \[
    |G \cap \Z^d| \le |d \kappa(G) B^{1,(d)} \cap \Z^d| \le 2 \cdot 10^{3d} \lesssim 1.
    \]
    Hence in this case we get 
    \[
    \se{x \in G \cap \Z^d} |x_1|^p \ge \frac{1}{|G \cap \Z^d|} \gtrsim 1 \approx  \widetilde{\kappa}(G). 
    \]
    If $\widetilde{\kappa}(G) \ge 100$ and $d \le C_2$, then again by Lemma \ref{Cor 2.2 from my lq paper} and \eqref{disc_incl} we have 
    \[
    \se{x \in G \cap \Z^d} |x_1|^p \ge \frac{1}{|G \cap \Z^d|} \sum_{x \in [-\widetilde{\kappa}(G),\widetilde{\kappa}(G)]^d \cap \Z^d} |x_1|^p \ge C^{-d} \widetilde{\kappa}(G)^p \approx_p \widetilde{\kappa}(G)^p.
    \]
    This exhausts all possible cases and thus finishes the proof of Theorem \ref{iso const}.
\end{proof}

\section{Proof of proposition \ref{counter-example}}
\begin{Defn}[Direct sum of convex bodies]
    For convex bodies $A \subseteq \R^{d_1}$, $B \subseteq \R^{d_2}$ we define their direct sum by
    \begin{align*}
        A \oplus B&= \Big\{(x,y) \in \R^{d_1} \times \R^{d_2}: \|x \|_{A}+ \|y \|_{B} \le 1 \Big\} \\
        &= \Big\{ \lambda (x,0)+ (1-\lambda)(0,y) \in  \R^{d_1} \times \R^{d_2}: \lambda \in [0,1], x \in A, y \in B \Big\}.
    \end{align*}
    Note that $A \oplus B$ is also a convex body.
    More generally for convex bodies $A_1 \subseteq \R^{d_1},A_2 \subseteq \R^{d_2}, \ldots,A_k \subseteq \R^{d_k}$ we define recursively
    \[
    \bigoplus_{i=1}^k A_i= \bigg(\bigoplus_{i=1}^{k-1} A_i \bigg) \oplus A_k= \Big\{(x^{(1)}, \ldots, x^{(k)}) \in \R^{d_1} \times \ldots \times \R^{d_k}: \sum_{i=1}^k \|x^{(i)}\|_{A_i} \le 1 \Big\}.
    \]
\end{Defn}

\begin{Lem} \label{isotropy of direct sum}
    Let $A_1 \subseteq \R^{d_1},A_2 \subseteq \R^{d_2}, \ldots,A_k \subseteq \R^{d_k}$ be 1-unconditional convex bodies in isotropic position, let $d= \sum_{i=1}^k d_i$. Then $\bigoplus_{i=1}^k A_i$ is in isotropic position if and only if for all $i,j \in [k]$ we have
    \[
    L(A_i)^2(d_i+2)(d_i+1)=L(A_j)^2(d_j+2)(d_j+1).
    \]
    Moreover if the above holds, then for all $i \in [k]$ we have 
    \[
    L\bigg(\bigoplus_{i=1}^k A_i \bigg)^2 (d+2)(d+1)=L(A_i)^2(d_i+2)(d_i+1).
    \]
\end{Lem}    
\begin{proof}
        It is not difficult to see, that due to induction it suffices to consider the case $k=2$. Let $A=A_1$, $B=A_2$ and $d=d_1+d_2$. We start with computing volume of $A \oplus B$.
        Using polar coordinates we have 
        \begin{align*}
         &\Vol(A \oplus B)= \int_{A \oplus B} du= \int_{\|x \|_A+ \|y \|_B \le 1} dx dy =   \int_{\|x \|_A\le 1}  \int_{\|y\|_B \le 1- \|x \|_A} dy dx \\
         & =\int_{\|x \|_A\le 1} \Big(1- \| x \|_A\Big)^{d_2}  \int_{\|y\|_B \le 1} dy dx = \Vol(B) \int_{\|x \|_A \le 1} \Big(1- \|x \|_A \Big)^{d_2} dx \\
         &= \Vol(B) \int_{S^{d_1-1}} \int_{0}^{\rho_A(\theta)} \Big(1- r \| \theta \|_A \Big)^{d_2} r^{d_1} \frac{dr}{r} d \sigma(\theta) \\
         &=\Vol(B) \int_{S^{d_1-1}} \| \theta \|_A^{-d_1} \int_0^1 \big(1-r\big)^{d_2} r^{d_1} \frac{dr}{r} d \sigma(\theta) \\
         &= \Vol(B) \frac{\Gamma(d_1)\Gamma(d_2+1)}{\Gamma(d+1)} \int_{S^{d_1-1}} \rho_A(\theta)^{d_1} d \sigma( \theta) \\
         &=\Vol(B) \frac{\Gamma(d_1)\Gamma(d_2+1)}{\Gamma(d+1)} d_1 \int_{S^{d_1-1}} \int_0^{\rho_A(\theta)} r^{d_1-1} dr d \sigma(\theta)= \Vol(B) \frac{\Gamma(d_1+1)\Gamma(d_2+1)}{\Gamma(d+1)} \int_{A} 1 dx \\
         & =\frac{\Gamma(d_1+1)\Gamma(d_2+1)}{\Gamma(d+1)} \Vol(A) \Vol(B).
        \end{align*}
        Now let us compute integral of $u_i^2$ over $A \oplus B$ for $u \in \R^d$. First assume that $i \le d_1$, by using polar coordinates we have 
        \begin{align*}
            &\int_{A \oplus B} u_i^2du= \int_{\|x \|_A+ \|y \|_B \le 1} x_i^2dx dy =   \int_{\|x \|_A\le 1} x_i^2  \int_{\|y\|_B \le 1- \|x \|_A} dy dx \\
         & =\int_{\|x \|_A\le 1} x_i^2\Big(1- \| x \|_A\Big)^{d_2}  \int_{\|y\|_B \le 1} dy dx = \Vol(B) \int_{\|x \|_A \le 1} x_i^2\Big(1- \|x \|_A \Big)^{d_2} dx \\
         &= \Vol(B) \int_{S^{d_1-1}} \int_{0}^{\rho_A(\theta)} \big(r \theta_i\big)^2 \Big(1- r \| \theta \|_A \Big)^{d_2} r^{d_1} \frac{dr}{r} d \sigma(\theta) \\
         &=\Vol(B) \int_{S^{d_1-1}} \theta_i^2\| \theta \|_A^{-d_1-2} \int_0^1 \big(1-r\big)^{d_2} r^{d_1+2} \frac{dr}{r} d \sigma(\theta) \\
         &=\Vol(B) \frac{\Gamma(d_1+2)\Gamma(d_2+1)}{\Gamma(d+3)} \int_{S^{d_1-1}} \theta_i^2 \rho_A(\theta)^{d_1+2} d \sigma(\theta) \\
         &=\Vol(B) \frac{\Gamma(d_1+2)\Gamma(d_2+1)}{\Gamma(d+3)} (d_1+2) \int_{S^{d_1-1}} \int_0^{\rho_A(\theta)} \big(r \theta_i\big)^2 r^{d_1} \frac{dr}{r} d \sigma(\theta) \\
         &=\Vol(B) \frac{\Gamma(d_1+3)\Gamma(d_2+1)}{\Gamma(d+3)} \int_{A} x_i^2 dx \\
         &=L(A)^2\Vol(A)\Vol(B) \frac{\Gamma(d_1+3)\Gamma(d_2+1)}{\Gamma(d+3)}.
        \end{align*}
        This way for $i \le d_1$ we get 
        \begin{equation} \label{eq:L(A)}
            \frac{1}{\Vol(A\oplus B )} \int_{A \oplus B} u_i^2 du= \frac{L(A)^2 (d_1+2)(d_1+1)}{(d+2)(d+1)},
        \end{equation}
        similarly one can show that for $i>d_1$ we have 
        \begin{equation} \label{eq:L(B)}
            \frac{1}{\Vol(A\oplus B )} \int_{A \oplus B} u_i^2 du= \frac{L(B)^2 (d_2+2)(d_2+1)}{(d+2)(d+1)}.
        \end{equation}
        Since $A \oplus B$ is 1-unconditional we get that $A \oplus B$ is in isotropic position if and only if 
        \[
        L(A)^2 (d_1+2)(d_1+1)=L(B)^2 (d_2+2)(d_2+1).
        \]
        Moreover if the above holds, then
        \[
        L(A \oplus B)^2=\frac{L(A)^2 (d_1+2)(d_1+1)}{(d+2)(d+1)} =\frac{L(B)^2 (d_2+2)(d_2+1)}{(d+2)(d+1)}.
        \]
\end{proof}
\begin{Defn}
    For $p \in [1,\infty]$, $d \in \N$ we define 
    \[
    c(p,d)=L(d^{1/p}B^{p,(d)})^{-1}.
    \]
\end{Defn}
Note that by \eqref{comparability of LG and kappaG}
we have $c(p,d) \in [ \frac{1}{\sqrt{2}}, \sqrt{3 \pi e}]$. Moreover for fixed $d \in \N$, $ p \mapsto c(p,d)$ is a continuous function.
\par We now proceed to the proof of Proposition \ref{counter-example}.
\begin{proof}{Proof of Proposition \ref{counter-example}}
It suffices to consider only $K \in \N$,
fix $K \in \N$. We will define recursively $d_i \in \N$, $p_i \in [1, \infty)$ for $i \in [K]$. Let $d_1=11$, note that 
\[
 \frac{c(1,d_1)d_1}{\sqrt{(d_1+2)(d_1+1)}} \ge \frac{11}{\sqrt{2} \cdot \sqrt{13 \cdot 12}}>\frac{1}{2}
 , \quad \text{and} \quad  \frac{ c(\infty, d_1)}{\sqrt{(d_1+2)(d_1+1)}} \le \frac{\sqrt{3 \pi e}}{\sqrt{13 \cdot 12}}<\frac{1}{2}.
\]
Hence there exists $p_1 \in [1, \infty)$ such that 
\[
 \frac{c(p_1,d_1) d_1^{1/p_1}}{\sqrt{(d_1+2)(d_1+1)}}= \frac{1}{2}.
\]
Assume that we have already defined $d_1, \ldots, d_j$ and $p_1, \ldots ,p_j$ for some $j<K$. Then we simply define $d_{j+1} \in \N$ and $p_{j+1}$ in the following way 
\[
 d_{j+1}= \prod_{i=1}^{j} 2^{(j+1)d_i  + 3 d_i+1} 
\]
and $p_{j+1} \in [1, \infty)$ such that 
\[
  \frac{c(p_{j+1},d_{j+1}) d_{j+1}^{1/p_{j+1}}}{\sqrt{(d_{j+1}+2)(d_{j+1}+1)}}= 2^{-j-1},
\]
existence of such $p_{j+1}$ can be justified in the exact same way as existence of $p_1$ was justified.
This way we have defined $d_1,...,d_K \in \N$ and $p_1, \ldots p_K \in (1, \infty]$ so that 
\begin{equation} \label{counter example property 1}
    (\forall l \in [K]) \  d_l =\prod_{i=1}^{l-1} 2^{l d_i  + 3 d_i+1}
\end{equation}
and
\begin{equation} \label{counter example property 2}
    (\forall l \in [K]) \ \frac{ c(p_l,d_l) d_l^{1/p_l}}{\sqrt{(d_l+2)(d_l+1)}}= 2^{-l}.
\end{equation}
Take any $d \ge K^2 4^K d_K^2 >11$ and define $d_{K+1}=d- \sum_{i \in [K]} d_i$.
Now for every $i \in [K]$ we introduce the following sets 
\[
A_i=  \frac{c(p_i,d_i)d_i^{1/p_i}}{\sqrt{(d_i+2)(d_i+1)}}B^{p_i,(d_i)} =2^{-i}B^{p_i,(d_i)} \subseteq \R^{d_i}
\]
and 
\[
A_{K+1}=  \frac{c( \infty, d_{K+1})}{\sqrt{(d_{K+1}+2)(d_{K+1}+1)}} B^{\infty,(d_{K+1})} \subseteq \R^{d_{K+1}}.
\]
Then we define 
\begin{equation} \label{defn of counter example}
    G= \bigoplus_{i=1}^{K+1} A_i \subseteq \R^{d}.
\end{equation}
By definition of $c(p_i,d_i)$ and Lemma \ref{isotropy of direct sum} we have that $G$ is in isotropic position.
Note that 
\begin{align*}
    &G= \Big\{ \Big( \frac{\lambda_1}{2}x^{(1)},\ldots, \frac{\lambda_K}{2^K}x^{(K)}, \lambda_{K+1}y) \in \R^{d_1} \times \ldots \R^{d_K} \times \R^{d_{K+1}}: (\forall i \in [K+1]) \lambda_i \in [0,1],\sum_{i=1}^{K+1} \lambda_i=1, \\
    & (\forall i \in [K]) x^{(i)} \in B^{p_i,(d_i)}, \ y \in  \frac{c( \infty, d_{K+1})}{\sqrt{(d_{K+1}+2)(d_{K+1}+1)}} B^{\infty,(d_{K+1})} \Big\}.
\end{align*}
For $i \in [K]$ let $m_i= \sum_{l<i} d_l$. Notice that for any $j \in [K]$ we have 
\begin{align*}
    &\Big\{\pm e_{m_j+1}, \ldots, \pm e_{m_j+d_j} \Big\} \subseteq  2^jG \cap \Z^d \\
    &\subseteq  \Big\{\pm e_{m_j+1}, \ldots, \pm e_{m_j+d_j} \Big\} \cup \bigg(\Big( 2^{j-1} B^{p_1,(d_1)} \Big) \times \ldots \times \Big( 2B^{p_{j-1},(d_{j-1})}\Big) \times \{ 0 \}^{d-m_j} \bigg).
\end{align*}
This way using Lemma \ref{Cor 2.2 from my lq paper} and assumption on $d_j$ we get
\begin{align*}
    & 2d_j \le \Big|2^j G \cap \Z^d \Big| \le \Big(2d_j+ \prod_{i=1}^{j-1} 2^{j d_i+3d_i+1} \Big) \le 3d_j.
\end{align*}
Moreover note that for any $j \in [K]$ we have
\[
\widetilde{\kappa}\Big(2^j G\Big) \le d^{-1} 2^j \sum_{i=1}^j d_i \le d^{-1} 2^K K d_K \le d^{-1/2}
\]
Now we define the same function $f: \Z^d \to \C$ as in \cite[Proof of Theorem 2]{cubes}. 
For $s \in [K]$ let 
\[
E_s= \Big\{y \in \Z^{d_s}: (\forall i \in [d_s])\  |y_i| \le 2^d, \ \text{and} \ \sum_{i=1}^{d_s} y_i \ \text{is odd} \Big\}
\]
and 
\[
f(x)= \mathds{1}_{E_1 \times \ldots \times E_K \times \{0\}^{d_{K+1}}}(x).
\]
Note that 
\[
\frac{1}{3} \Big(2^{d+1}+1\Big)^{d_s} \le |E_s| \le \Big(2^{d+1}+1\Big)^{d_s}
\]
Moreover for any $j \in [K]$ and $x=(x^{(1)}, \ldots, x^{(K)}, \overrightarrow{0}) \in \R^{d_1} \times \ldots \times \R^{d_K} \times \{0\}^{d_{K+1}}$ we have 
\begin{align*}
    &\mathcal{M}_{2^j}^Gf(x) \ge \frac{1}{|2^j G \cap \Z^d|} \sum_{y \in 2^jG \cap \Z^d} f(x-y) \ge  \frac{1}{3 d_j} \sum_{i=m_j+1}^{m_j+d_j} \big(f(x + e_i)+ f(x-e_i)\big) \\
    & \ge \frac{2}{3} \prod_{i \in [K] \setminus \{j\}} \mathds{1}_{E_i}(x^{(i)}) \cdot \mathds{1}_{E'_j}(x^{(j)}),
\end{align*}
where 
\[
E_j'= \Big\{y \in \Z^{d_j}: (\forall i \in [d_j]) |y_i| \le 2^{d}-1, \ \text{and} \ \sum_{i=1}^{d_j} y_i \ \text{is even} \Big\}.
\]
Note that 
\[
E_j' \cap E_j= \emptyset, \ \text{and} \ |E'_j| \ge \Big(2^{d+1}-1\Big)^{d_j-1} \Big(2^d-1 \Big) \ge \frac{1}{4} \Big(2^{d+1}+1\Big)^{d_j} \ge \frac{1}{4} |E_j|,
\]
above we've used inequalities 
\[
 \frac{2^{d+1}+1}{2^d-1} \le 3, \quad \text{and} \quad \Big(\frac{2^{d+1}+1}{2^{d+1}-1} \Big)^{d_j-1} < \exp \Big( \frac{2d}{2^{d+1}-1} \Big) < \exp \Big( \frac{22}{2^{12}-1}\Big)< \frac{4}{3},
\]
which holds since $d \ge 11$.
This way if we define
\[
B_j:=E_1 \times \ldots E_{j-1} \times E'_j \times E_{j+1} \times \ldots E_{K} \times \{0\}^{d_{K+1}}
\]
then $B_j \cap B_{i} = \emptyset$ for any $i,j \in [K]$ with $i \neq j$, moreover 
\[
|B_j| \ge \frac{1}{4} \prod_{i=1}^K |E_i| \ge \frac{1}{4} \|f \|_{\ell^p(\Z^d)}^p.
\]
Thus for any $p \in (1, \infty)$ we obtain
\begin{align*}
    & \sup_{\substack{n \in \Z, \\ \widetilde{\kappa}(2^n G) \le d^{-1/2}}} |\mathcal{M}_{2^n}^G f(x)|  \ge \max_{1 \le j \le K} \mathcal{M}_{2^j}^G f(x) \ge \frac{2}{3} \max_{1 \le j \le K} \mathds{1}_{B_j}(x) \\
    &\ge \frac{2}{3} \bigg( \sum_{j=1}^K \mathds{1}_{B_j}(x) \bigg)^{1/p},
\end{align*}
where above we have used disjointness of sets $B_j$. Thus we finally obtain
\begin{align*}
    \bigg\|\sup_{\substack{n \in \Z, \\ \widetilde{\kappa}(2^n G) \le d^{-1/2}}} |\mathcal{M}_{2^n}^G f| \bigg\|_{\ell^p(\Z^d)} &\ge \frac{2}{3} \cdot \bigg( \sum_{j=1}^K |B_j| \bigg)^{1/p} \ge \frac{2}{3} 4^{-1/p} K^{1/p} \|f \|_{\ell^p (\Z^d)}\\
    &\ge \frac{1}{6}K^{1/p} \|f \|_{\ell^p (\Z^d)},
\end{align*}
this completes the proof, since $K$ was chosen to be arbitrary.
\end{proof}
\section{Dyadic maximal function in large-scale regime}
With insights gained in previous sections we can easily adapt arguments from \cite[Section 8.2]{balls}.
In this section we will prove the following.
\begin{Thm} \label{large scale}
    We have
    \[
    \sup_{d \in \N} \sup_{G \in \mathcal{F}_d} \Big\| \sup_{\substack{n \in \Z, \\ 1 \le 2^n \kappa(G) \le d}} \big|\mathcal{M}_{2^n}^G\big| \Big\|_{\ell^2(\Z^d) \to \ell^2(\Z^d)}< \infty.
    \]
\end{Thm}
Let us recall that in \cite[Theorem 1]{cubes} it was showed that for any symmetric convex body $G \subseteq \R^d$ we have
\[
    \Big\| \sup_{t \ge c(G) d} \big| \mathcal{M}_t^G  \big| \Big\|_{\ell^2(\Z^d) \to \ell^2(\Z^d)} \le e^6 \Big\| \sup_{t>0} |M_t^G| \Big\|_{L^2(\R^d) \to L^2(\R^d) },
    \]
    where $c(G)= \inf\{t>0: [-1/2,1/2]^d \subseteq tG \}$. Due to Bourgain's classical result \cite{B1}, we see that the right-hand side of the above inequality is bounded by an absolute constant. In particular for any $G \in \mathcal{F}_d$, because of Lemma \ref{lem:kappa_def}, we have $c(G)= \frac{1}{2 \kappa(G)}$ and get 
\[
\Big\| \sup_{t \ge  \frac{d}{2 \kappa(G)}} \big| \mathcal{M}_t^G  \big| \Big\|_{\ell^2(\Z^d) \to \ell^2(\Z^d)} \lesssim 1. 
\]
Inequality above combined with Theorem \ref{large scale} leads to the following corollary.
\begin{Cor} \label{large scale no upper restr}
    We have
    \[
    \sup_{d \in \N} \sup_{G \in \mathcal{F}_d} \Big\| \sup_{\substack{n \in \Z, \\ 1 \le 2^n \kappa(G) }} \big|\mathcal{M}_{2^n}^G \big| \Big\|_{\ell^2(\Z^d) \to \ell^2(\Z^d)}< \infty.
    \]
\end{Cor}
Proof of the Theorem \ref{large scale} relies on two propositions.
\begin{Prop} \label{large scale m-1 bound}
    For every $d \in \N, G \in \mathcal{F}_d$ and every $\xi \in \T^d$ we have
    \[
    |\mathfrak{m}_G(\xi)-1| \lesssim \max\Big(\widetilde{\kappa}(G),\widetilde{\kappa}(G)^2 \Big)  \| \xi\|^2.
    \]
    In particular if $\kappa(G) \ge 10$, then
    \[
    |\mathfrak{m}_G(\xi)-1| \lesssim \min\big(1, \kappa(G)^2 \| \xi\|^2\big) \lesssim \kappa(G) \| \xi \|.
    \]
\end{Prop}
\begin{Prop} \label{large scale m bound}
    For every $d \in \N, G \in \mathcal{F}_d$ with $10 \le \kappa(G) \le d$ and every $\xi \in \T^d$ we have
    \begin{equation} \label{eq:large m bound}
    |\mathfrak{m}_G(\xi)| \lesssim \Big( \kappa(G) \| \xi\| \Big)^{-1} + \kappa(G)^{-1/7}.
    \end{equation}
\end{Prop}
Let us see one can deduce Theorem \ref{large scale} from Propositions \ref{large scale m-1 bound}, \ref{large scale m bound}.
\begin{proof}[Proof of Theorem \ref{large scale}]
Take any $d \in \N$, $G \in \mathcal{F}_d$, $f \in \ell^2(\Z^d)$. Let
\[
\mathfrak{p}_t(\xi)= e^{-t \sum_{i=1}^d \sin^2(\pi \xi_i)}.
\]
It follows from theory of symmetric diffusion semigroups, that
\[
\Big\|\sup_{t>0} \big| \mathcal{F}^{-1}(\mathfrak{p}_t \widehat{f}) \big| \Big\|_{\ell^2(\Z^d)} \lesssim \|f \|_{\ell^2(\Z^d)},
\]
see for instance \cite[p. 73]{Ste1} and \cite[Theorem 2.11]{NW}.
Fix $n \in \Z$ such that $10 \le 2^n \kappa(G) \le d$. By applying Propositions \ref{large scale m-1 bound} and \ref{large scale m bound} to $2^n G$ and invoking the inequalities $1 - e^{-x} \le \min(1, x) \le x^{1/2}$ and $e^{-x} \le \min(1,x^{-1}) \le x^{-1/2}$ (for $x > 0$), we obtain for every $\xi \in \T^d$
\[
|\mathfrak{m}_{2^nG}(\xi)-\mathfrak{p}_{4^n \kappa(G)^2}(\xi)| \lesssim \min \Big( 2^n \kappa(G) \|\xi \|, \big(2^n \kappa(G) \| \xi \|\big)^{-1} \Big)+ \Big(2^n\kappa(G) \Big)^{-1/7}.
\]
Thus we get
\begin{align*}
    &\Big\| \sup_{\substack{n \in \Z, \\
    1 \le 2^n \kappa(G) \le d}}\big| \mathcal{M}_{2^n}^G f\big| \Big\|_{\ell^2(\Z^d)} \le 
    \Big\| \sup_{\substack{n \in \Z, \\
    10 \le 2^n \kappa(G) \le d}} \big|\mathcal{M}_{2^n}^G f \big|\Big\|_{\ell^2(\Z^d)} +
    \Big\| \sup_{\substack{n \in \Z, \\
    1 \le 2^n \kappa(G) <10}} \big|\mathcal{M}_{2^n}^G f \big|\Big\|_{\ell^2(\Z^d)} \\
    &\lesssim  \bigg\| \sup_{\substack{n \in \Z, \\
    10 \le 2^n \kappa(G) \le d}} \Big|\mathcal{F}^{-1}\Big(\big(\mathfrak{m}_{2^nG}-\mathfrak{p}_{4^n \kappa(G)^2} \big)\widehat{f} \Big) \Big|\bigg\|_{\ell^2(\Z^d)} +\Big\|\sup_{t>0} \big| \mathcal{F}^{-1}(\mathfrak{p}_t \widehat{f}) \big| \Big\|_{\ell^2(\Z^d)}+ \|f \|_{\ell^2(\Z^d)} \\
    &\lesssim
    \bigg\| \sup_{\substack{n \in \Z, \\
    10 \le 2^n \kappa(G) \le d}} \Big|\mathcal{F}^{-1}\Big(\big(\mathfrak{m}_{2^nG}-\mathfrak{p}_{4^n \kappa(G)^2}\big) \widehat{f} \Big) \Big|\bigg\|_{\ell^2(\Z^d)} + \|f \|_{\ell^2(\Z^d)}.
\end{align*}
It remains to show that 
\begin{equation*}
     \bigg\| \sup_{\substack{n \in \Z, \\
    10 \le 2^n \kappa(G) \le d}} \Big|\mathcal{F}^{-1}\Big(\big(\mathfrak{m}_{2^nG}-\mathfrak{p}_{4^n \kappa(G)^2}\big) \widehat{f} \Big) \Big|\bigg\|_{\ell^2(\Z^d)} \lesssim \|f \|_{\ell(\Z^d)}.
\end{equation*}
By a standard square function argument and Parseval's theorem we have
\begin{align*}
    &\bigg\| \sup_{\substack{n \in \Z, \\
    10 \le 2^n \kappa(G) \le d}} \Big|\mathcal{F}^{-1}\Big(\big(\mathfrak{m}_{2^nG}-\mathfrak{p}_{4^n \kappa(G)^2}\big) \widehat{f}  \Big)\Big|\bigg\|_{\ell^2(\Z^d)} \lesssim \Bigg\| \bigg(\sum_{\substack{n \in \Z, \\
    10 \le 2^n \kappa(G) \le d}} \Big|\mathcal{F}^{-1}\Big(\big(\mathfrak{m}_{2^nG}-\mathfrak{p}_{4^n \kappa(G)^2}\big) \widehat{f} \Big) \Big|^2 \bigg)^{1/2}\Bigg\|_{\ell^2(\Z^d)} \\
    &= \bigg(\sum_{\substack{n \in \Z, \\
    10 \le 2^n \kappa(G) \le d}} \Big\|\mathcal{F}^{-1}\Big(\big(\mathfrak{m}_{2^nG}-\mathfrak{p}_{4^n \kappa(G)^2} \big)\widehat{f} \Big) \Big\|_{\ell^2(\Z^d)}^2 \bigg)^{1/2}= \bigg(\sum_{\substack{n \in \Z, \\
    10 \le 2^n \kappa(G) \le d}} \int_{\T^d}|(\mathfrak{m}_{2^nG}(\xi)-\mathfrak{p}_{4^n \kappa(G)^2}(\xi)|^2| \widehat{f}(\xi)|^2 d\xi \bigg)^{1/2} \\
    &\lesssim \Bigg(\int_{\T^d}| \widehat{f}(\xi)|^2\bigg(\sum_{\substack{n \in \Z, \\
    10 \le 2^n \kappa(G) \le d}}  \min\Big( 2^n \kappa(G) \|\xi \|, \big(2^n \kappa(G) \| \xi \|\big)^{-1} \Big)^2+\sum_{\substack{n \in \Z, \\
    10 \le 2^n \kappa(G) \le d}} \Big(2^n\kappa(G) \Big)^{-2/7}\bigg) d\xi \Bigg)^{1/2} \\
    &\lesssim \| \widehat{f} \|_{L^2(\T^d)}= \|f \|_{\ell^2(\Z^d)}.
\end{align*}
\end{proof}
We already have sufficient tools to prove Proposition \ref{large scale m-1 bound}.
\begin{proof}
    Note that by invariance under sign changes of $G$ we have
    \begin{align*}
      &|\mathfrak{m}_G(\xi)-1|= \bigg| \se{x \in G \cap \Z^d} e(x \cdot \xi) -1 \bigg|= \bigg| \se{x \in G \cap \Z^d} \Big( \prod_{i=1}^d \cos(2 \pi x_i \xi_i) -1  \Big)\bigg| \\
      &\le \se{x \in G \cap \Z^d}\Big| \prod_{i=1}^d \cos(2 \pi x_i \xi_i) -1  \Big| \le  \se{x \in G \cap \Z^d} \sum_{i=1}^d | \cos(2 \pi x_i \xi_i) -1|  \\
      &=2 \se{x \in G \cap \Z^d} \sum_{i=1}^d \sin^2(\pi x_i \xi_i) \le 2\se{x \in G \cap \Z^d} \sum_{i=1}^d x_i^2\sin^2(\pi \xi_i) \\
      &=2\sum_{i=1}^d \sin^2(\pi \xi_i)\se{x \in G \cap \Z^d} x_i^2.
    \end{align*}
    Above we've used the fact that for any sequences $(a_n)_{n=1}^d, (b_n)_{n=1}^d$ with $\max_{1\le n \le d} |a_n|, \max_{1\le n \le d} |b_n| \le 1$ we have
    \[
    \Big|\prod_{i=1}^d a_i- \prod_{i=1}^d b_i \Big| \le \sum_{i=1}^d |a_i-b_i|.
    \]
    We have also used the inequality $|\sin(\pi n x)| \le n |\sin(\pi x)| $, which holds for all $n \in \Z$, $\xi \in \T$. By invariance under permutations of $G$ and Theorem \ref{iso const} for $p=2$ we have 
    \begin{align*}
        |\mathfrak{m}_G(\xi)-1| &\le2\sum_{i=1}^d \sin^2(\pi \xi_i)\se{x \in G \cap \Z^d} x_i^2\\
        &= 2\sum_{i=1}^d \sin^2(\pi \xi_i)\se{x \in G \cap \Z^d} x_1^2 \lesssim \max\Big( \widetilde{\kappa}(G), \widetilde{\kappa}(G)^2 \Big)  \|\xi \|^2.
    \end{align*}
    The second part of the proposition directly follows from the first part by appealing to Lemma \ref{tilde kappa comparable to kappa in large scale} and trivial estimate.
\end{proof}
\subsection{Dimension decrease} 
The remainder of this section is devoted to the proof of Proposition \ref{large scale m bound}. Our strategy for bounding $\mathfrak{m}_G(\xi)$ involves projecting $G$ onto a lower-dimensional subspace to obtain a new convex body while preserving the magnitude of the constant $\kappa$ (cf. Lemma \ref{m bound by lower dim m}). This reduction allows the comparison of the discrete multiplier corresponding to a lower-dimensional convex body with its continuous counterpart. To successfully implement this approach and recover the desired bound for $\mathfrak{m}_G$, we rely crucially on the permutation invariance of $G$. We begin by recalling two lemmas from \cite{balls}; hereafter, $\mathbb{P}$ denotes the uniform probability measure on the symmetric group $\Sym(d)$.
\begin{Lem}[{\cite[Lemma 2.5]{balls}}] \label{sym inter bound}
Assume that  $I, J\subseteq [d]$ and 
$|J|=r$ for some $0\le r\le d$. Then
\begin{align*}
\spr{\sigma \in \Sym(d)}\Big(|\sigma(I)\cap J|\le {r|I|}/{(5d)}\Big) \le  e^{-\frac{r|I|}{10d}}.
\end{align*}
 In particular, if $\delta_1, \delta_2\in(0, 1]$ satisfy
$5\delta_2\le\delta_1$ and $\delta_1d\le |I|\le d$, then we have
\begin{align*}
\spr{\sigma \in \Sym(d)}\Big( |\sigma(I)\cap
  J|\le \delta_2 r\Big) \le e^{-\frac{\delta_1r}{10}}.
\end{align*}  
\end{Lem}
\begin{Lem}[{\cite[Lemma 2.6]{balls}}]
\label{lem:2.6} Assume that we have a sequence $(u_j: j \in [d])$ with $0 \leq u_j \leq \frac{1- \delta_0}{2}$ for some $\delta_0 \in (0,1)$. Suppose that $I \subseteq [d]$ satisfies $\delta_1d \leq |I| \leq d$ for some $ \delta_1 \in (0,1]$. Then for every $J=(d_0,d] \cap \Z$ with $0 \leq d_0 \leq d$ we have
\[
\se{\tau \in \Sym(d)} \exp\Big(- \sum_{j \in \tau(I) \cap J} u_j\Big) \leq 3 \exp\Big(-\frac{\delta_0 \delta_1}{20} \sum_{j \in J} u_j \Big).
\]
\end{Lem}
\noindent In \cite[Lemma 2.6]{balls} there is an extra assumption that sequence $(u_j: j \in [d])$ is decreasing, however this assumption can be removed, see \cite[Lemma 6.7]{MiSzWr}.
Let $\pi_{[r]}: \R^d \to \R^r$ be defined by $\pi_{[r]}(x)=(x_1,\ldots,x_r)$.
\begin{Lem}
\label{decomp}
For $d \in N$, $G \in \mathcal{F}_d$, $\varepsilon\in(0, 1/50]$ and an integer $1\le r\le d$  we define
\[
E=\{x\in G\cap\Z^d\colon \sum_{i=1}^r|x_i|<\varepsilon^2\kappa(G)r\}.
\]
If $\kappa(G)\ge10$, then we have
\begin{align}
  \label{eq:pre-decomp}
  |E|\le 4e^{-\frac{\varepsilon r}{2}}|G\cap\Z^d|.
\end{align}
Moreover for some $E' \subseteq E$ we have the following disjoint decomposition
\begin{align}
  \label{eq:decomp}
  \begin{split}
  G\cap\Z^d&=\Bigg(\bigcup_{\substack{y \in \pi_{[d-r]}(G \cap \Z^d), \\ 
  \kappa\big(\pi_{[r]}(G;y)\big) \ge \varepsilon^2 \kappa(G)}}\Big(\pi_{[r]}(G;y) \cap \Z^r \Big) \times  \{y\} \Bigg)\cup E',
  \end{split}
  \end{align}
  where $\pi_{[r]}(G;y) \subseteq \R^r$ is a 1-symmetric convex body defined by
  \[
  \pi_{[r]}(G;y)= \Big\{ (x_1,\ldots, x_r) \in \R^r: (x_1,\ldots,x_r,y_1,\ldots,y_{d-r}) \in G \Big\},
  \]
  where the fact that it has non-empty interior follows from the inequality $ \kappa\big(\pi_{[r]}(G;y)\big)>0$. 
\end{Lem}
\begin{proof}
Define $I_x=\{i\in [d]: |x_i|\ge\varepsilon\kappa(G)\}$. We have
$E\subseteq E_1\cup E_2$, where
\begin{align*}
E_1&=\{x\in G\cap\Z^d\colon \sum_{i\in I_x\cap[r]}|
x_i|<\varepsilon^2\kappa(G)r\ \text{ and }\ |I_x|\ge d/10\},\\
E_2&=\{x\in G\cap\Z^d\colon |I_x|<d/10\}.
\end{align*}
By Lemma \ref{pos_prop_conc} we have $|E_2|\le2e^{-\frac{d}{10}}|G\cap\Z^d|$, provided that $\kappa(G)\ge10$. Observe that
\begin{align*}
|E_1|&=\sum_{x\in G\cap\Z^d}\se{\sigma\in{\rm Sym}(d)}\mathds{1}_{E_1}(\sigma^{-1}\cdot x)\\
&=\sum_{x\in G\cap\Z^d}\spr{\sigma \in \Sym(d)}\Big(
\sum_{i\in \sigma(I_x)\cap [r]}|x_{\sigma^{-1}(i)}|<\varepsilon^2\kappa(G)r\ \text{ and }\ |\sigma(I_x)|\ge d/10\Big),
\end{align*}
since $I_{\sigma^{-1}\cdot x}=\sigma(I_x)$. 
Now by Lemma \ref{sym inter bound} (with $J=[r]$, $\delta_2=\varepsilon$ and $\delta_1=1/10$ ) we obtain, for every $x\in G\cap\Z^d$, that
\begin{align*}
  &\spr{\sigma \in \Sym(d)}\Big(
\sum_{i\in \sigma(I_x)\cap [r]}|x_{\sigma^{-1}(i)}|<\varepsilon^2\kappa(G)r\ \text{ and }\ |\sigma(I_x)|\ge d/10\Big)\\
  &\le
\spr{\sigma \in \Sym(d)}\Big( |\sigma(I_x)\cap [r]|\le \varepsilon r\Big)\le 2e^{-\frac{r}{100}}.
\end{align*}
  Thus $|E_1|\le 2e^{-\frac{\varepsilon r}{2}}|G \cap \Z^d|$, which proves \eqref{eq:pre-decomp}. Now we prove \eqref{eq:decomp}, let 
  \[
  \widetilde{\pi}: \R^d \to \R^{d-r}, \quad \widetilde{\pi}(x)=(x_{r+1},\ldots,x_{d}).
  \]
  Note that by symmetry invariance we have $\widetilde{\pi}(G \cap \Z^d)=\pi_{[d-r]}(G \cap \Z^d)$.
  We write
  \begin{align*}
&G\cap\Z^d=\bigcup_{\substack{y \in \widetilde{\pi}(G \cap \Z^d)}}\Big(\pi_{[r]}(G;y) \cap \Z^r \Big) \times  \{y\} = \bigcup_{\substack{y \in \pi_{[d-r]}(G \cap \Z^d)}}\Big(\pi_{[r]}(G;y) \cap \Z^r \Big) \times  \{y\}   \\
&=\Bigg(\bigcup_{\substack{y \in \pi_{[d-r]}(G \cap \Z^d), \\\kappa\big(\pi_{[r]}(G;y)\big) \ge \varepsilon^2 \kappa(G)}}\Big(\pi_{[r]}(G;y) \cap \Z^r \Big) \times  \{y\} \Bigg) \cup \Bigg(\bigcup_{\substack{y \in \pi_{[d-r]}(G \cap \Z^d), \\\kappa\big(\pi_{[r]}(G;y)\big) < \varepsilon^2 \kappa(G)}}\Big(\pi_{[r]}(G;y) \cap \Z^r \Big) \times  \{y\} \Bigg) \\
&=\Bigg(\bigcup_{\substack{y \in \pi_{[d-r]}(G \cap \Z^d), \\\kappa\big(\pi_{[r]}(G;y)\big) \ge \varepsilon^2 \kappa(G)}}\Big(\pi_{[r]}(G;y) \cap \Z^r \Big) \times  \{y\} \Bigg)   \cup E'.
   \end{align*}                
\end{proof}
The next result is a straightforward consequence of Lemma \ref{decomp}, so we skip its proof.
\begin{Lem}
\label{m bound by lower dim m}
For $d\in\N, G \in \mathcal{F}_d$ and $\varepsilon\in(0, 1/50]$ if $\kappa(G)\ge 10$, then for every $1\le r\le d$ and $\xi\in\T^d$ we have
\begin{align}
  \label{eq:m bound by lower dim m}
  | \mathfrak{m}_G(\xi)|\le\sup_{_{\substack{y \in \pi_{[d-r]}(G \cap \Z^d), \\\kappa\big(\pi_{[r]}(G;y)\big) \ge \varepsilon^2 \kappa(G)}}}|
 \mathfrak{m}_{\pi_{[r]}(G;y)}(\xi_1,\ldots, \xi_r)|+4e^{-\frac{\varepsilon r}{2}},
\end{align}
where $\pi_{[d-r]}: \R^d \to \R^{d-r}$ is defined by $\pi_{[d-r]}(x)=(x_1, \ldots, x_{d-r})$ and for $y \in \pi_{[d-r]}(G \cap \Z^d)$ we denoted 
\[
\pi_{[r]}(G;y)=\Big\{ (x_1,\ldots, x_r) \in \R^r: (x_1,\ldots,x_r,y_1,\ldots,y_{d-r}) \in G \Big\}
\]
\end{Lem}
Lemma \ref{m bound by lower dim m} will play an essential role in the proof of Proposition \ref{large scale m bound}. Now our goal will be to bound multipliers appearing on the right-hand side of \eqref{eq:m bound by lower dim m} in terms of $\sum_{i=1}^r \|\xi_i \|^2$.
First we need two simple lemmas regarding behavior of lattice points in a 1-symmetric convex body under translations. 
\begin{Lem}
\label{low dim shift comp}
Let $r\in\N, K \in \mathcal{F}_r $ be  such that   $r\le 2\kappa(K)$ and $\kappa(K) \ge 1/2$ . Then for every $z\in\R^r$ we have
\begin{align}
  \label{eq:low dim shift comp}
  \big||(z+K)\cap\Z^r|-|K|\big|\le  \frac{e r}{2\kappa(K)} |K|.
\end{align}

\end{Lem}
\begin{proof}
Let $Q^{(r)}=[-1/2,1/2]^r$, observe that 
\[
Q^{(r)}+K= \frac{1}{2\kappa(K)}[-\kappa(K),\kappa(K)]^r+K \subseteq \Big(1+\frac{1}{2\kappa(K)}\Big)K.
\]
From the above we have
\begin{align*}
  &|(z+K)\cap\Z^r|=\sum_{x\in
                  \Z^r} \int_{Q^{(r)}}\mathds{1}_{z+K}(x) dy
                \le\sum_{x\in
                  \Z^r} \int_{Q^{(r)}}\mathds{1}_{z+\big(1+\frac{1}{2\kappa(K)}\big)K}(x+y)dy=\Big|z+\Big(1+\frac{1}{2\kappa(K)}\Big)K\Big|,
\end{align*}
and
\[
\Big|z+\Big(1+\frac{1}{2\kappa(K)}\Big)K\Big|=|K|\bigg(1+\frac{1}{2\kappa(K)}\bigg)^r\le
e^{\frac{r}{2\kappa(K)}}|K|\le |K|\big(1+ \frac{e r}{2\kappa(K)}\big),
\]
since $e^x\le(1+xe^x)$ and $r \le 2\kappa(K)$. Arguing in a similar way we obtain
\begin{align*}
  |(z+K)\cap\Z^r|=\sum_{x\in
                  \Z^r} \int_{Q^{(r)}}\mathds{1}_{z+K}(x){\rm d}y
                \ge\sum_{x\in
                  \Z^r} \int_{Q^{(r)}}\mathds{1}_{z+\Big(1-\frac{1}{2\kappa(K)}\Big)K}(x+y){\rm d}y=\Big|z+\Big(1-\frac{1}{2\kappa(K)}\Big)K\Big|,
\end{align*}
and
\[
\Big|z+\Big(1-\frac{1}{2\kappa(K)}\Big)K\Big|=|K|\bigg(1-\frac{1}{2\kappa(K)}\bigg)^r\ge
|K|\big(1- \frac{r}{2 \kappa(K)}\big).
\]
These inequalities imply \eqref{eq:low dim shift comp}.
\end{proof}

\begin{Lem}
\label{low dim shift sym diff}
Let $r\in\N, K \in \mathcal{F}_r $ be  such that   $r\le 2\kappa(K)$ and $\kappa(K) \ge 1/2$. Then for every $z\in\R^r$ such that $r|z|_{\infty} \le \kappa(K)$ we have
\begin{align}
  \label{eq:low dim shift sym diff}
  \begin{split}
  \big|\big(K\cap\Z^r\big)\triangle\big((z+K)\cap\Z^r\big)\big|&\le |K|\frac{er}{\kappa(K)}\Big(|z|_{\infty}+2\Big).
  \end{split}
\end{align}
\end{Lem}

\begin{proof}
Observe that
\[
z \in [-|z|_{\infty},|z|_\infty]^r = \frac{|z|_{\infty}}{\kappa(K)} [-\kappa(K),\kappa(K)]^r \subseteq \frac{|z|_{\infty}}{\kappa(K)}K.
\]
From the above we get
\[
\big(K\cap\Z^r\big)\setminus\big((z+K)\cap\Z^r\big)\subseteq
\bigg(\Big(z+\Big(1+\frac{|z|_{\infty}}{\kappa(K)}\Big)K\Big)\cap\Z^r\bigg)\setminus\Big((z+K)\cap\Z^r\Big). 
\]
Thus by Lemma \ref{low dim shift comp} one has
\begin{align*}
  &\bigg|\bigg(\Big(z+\Big(1+\frac{|z|_{\infty}}{\kappa(K)}\Big)K\Big)\cap\Z^r\bigg)\setminus\Big((z+K)\cap\Z^r\Big)\bigg|=
\bigg|\Big(z+\Big(1+\frac{|z|_{\infty}}{\kappa(K)}\Big)K\Big)\cap\Z^r\bigg|-\big|\big((z+K)\cap\Z^r\big)\big|\\                    &\le|K|\bigg(\Big(1+ \frac{|z|_{\infty}}{\kappa(K)}\Big)^r-1\bigg)+\frac{er}{2(1+|z|_{\infty}/\kappa(K)\big)\kappa(K)}|K|\Big(1+ \frac{|z|_{\infty}}{\kappa(K)}\Big)^r+\frac{er}{2\kappa(K)}|K|\\
&\le|K| \frac{r|z|_{\infty}}{\kappa(K)}\Big(1+ \frac{|z|_{\infty}}{\kappa(K)}\Big)^r+ |K|\frac{er}{2\kappa(K)} \big(1+e^{r|z|_{\infty}/\kappa(K)}\big) \le|K|\frac{er}{\kappa(K)}\Big(|z|_{\infty}+2\Big) .
\end{align*}
One can obtain the same bound for
$\big|\big((z+K)\cap\Z^r\big)\setminus\big(K\cap\Z^r\big)\big|$
and this gives \eqref{eq:low dim shift sym diff}.
\end{proof}

We now recall the dimension-free estimates of the Fourier transform for the multipliers associated with averaging operators \eqref{cont avr op defn} in $\R^d$. For a symmetric convex body $G\subset \R^d$ we define the multiplier 
\[
m_G(\xi)=\frac{1}{\Vol_d(G)}\int_G e(x \cdot \xi) d x\quad\text{for}\quad \xi\in \R^d.
\] 
It is easy to see that $m(t\xi)$ is the multiplier corresponding to the operator $M_t^G$ from \eqref{cont avr op defn}.

Assume that $|G|=1$ and that $G$ is
in the isotropic position, that is there exists a constant $L=L(G)>0$ such that for every unit vector
$\theta \in S^{d-1}$ we have
\begin{align}
\label{eq:25}
\int_G (\theta \cdot x)^2{\rm d}x= L(G)^2.
\end{align}
Then the kernel of the averaging operator \eqref{cont avr op defn} satisfies
\[
|tG|^{-1}\mathds{1}_{tG}(x)=t^{-d}\mathds{1}_{G}(t^{-1} x)
\]
for all $t>0$, and the multiplier satisfies 
\[
m_{tG}(\xi)=\mathcal F (|tG|^{-1}\mathds{1}_{tG})(\xi)=\mathcal F(\mathds{1}_{G})(t\xi)=m_G(t\xi).
\]
The isotropic position of $G$ allows us to provide dimension-free
estimates for the multiplier $m_G$. 
\begin{Thm}[{\cite[eq. (10),(11),(12)]{B1}}]
\label{thm:cont mult bounds}
	Given  a symmetric convex body  $G\subset \R^d$ with volume one,
	which is in the isotropic
	position, there exists a
	constant $C>0$ such that for every $\xi\in\R^d$ we
	have
	\begin{align}
\label{eq:28}
	|m_G(\xi)|\leq C(L |\xi|_2)^{-1},\qquad  |m_G(\xi)-1|\leq CL
	|\xi|_2,\qquad   |\langle\xi,\nabla m_G(\xi)\rangle|\le C.
	\end{align}
	The constant $L=L(G)$ is defined  in \eqref{eq:25}, while $C$ is a
	universal constant which does not depend on $d$.

\end{Thm}
The following lemma is a simple consequence of  Theorem \ref{thm:cont mult bounds} and \eqref{comparability of LG and kappaG}.
\begin{Lem}
	\label{lem: sup sym cont m bound}
	There exists $C>0$, such that for all $r \in \N, K \in \mathcal{F}_r$, $\xi \in \R^r$ we have
	\begin{equation}
\label{eq:29}
	|m_K(\xi)|\leq C (\kappa(K) |\xi|_2)^{-1},
        \end{equation}
\end{Lem}
We can proceed to establishing satisfactory bound for multipliers appearing on the right-hand side of \eqref{m bound by lower dim m}.
\begin{Lem}

\label{m_K bound}
For all $r \in \N,K \in \mathcal{F}_r$ such that $\kappa(K) \ge \max(r^2,2r)$ and every $\eta \in \T^r$ we have  
\begin{align*}
|\mathfrak m_K(\eta)|\lesssim \big(\kappa(K)\|\eta\|\big)^{-1}+ \frac{r^{2/3}}{\kappa(K)^{1/3}}.
\end{align*}
\end{Lem}
\begin{proof}
We will consider two cases separately.
\\
\textbf{Case 1)} Assume that
$\max\{\|\eta_1\|,\ldots,\|\eta_r\|\}>r^{-1/3} \kappa(K)^{-1/3}$.  Let $M=\big\lfloor
\frac{\kappa(K)^{2/3}}{r^{1/3}}\big\rfloor \ge 1$, due to permutation invariance of $K$ we can assume without loss of generality that $\|\eta_1\|>r^{-1/3} \kappa(K)^{-1/3}$. Then
\begin{align}
  \label{eq:48}
  \begin{split}
  |\mathfrak m_K(\eta)|&\le
  \frac{1}{|K\cap\Z^r|}\sum_{x\in
  K\cap\Z^r}\frac{1}{M}\Big|\sum_{s=1}^Me^{2\pi i
  (x+se_1)\cdot\eta}\Big|\\
&+\frac{1}{M}\sum_{s=1}^M\frac{1}{|K\cap\Z^r|}\Big|\sum_{x\in
  K\cap\Z^r}e^{2\pi i
  x\cdot\eta}-e^{2\pi i
  (x+se_1)\cdot\eta}\Big|.
  \end{split}
\end{align}
By our assumptions we see that
\begin{align}
  \label{eq:49}
\frac{1}{M}\Big|\sum_{s=1}^Me^{2\pi i
  (x+se_1)\cdot\eta}\Big|\le M^{-1}\|\eta_1\|^{-1}\le
2 \frac{r^{2/3}}{\kappa(K)^{1/3}}.  
\end{align}
We have assumed that $\kappa(K) \ge \max(r^2,2r)$, thus by Lemma \ref{low dim shift sym diff}, with $z=se_1$
 and $s\le M\le \kappa(K)/r$, we get
 \begin{align}
   \label{eq:50}
   \begin{split}
   \frac{1}{|K\cap\Z^r|}\Big|\sum_{x\in
  K\cap\Z^r}e^{2\pi i
  x\cdot\eta}-e^{2\pi i
  (x+se_1)\cdot\eta}\Big|&\le\frac{1}{|K\cap\Z^r|}\big|\big(K\cap\Z^r\big)\triangle\big((se_1+K)\cap\Z^r\big)\big|\\
                         &\lesssim  \frac{rM}{\kappa(K)} \cdot \frac{|K|}{|K \cap \Z^r|} \lesssim \frac{rM}{\kappa(K)} \lesssim \frac{r^{2/3}}{\kappa(K)^{1/3}},
   \end{split}
 \end{align}
 above we have use the inequalities below, which follow from the proof of Lemma \ref{low dim shift comp} and our assumptions
 \begin{equation} \label{eq:5.12}
|K\cap\Z^r|\ge \Big|\Big(1-\frac{1}{2\kappa(K)}\Big)K\Big|=|K|\bigg(1-\frac{1}{2\kappa(K)}\bigg)^r\ge
|K|\Big(1-\frac{r}{2\kappa(K)}\Big)\ge|K|/2.
 \end{equation}
Combining \eqref{eq:48} with \eqref{eq:49} and \eqref{eq:50} we obtain
\begin{equation*}
|\mathfrak m_K(\eta)|\lesssim \big(\kappa(K)\|\eta\|\big)^{-1}+ \frac{r^{2/3}}{\kappa(K)^{1/3}}.
\end{equation*}
\textbf{Case 2)} Assume that
$\max\{\|\eta_1\|,\ldots,\|\eta_r\|\} \le r^{-1/3} \kappa(K)^{-1/3}$.  Observe that
by Lemma \ref{low dim shift comp} and \eqref{eq:5.12} we have
\begin{align*}
  \bigg| \frac{1}{|K\cap\Z^r|}-
  \frac{1}{|K|}\bigg|\le\frac{er}{2 \kappa(K)|K\cap\Z^r|}
  \le\frac{er}{\kappa(K)|K|}.
\end{align*}
Fix a $\widetilde{ \eta} \in [-1/2,1/2]^r$ such that 
$\widetilde{\eta}_i \equiv \eta \pmod{1}$, then $|\widetilde{\eta}|_2 \approx \| \eta \|$. By abuse of notation we will denote $\widetilde{\eta}$ by $\eta$, then
\begin{equation}
\label{eq:42}
  \begin{split}
  |\mathfrak m_K(\eta)|&\le\Big|\mathfrak
  m_K(\eta)-m_K(\eta)\Big|+\big|m_K(\eta)\big|\\
  &\le\frac{er}{\kappa(K)}+\frac{1}{|K\cap\Z^r|}\Big|\sum_{x\in
  K\cap\Z^r}e^{2\pi i
  x\cdot\eta}-\int_{K}e^{2\pi i
  y\cdot\eta}d y\Big|+\big|m_K(\eta)\big|.
  \end{split}
\end{equation}
Let $Q^{(r)}=[-1/2, 1/2]^r$ and note that by Lemma \ref{low dim shift sym diff}, with
$z=t\in Q^{(r)}$ and \eqref{eq:5.12} we get
\begin{align}
  \label{eq:51}
  \begin{split}
  &\frac{1}{|K\cap\Z^r|}\Big|\sum_{x\in
  K\cap\Z^r}e^{2\pi i
  x\cdot\eta}-\int_{K}e^{2\pi i
  y\cdot\eta}d y\Big|=\frac{1}{|K\cap\Z^r|}\Big|\sum_{x\in
  \Z^r}\int_{Q^{(r)}}e^{2\pi i
  x\cdot\eta}\mathds{1}_{K}(x)-e^{2\pi i
  (x+t)\cdot\eta}\mathds{1}_{K}(x+t)d t\Big|\\
&\le\frac{1}{|K\cap\Z^r|}\int_{Q^{(r)}}\big|(K\cap\Z^r)\triangle\big((t+K)\cap\Z^r\big)\big|d
t+
  \frac{1}{|K\cap\Z^r|}\sum_{x\in
  \Z^r}\mathds{1}_{K}(x)\int_{Q^{(r)}}|e^{2\pi i
  x\cdot\eta}-e^{2\pi i
  (x+t)\cdot\eta}|d t\\
& \lesssim \frac{r}{\kappa(K)}+\big(\|\eta_1\|+\ldots+\|\eta_r\|\big) \lesssim  \frac{r}{\kappa(K)}+ \frac{r^{2/3}}{\kappa(K)^{1/3}} \lesssim \frac{r^{2/3}}{\kappa(K)^{1/3}}.
  \end{split}
\end{align}
Finally, by Lemma \ref{lem: sup sym cont m bound}  we obtain
\begin{align}
  \label{eq:53}
  \begin{split}
  \big| m_K(\eta)\big|\le  C\big(\kappa(K)\|\eta\|\big)^{-1}.
  \end{split}
\end{align}
Combining \eqref{eq:42} with \eqref{eq:51} and \eqref{eq:53} we obtain
 that
\[
|\mathfrak m_K(\eta)|\lesssim \big(\kappa(K)\|\eta\|\big)^{-1}+ \frac{r^{2/3}}{\kappa(K)^{1/3}},
\]
which completes the proof.
\end{proof}
Next lemma is a simple consequence from Lemmas \ref{m bound by lower dim m}, \ref{m_K bound}.
\begin{Lem}
\label{m bound from m_K bound}
For $d\in\N, G \in \mathcal{F}_d, r \in [d]$, if $\kappa(G)\ge \max(50^2r^2, 2 \cdot 50^2r)$, then for every $1\le r\le d$ and $\xi=(\xi_1,\ldots,\xi_d)\in\T^d$ we have
\begin{align*}
  |\mathfrak m_G(\xi)|\lesssim\big(\kappa(G)\|\eta\|\big)^{-1}+ \frac{r^{2/3}}{\kappa(G)^{1/3}} + e^{-r/100},
\end{align*}
 where $\eta=(\xi_1,\ldots, \xi_r)$.
\end{Lem}

\subsection{Conclusion} We have prepared all necessary tools to
prove Proposition \ref{large scale m bound}.
\begin{proof}[Proof of Proposition \ref{large scale m bound}]
If
$\kappa(G)< 50^{14/3}$ then the conclusion trivially
holds. Therefore, we can assume that $G\in\mathcal{F}_d$ satisfies
$\kappa(G)\ge50^{14/3}$.  We let 
$r= \lfloor \kappa(G)^{2/7} \rfloor$.
Without loss of generality we will also assume that $\|\xi_1\|\ge\ldots\ge\|\xi_d\|$ and we shall
distinguish two cases.  Suppose first that
\begin{align*}
  \|\xi_1\|^2+\ldots+\|\xi_r\|^2\ge\frac{1}{4}\|\xi\|^2,
\end{align*}
then in view of Lemma \ref{m bound from m_K bound} we obtain 
\begin{align*}
  |\mathfrak m_G(\xi)|\lesssim \kappa(G)^{-\frac{1}{7}}+\big(\kappa(G)\|\xi\|\big)^{-1},
\end{align*}
and we are done. So we can assume that
\begin{align}
  \label{eq:54}
  \|\xi_1\|^2+\ldots+\|\xi_r\|^2\le\frac{1}{4}\|\xi\|^2.  
\end{align}
Assume first that
\begin{align}
  \label{eq:55}
  \|\xi_j\|\le\frac{1}{100\kappa(G)}\quad\text{ for all } \quad r\le
  j\le d.
\end{align}
Using invariance under sign changes of set $G$ and the Cauchy--Schwarz inequality we have
\begin{align}
  \label{eq:56}
  \begin{split}
  |\mathfrak m_G(\xi)|^2&\le \se{x\in
                          G\cap\Z^d}\prod_{j=1}^d \cos^2(2\pi x_j \xi_j)\\
  &\le \se{x\in
    G\cap\Z^d}\prod_{j=1}^d (1-\sin^2(2\pi x_j \xi_j))\\
  &\le \se{x\in
  G\cap\Z^d}\exp\Big(-\sum_{j=r+1}^d\sin^2(2\pi x_j \xi_j)\Big).
  \end{split}
\end{align}
For $x\in G\cap\Z^d$ we define
\begin{align*}
  I_x&=\{i\in [d]\colon \kappa(G)/10\le |x_i|\le 20\kappa(G) \},\\
  I_x'&=\{i\in [d]\colon 20\kappa(G)< |x_i| \},\\
    I_x''&=\{i\in [d]\colon \kappa(G)/10\le |x_i|\}=I_x\cup I_x',
\end{align*}
 and
\[
E=\big\{x\in G\cap\Z^d\colon |I_x|\ge d/20\big\}.
\]
Observe that
\begin{align*}
  E^{\bf c}=&\big\{x\in G\cap\Z^d\colon |I_x|< d/20\big\}
  =\big\{x\in G\cap\Z^d\colon |I_x''|<d/20+|I_x'|\big\}\\
   \subseteq&\big\{x\in G\cap\Z^d\colon |I_x''|<d/20+|I_x'|\text{
     and } |I_x'|\le  d/20\big\}
  \cup
  \big\{x\in G\cap\Z^d\colon |I_x'|> d/20\big\}.
\end{align*}
Then it is not difficult to see that
\[
E^{\bf c}\subseteq \big\{x\in G\cap\Z^d\colon |I_x''|< d/10\big\},
\]
since by definition of $\kappa(G)$ one sees that $\big\{x\in G\cap\Z^d\colon |I_x'|>  d/20\big\}=\emptyset$. By Lemma \ref{pos_prop_conc} we obtain
\begin{align*}
  |E^{\bf c}|\le |\big\{x\in G\cap\Z^d\colon |I_x''|< d/10\big\}|\le 2e^{-\frac{d}{10}}|G\cap\Z^d|.
\end{align*}
Therefore, by \eqref{eq:56} we have
\begin{align}
\label{eq:57}
  \begin{split}
  |\mathfrak m_G(\xi)|^2
  &\le \se{x\in
  G\cap\Z^d}\exp\Big(-\sum_{j\in I_x\cap J_r}\sin^2(2\pi x_j
\xi_j)\Big)\mathds{1}_{E}(x)+2e^{-\frac{d}{10}},
\end{split}
\end{align}
where $J_r=\{r+1,\ldots, d\}$. Using inequality $|\sin(\pi t)| \ge 2|t|$ (for $|t| \le 1/2$) and definition of $I_x$ we have
\[
\sin^2(2\pi x_j \xi_j)\ge 16|x_j|^2\|\xi_j\|^2\ge \frac{1}{100}\kappa(G)^2\|\xi_j\|^2,
\]
since $2|x_j|\|\xi_j\|\le 1/2$ by \eqref{eq:55}, and consequently we estimate \eqref{eq:57} and obtain for some $C, c>0$ that
\begin{align}
  \label{eq:58}
  \begin{split}
  &\se{x\in
  G\cap\Z^d}\exp\Big(-\sum_{j\in I_x\cap J_r}\sin^2(2\pi x_j
  \xi_j)\Big)\mathds{1}_{E}(x)\\
  &\le
  \frac{1}{|G\cap\Z^d|}\sum_{x\in
  G\cap\Z^d\cap E}\exp\Big(-\frac{1}{100}\kappa(G)^2\sum_{j\in I_x\cap
  J_r}\|\xi_j\|^2\Big)\le Ce^{-c\kappa(G)^2\|\xi\|^2}.
  \end{split}
\end{align}
In order to obtain the last inequality in \eqref{eq:58} observe that 
\begin{align*}
\frac{1}{|G\cap\Z^d|}&\sum_{x\in G\cap\Z^d\cap E}
\exp\Big(-\frac{1}{100}\kappa(G)^2\sum_{j\in I_x\cap J_r}\|\xi_j\|^2\Big)\\
&=\se{\sigma \in \Sym(d)}\bigg[\frac{1}{|G\cap\Z^d|}\sum_{\sigma^{-1} \cdot x\in
  G\cap\Z^d\cap E}\exp\Big(-\frac{1}{100}\kappa(G)^2\sum_{j\in I_{\sigma^{-1} \cdot x}\cap
  J_r}\|\xi_j\|^2\Big)\bigg]\\
  &=\frac{1}{|G\cap\Z^d|}\sum_{x\in
  G\cap\Z^d\cap E}\se{\sigma \in \Sym(d)}\bigg[\exp\Big(-\frac{1}{100}\kappa(G)^2\sum_{j\in \sigma(I_x)\cap
  J_r}\|\xi_j\|^2\Big)\bigg],
\end{align*}
since $\sigma\cdot (G\cap\Z^d\cap E)=G\cap\Z^d\cap E$ for every $\sigma\in{\rm Sym}(d)$. 
Appealing now to  Lemma \ref{lem:2.6} with $\delta_1=1/20$, $d_0=r$, $I=I_x$ and $\delta_0=3/5$, we conclude that
\begin{align*}
\se{\sigma \in \Sym(d)}\bigg[\exp\Big(-\frac{1}{100}\kappa(G)^2\sum_{j\in \sigma(I_x)\cap
  J_r}\|\xi_j\|^2\Big)\bigg]\le C\exp\Big(-c'\kappa(G)^2\sum_{j=r+1}^d\|\xi_j\|^2\Big),
  \end{align*}
  for some $c'>0$ and for all $x\in G\cap\Z^d\cap E$.
  This proves \eqref{eq:58} since by
\eqref{eq:54} we obtain
\begin{align*}
\exp\Big(-c'\kappa(G)^2\sum_{j=r+1}^d\|\xi_j\|^2\Big)\le \exp\Big(-\frac{c'\kappa(G)^2}{4}\sum_{j=1}^d\|\xi_j\|^2\Big).
\end{align*}
Therefore
\begin{align*}
&|\mathfrak{m}_G(\xi)|^2 \lesssim \exp\Big(-\frac{c'\kappa(G)^2}{4}\sum_{j=1}^d\|\xi_j\|^2\Big)+ e^{-d/10} \\
&\lesssim \Big( \kappa(G) \|\xi \| \Big)^{-2} + \kappa(G)^{-2/7} \lesssim \bigg( \Big( \kappa(G) \|\xi \| \Big)^{-1} + \kappa(G)^{-1/7} \bigg)^{2},
\end{align*}
above we have used $\kappa(G) \le d$.
Assume now that \eqref{eq:55} does not hold. Then
\begin{align}
\label{eq:44}
  \|\xi_j\|\ge\frac{1}{100\kappa(G)}\quad\text{ for all
  } \quad 1\le j\le r.
\end{align}
Hence \eqref{eq:44} gives that
\begin{align*}
  \|\xi_1\|^2+\ldots+\|\xi_r\|^2\ge\frac{r}{10^4\kappa(G)^2}.
\end{align*}
Therefore, we invoke Lemma \ref{m bound from m_K bound} with $\eta=(\xi_1,\ldots, \xi_r)$ again and obtain 
\begin{align*}
  |\mathfrak {m}_G(\xi)|&\lesssim\kappa(G)^{-\frac{1}{7}}+(\kappa(G)\|\eta\|)^{-1} \lesssim \kappa(G)^{-\frac{1}{7}},
\end{align*}
since $r\simeq\kappa(G)^{\frac{2}{7}}$. This completes the proof of Proposition \ref{large scale m bound}.
\end{proof}
\section{Dyadic maximal function in small-scale regime}
In this section we will finish the proof of Theorem \ref{dyadic dim-free} by considering dyadic maximal function in the small-scale regime.
We start with the following key proposition.
\begin{Prop} \label{prop:3.6}
Let $d \in \N$, $G \in \mathcal{F}_d$ satisfy $\widetilde{\kappa}(G) \le 10^{-3}$. Then for every $\xi \in \T^d$ the following inequalities hold:
    \begin{itemize}
        \item \begin{equation} \label{eq:3.10}
        |\mathfrak{m}_G(\xi)| \lesssim e^{- c\widetilde{\kappa}(G) \min\big(\|\xi\|^2, \| \xi +1/2\|^2 \big)} + 2^{-d\widetilde{\kappa}(G)/2},
    \end{equation}
    \item \begin{equation} \label{eq:3.11}
    |\mathfrak{m}_G(\xi)-1| \lesssim \widetilde{\kappa}(G) \| \xi\|^2,
    \end{equation}
    \item 
    \begin{equation} \label{eq:3.12}
    \bigg|\mathfrak{m}_G(\xi)-\se{x \in G \cap \Z^d} (-1)^{\sum_{i=1}^d x_i}\bigg| \lesssim  \widetilde{\kappa}(G) \|\xi+1/2\|^2,
    \end{equation}
    \end{itemize}
    where $c>0$ is an absolute constant.
\end{Prop}

In order to prove Proposition \ref{prop:3.6} we need to introduce crucial multiplier corresponding to certain combinatorial averaging operator. Generalization of this multipliers symbol will be of importance in the next section as well and was heavily studied in \cite{NW}. Ideas utilized here were already present in \cite[Proposition 3.3]{balls} and were later refined in \cite[Section 3]{Ni2}.
\begin{Defn}
    For every $n,d \in \N$, $J\subseteq [d]$ with $n \leq |J|$ we define $\beta_n^J: \T^d \to \mathbb{C}$ by the formula
    \[
    \beta_{n}^J(\xi)=  \se{x \in D_n^J}e( x \cdot \xi),
    \]
    where
    \[
    D_n^J=\Big\{ y \in \{-1,0,1\}^d: |\{i \in [d]: |y_i|=1\}|=n, \supp(y) \subseteq J \Big\}.
    \]
\end{Defn}
We have the following crucial bound, which follows from bounds for Krawtchouk polynomials.
\begin{Prop}[{\cite[Proposition 3.4]{Ni2}}] \label{prop:3.4}
 For every $n,d \in \N$, $J\subseteq [d]$ with $n \leq |J|/2$ and $\xi \in \mathbb{T}^d$ we have
     \[ \tag{2} 
     |\beta_n^J(\xi)| \leq 2 e^{- \frac{cn}{2|J|} \min\big( \sum_{i \in J} \sin^2(\pi \xi_i), \sum_{i \in J} \cos^2(\pi \xi_i)  \big)},
     \]
     where $c \in (0,1)$ is an absolute constant.
\end{Prop}
Let us derive Proposition \ref{prop:3.6} as a consequence of previous bound for multipliers $\beta^J_n$'s.
\begin{proof}[Proof of Proposition \ref{prop:3.6}]
Note that \eqref{eq:3.11} is a direct consequence of 
Proposition \ref{large scale m-1 bound}. For \eqref{eq:3.12} we can simply apply the same arguments. Let $\eta=\xi+1/2 \in \T^d$, then we have
\begin{align*}
    &\bigg|\mathfrak{m}_G(\xi)-\se{x \in G \cap \Z^d} (-1)^{\sum_{i=1}^d x_i}\bigg| = \bigg| \se{x \in G \cap \Z^d}\Big( e(x \cdot \xi)- (-1)^{\sum_{i=1}^d x_i} \Big) \bigg| \\
    &=\bigg| \se{x \in G \cap \Z^d}\Big( \prod_{i=1}^d \cos(2 \pi x_i \xi_i)- (-1)^{\sum_{i=1}^d x_i} \Big) \bigg|=\bigg| \se{x \in G \cap \Z^d}(-1)^{\sum_{i=1}^d x_i}\Big( \prod_{i=1}^d \cos(2 \pi x_i \eta_i)- 1 \Big)  \bigg| \\
    &\le \se{x \in G \cap \Z^d} \Big| \prod_{i=1}^d \cos(2 \pi x_i \eta_i)- 1  \Big|.
\end{align*}
To the final expression we can just repeat the rest of the argument from the proof of Proposition \ref{large scale m-1 bound} and get
\[
\bigg|\mathfrak{m}_G(\xi)-\se{x \in G \cap \Z^d} (-1)^{\sum_{i=1}^d x_i}\bigg| \lesssim \widetilde{\kappa}(G) \| \eta \|^2 = \widetilde{\kappa}(G) \| \xi+1/2 \|^2.
\]
It remains to prove \eqref{eq:3.10}. Let $n=d \widetilde{\kappa}(G) \in \N$ and 
\[
E= \Big\{ x \in G \cap \Z^d: |\{ i \in [d]: |x_i|=1 \}| >n/2 \Big\}
\]
and
\[
\mathfrak{s}_{G}(\xi)= \frac{1}{|G \cap \Z^d |} \sum_{x \in E} e(x \cdot \xi).
\]
By Lemma \ref{BMSW a lot of 1} (with $k=\lfloor n/2 \rfloor$) for every $\xi \in \T^d$ we have
\begin{equation} \label{eq: comp with sG}
    |\mathfrak{m}_G(\xi)- \mathfrak{s}_{G}(\xi)| \le \frac{|G \cap \Z^d \setminus E|}{| G \cap \Z^d|} \lesssim 2^{-n/2},
\end{equation}
so it suffices to show that for
\begin{equation} \label{eq: sG bound}
    |\mathfrak{s}_{G}(\xi)| \lesssim e^{- \frac{c\widetilde{\kappa}(G)}{320} \min\big(\|\xi\|^2, \| \xi +1/2\|^2 \big)}.
\end{equation}
We decompose any $x \in \Z^d$ into $x=y(x)+z(x)$, where
    \[
z(x)_j= \begin{cases}
    x_j, \ \ \text{if $|x_j| \geq 2$} \\
    0, \ \ \text{otherwise}
\end{cases}
\]
and denote $Z(E)=\{z(x): x \in E \}$.
For $z \in Z(E)$ we define 
\begin{align*}
    J(z)= &\Big\{ k \in [d] \cup \{0 \}: (\exists y \in \{-1,0,1\}^d) \Big( \supp(y) \cap \supp(z)= \emptyset,  \\
    &y+z \in G,|\{i \in [d]: |y_i|=1 \}|=k \Big)\Big\}.
\end{align*}
Note that since $G$ is 1-symmetric, if $k \in J(z)$, then for all $y \in \{-1,0,1\}^d$ satisfying
\[
|\{i \in [d]: |y_i|=1 \}|=k, \quad \supp(y) \cap \supp(z)= \emptyset
\]
we have $y+z \in G$.
Moreover we have the following disjoint decomposition of the set $E$:
\[
E= \bigcup_{z \in Z(E)} \bigcup_{\substack{n/2<k \leq n-|z|_1, \\ k \in J(z)}} \Big(z+ \Big\{y \in \{-1,0,1\}^d: \supp(y) \cap \supp(z)= \emptyset, |\{i \in [d]: |y_i|=1 \}|=k\Big\}\Big).
\]
Using the above we get that
\begin{equation} \label{eq:3.9}
 \begin{split} 
    \mathfrak{s}_G(\xi)&= \frac{1}{|G\cap \Z^d|} \sum_{x \in E} e(x \cdot \xi)= \frac{1}{|G\cap \Z^d|}\sum_{z \in Z(E)} e(z \cdot \xi) \sum_{\substack{n/2<k \leq n-|z|_1, \\ k \in J(z)}} \sum_{\substack{y \in \{-1,0,1\}^d, \\ |\{i \in [d]: |y_i|=1\}|=k, \\
    \supp(y) \cap \supp(z)= \emptyset}} e(y \cdot \xi)
    \\
    &=\frac{1}{|G\cap \Z^d|}\sum_{z \in Z(E)} \prod_{j=1}^d \cos(2 \pi z_j \xi_j) \sum_{\substack{n/2<k \leq n-|z|_1, \\ k \in J(z)}} \sum_{\substack{y \in \{-1,0,1\}^d, \\ |\{i \in [d]: |y_i|=1\}|=k, \\
    \supp(y) \cap \supp(z)= \emptyset}} e(y \cdot \xi),
 \end{split}   
\end{equation}
where the last equality follows from invariance of the set $Z(E)$ under sign changes. It is helpful to note that $n =d \widetilde{\kappa}(G) \le d/2$ and $|z |_1 \ge |\supp(z)| \ge \frac{|\supp(z)|}{2}$, so that $n- |z |_1\le \frac{d- |\supp(z)|}{2}$ is satisfied.
By Proposition \ref{prop:3.4} for any $z \in Z(E)$, $k \in \N$, such that $ n/2 < k \le n -|z|_1 \le \frac{d-|\supp(z)|}{2}  $, let $J= [d] \setminus \supp(z)$, then we have
\begin{align*}
&\bigg|\sum_{\substack{y \in \{-1,0,1\}^d, \\ |\{i \in [d]: |y_i|=1\}|=k, \\
    \supp(y) \cap \supp(z)= \emptyset}} e(y \cdot \xi) \bigg|= |D_k^{J}| |\beta_k^{ J}(\xi)| \leq 2|D_k^{J}| e^{-\frac{ck}{2|J|} \min\big( \sum_{j \in J} \sin^2(\pi \xi_j), \sum_{j \in J} \cos^2(\pi \xi_j) \big)}  \\
& \le  2 e^{-\frac{cn}{4d} \min\big( \sum_{j \in [d] \setminus \supp(z)} \sin^2(\pi \xi_j), \sum_{j \in [d] \setminus \supp(z)} \cos^2(\pi \xi_j) \big)} \cdot 2^k \binom{d-|\supp(z)|}{k}.
\end{align*}
Plugging the above into \eqref{eq:3.9} we get
\begin{align*}
    |\mathfrak{s}_G(\xi)| &\le \frac{2}{|G\cap \Z^d|}\sum_{z \in Z(E)} \sum_{\substack{n/2<k \leq n-|z|_1, \\ k \in J(z)}} 2^k \binom{d-|\supp(z)|}{k} \exp\Big(-\frac{cn}{4d} \sum_{j \in [d] \setminus \supp(z)} \sin^2(\pi \xi_j) \Big) \\
    &+ \frac{2}{|G\cap \Z^d|}\sum_{z \in Z(E)}  \sum_{\substack{n/2<k \leq n-|z|_1, \\ k \in J(z)}} 2^k \binom{d-|\supp(z)|}{k} \exp\Big(-\frac{cn}{4d} \sum_{j \in [d] \setminus \supp(z)} \cos^2(\pi \xi_j) \Big).
\end{align*}
We will bound only the first term above, for the second summand treatment is analogous. Note that by symmetry invariance of the set $Z(E)$ and Lemma \ref{lem:2.6} (with $\delta_0=\delta_1=1/2$) one has
\begin{align*}
    &\frac{2}{|G\cap \Z^d|}\sum_{z \in Z(E)} \sum_{\substack{n/2<k \leq n-|z|_1, \\ k \in J(z)}} 2^k \binom{d-|\supp(z)|}{k} \exp\Big(-\frac{cn}{4d} \sum_{j \in [d] \setminus \supp(z)} \sin^2(\pi \xi_j) \Big) \\
    &=\frac{2}{|G\cap \Z^d|}\sum_{z \in Z(E)} \sum_{\substack{n/2<k \leq n-|z|_1, \\ k \in J(z)}} 2^k \binom{d-|\supp(z)|}{k} \se{\sigma \in \Sym(d)} \exp\Big(-\frac{cn}{4d} \sum_{j \in \sigma\big( [d] \setminus \supp(z)\big)} \sin^2(\pi \xi_j) \Big) \\
    &\le \frac{8}{|G\cap \Z^d|}\sum_{z \in Z(E)} \sum_{\substack{n/2<k \leq n-|z|_1, \\ k \in J(z)}} 2^k \binom{d-|\supp(z)|}{k}  e^{- \frac{cn}{320d} \| \xi \|^2}= 8 \mathfrak{s}_G(0)e^{- \frac{cn}{320d} \| \xi \|^2} \le 8 e^{- \frac{cn}{320d} \| \xi \|^2}.
\end{align*}
Similarly one can show that 
\[
\frac{2}{|G\cap \Z^d|}\sum_{z \in Z(E)}  \sum_{\substack{n/2<k \leq n-|z|_1, \\ k \in J(z)}} 2^k \binom{d-|\supp(z)|}{k} \exp\Big(-\frac{cn}{4d} \sum_{j \in [d] \setminus \supp(z)} \cos^2(\pi \xi_j) \Big) \le 8 e^{- \frac{cn}{320d} \| \xi +1/2\|^2}.
\]
Since $\frac{n}{d}=\widetilde{\kappa}(G)$, the above combined with \eqref{eq: comp with sG} finishes the proof of \eqref{eq:3.10} after replacing $c/320$ by $c$. 
\end{proof}

Let us introduce new auxiliary multipliers, which approximate $\mathfrak{m}_G$ well enough.


\begin{Defn}
    For any $d \in \N, G\in \mathcal{F}_d$ we define $\lambda_{G}^1,\lambda_{G}^2: \T^d \to \mathbb{C}$ by following formulas:
    \[
\lambda_{G}^1(\xi)=e^ {- \widetilde{\kappa}(G) \| \xi \|^2 },
\]
\[
\lambda_{G}^2= \Big(\se{x \in G \cap \Z^d} (-1)^{\sum_{i=1}^d x_i} \Big) e^ {- \widetilde{\kappa}(G)  \| \xi +1/2\|^2 }. 
\]
\end{Defn}
The following theorem is a consequence of the theory of symmetric diffusion semigroups, see  \cite[p. 73]{Ste1} and \cite[Theorem 2.11]{NW}.
\begin{Thm} \label{thm:4.2}
    There exists a constant $C>0$ such that for all $d \in \N$, $G \in \mathcal{F}_d$  $f \in \ell^2(\Z^d)$ we have
    \[ 
    \Big\| \sup_{t>0} \big|\mathcal{F}^{-1}(\lambda_{tG}^1 \widehat{f})\big|\Big\|_{\ell^2(\Z^d)} \leq C \|f \|_{\ell^2(\Z^d)}.
    \]
\end{Thm}
Notice that $\lambda_{G}^2(\xi)= \Big(\mathbb E_{x \in G \cap \Z^d} (-1)^{\sum_{i=1}^d x_i} \Big)\lambda_{G}^1(\xi+\frac{1}{2})$, where $\xi+\frac{1}{2}=(\xi_1 + \frac{1}{2},..., \xi_d+ \frac{1}{2})$. Because of this we have
\[
\sup_{\|f \|_{\ell^2(\Z ^d)}=1} \Big\| \sup_{t>0} \big|\mathcal{F}^{-1}(\lambda_{tG}^2 \widehat{f})\big|\Big \|_{\ell^2(\Z^d)} \le \sup_{\|f \|_{\ell^2(\Z ^d)}=1} \Big\| \sup_{t>0} \big|\mathcal{F}^{-1}(\lambda_{tG}^1 \widehat{f})\big|\Big \|_{\ell^2(\Z^d)}.
\]
Relation between multipliers $\lambda_{G}^1, \lambda_{G}^2$ and $\mathfrak{m}_{G}$ is stated below.
\begin{Prop} \label{prop:4.3}
Let $d \in \N, G \in \mathcal{F}_d$ be such that $\widetilde{\kappa}(G) \le 10^{-3}$. Then for every $\xi \in \T^d$ we have \\
\begin{enumerate}
    \item If $\|\xi \| \leq \|\xi+1/2 \|$, then
    \begin{equation*}
    |\mathfrak{m}_G(\xi)-\lambda_G^1(\xi)| \lesssim \min \bigg( \widetilde{\kappa}(G) \|\xi \|^2,  \Big(\widetilde{\kappa}(G)\|\xi \|^2 \Big)^{-1} \bigg)+ 2^{-d\widetilde{\kappa}(G)/2}.
    \end{equation*}

    \item If $\|\xi +1/2\| < \|\xi \|$, then
    \begin{equation*}
    |\mathfrak{m}_G(\xi)-\lambda_G^2(\xi)| \lesssim \min \bigg( \widetilde{\kappa}(G) \|\xi +1/2\|^2,  \Big(\widetilde{\kappa}(G)\|\xi +1/2\|^2 \Big)^{-1} \bigg)+ 2^{-d\widetilde{\kappa}(G)/2}.
    \end{equation*}
\end{enumerate}
\end{Prop}
Proof of the proposition above is a simple consequence of Proposition \ref{prop:3.6} together with elementary inequalities $1-e^{-x} \le x$, $e^{-x} \le x^{-1}$ for $x>0$. We can finally prove
Theorem \ref{dyadic dim-free}. 
\begin{proof}[Proof of Theorem \ref{dyadic dim-free}.]
Take $d \in \N, G \in \mathcal{F}_d$, $f \in \ell^2(\Z^d)$, due to Corollary \ref{large scale no upper restr} and inequality the $\widetilde{\kappa}(2^nG) \le \kappa(2^nG)$ it suffices to show that 
    \begin{equation} \label{eq:small-scale dyad dimfree}
        \Big\| \sup_{\substack{n \in \Z, \\ \widetilde{\kappa}(2^nG) \le 1}} \big| \mathcal{M}_{2^n}^Gf \big| \Big\|_{\ell^2(\Z^d) } \lesssim \|f \|_{\ell^2(\Z^d)}.
    \end{equation}
    Note that if $\widetilde{\kappa}(2^nG)< \frac{1}{d}, $ then $2^n G \cap \Z^d=\{ \overrightarrow{0}\}$, moreover the sequence $\big(\widetilde{\kappa}(2^nG)\big)_{n \in \Z}$ is lacunary, hence 
    \begin{align*}
        \Big\| \sup_{\substack{n \in \Z, \\ \widetilde{\kappa}(2^nG) \le 1}} \big|\mathcal{M}_{2^n}^Gf \big| \Big\|_{\ell^2(\Z^d) } &\lesssim \Big\| \sup_{\substack{n \in \Z, \\ \frac{1}{d} \le \widetilde{\kappa}(2^nG) \le 1}} \big|\mathcal{M}_{2^n}^Gf \big|\Big\|_{\ell^2(\Z^d) } + \|f \|_{\ell^2(\Z^d)} \\
    &\lesssim \Big\| \sup_{\substack{n \in \Z, \\ \frac{1}{d} \le \widetilde{\kappa}(2^nG) \le 10^{-3}}} \big|\mathcal{M}_{2^n}^Gf \big| \Big\|_{\ell^2(\Z^d) } + \|f \|_{\ell^2(\Z^d)}.
    \end{align*}

Now we write $f=f_1+f_2$, where
\[
\supp(\widehat{f_1}) \subseteq \{ \xi \in \T^d : \|\xi\| \leq \|\xi +1/2\|) \},
\quad
\supp(\widehat{f_2}) \subseteq \{ \xi \in \T^d : \|\xi\| > \|\xi +1/2\| \}.
\]
Notice that in order to prove \eqref{eq:small-scale dyad dimfree} it suffices to show that for $i \in \{1,2\}$ we have 
\begin{equation} \label{eq: dyadic dim-free for fi}
    \bigg\| \sup_{\substack{n \in \Z, \\ \frac{1}{d} \le \widetilde{\kappa}(2^nG) \le 10^{-3}}} \big|\mathcal{M}_{2^n}^Gf_i \big| \bigg\|_{\ell^2(\Z^d) } \lesssim \|f_i \|_{\ell^2(\Z^d)}.
\end{equation}
We shall prove \eqref{eq: dyadic dim-free for fi} only for $i=1$, proof for $i=2$ is exactly the same but with $f_1$ and $\lambda_G^1$ replaced by $f_2$ and $\lambda_G^2$ respectively and $\|\xi\| $ replaced by $\| \xi+1/2\|$.
By triangle inequality and Theorem \ref{thm:4.2} we have
\begin{align*}
    &\bigg\| \sup_{\substack{n \in \Z, \\ \frac{1}{d} \le \widetilde{\kappa}(2^nG) \le 10^{-3}}} \big|\mathcal{M}_{2^n}^Gf_1 \big| \bigg\|_{\ell^2(\Z^d) } \le
    \bigg\| \sup_{\substack{n \in \Z, \\ \frac{1}{d} \le \widetilde{\kappa}(2^nG) \le 10^{-3}}} \Big|\mathcal{F}^{-1}\Big(\big(\mathfrak{m}_{2^nG}-\lambda_{2^nG}^1\big)\widehat{f_1} \Big) \Big| \bigg\|_{\ell^2(\Z^d) } \\
    &+ \bigg\| \sup_{t>0} \big|\mathcal{F}^{-1}(\lambda_{tG}^1 \widehat{f_1})\big|\bigg\|_{\ell^2(\Z^d)} \lesssim 
    \bigg\| \sup_{\substack{n \in \Z, \\ \frac{1}{d} \le \widetilde{\kappa}(2^nG) \le 10^{-3}}} \Big|\mathcal{F}^{-1}\Big(\big(\mathfrak{m}_{2^nG}-\lambda_{2^nG}^1\big)\widehat{f_1} \Big) \Big| \bigg\|_{\ell^2(\Z^d) } + \|f_1 \|_{\ell^2(\Z^d)}.
\end{align*}
In order to bound the first term we use a standard square function argument and Parseval's theorem combined with Proposition \ref{prop:4.3} and assumption on $\supp(\widehat{f_i})$. This way we obtain
\begin{align*}
     &\bigg\| \sup_{\substack{n \in \Z, \\ \frac{1}{d} \le \widetilde{\kappa}(2^nG) \le 10^{-3}}} \Big|\mathcal{F}^{-1}\Big(\big(\mathfrak{m}_{2^nG}-\lambda_{2^nG}^1\big)\widehat{f_1} \Big) \Big| \bigg\|_{\ell^2(\Z^d) } \lesssim 
     \Bigg\| \bigg( \sum_{\substack{n \in \Z, \\ \frac{1}{d} \le \widetilde{\kappa}(2^nG) \le 10^{-3}}} \Big|\mathcal{F}^{-1}\Big(\big(\mathfrak{m}_{2^nG}-\lambda_{2^nG}^1\big)\widehat{f_1} \Big) \Big|^2 \bigg)^{1/2} \Bigg\|_{\ell^2(\Z^d) } \\
     &= \Bigg( \sum_{\substack{n \in \Z, \\ \frac{1}{d} \le \widetilde{\kappa}(2^nG) \le 10^{-3}}} \bigg\|   \mathcal{F}^{-1}\Big(\big(\mathfrak{m}_{2^nG}-\lambda_{2^nG}^1\big)\widehat{f_1} \Big) \bigg\|_{\ell^2(\Z^d) }^2 \Bigg)^{1/2}= \bigg( \int_{\T^d} |\widehat{f_1}|^2 \sum_{\substack{n \in \Z, \\ \frac{1}{d} \le \widetilde{\kappa}(2^nG) \le 10^{-3}}} |\mathfrak{m}_{2^nG}(\xi)- \lambda_{2^nG}^1(\xi)|^2 d\xi \bigg)^{1/2} \\
     &\lesssim  \bigg( \int_{\T^d} |\widehat{f_1}|^2 \bigg(\sum_{\substack{n \in \Z, \\ \frac{1}{d} \le \widetilde{\kappa}(2^nG) \le 10^{-3}}} \Big(\min\Big( \widetilde{\kappa}(2^nG) \| \xi \|^2, (\widetilde{\kappa}(2^nG) \| \xi \|^2)^{-1} \Big)^2+2^{-d \widetilde{\kappa}(2^nG)}
     \Big)\bigg)
     d\xi \bigg)^{1/2} \\
     &\lesssim \|\widehat{f_1} \|_{L^2(\T^d)}= \|f_1 \|_{\ell^2(\Z^d)}.
\end{align*}
In the last inequality we've used the fact that $\widetilde{\kappa}(2K) \ge 2 \widetilde{\kappa}(K)$ for every $K \in \mathcal{F}_d$. This finishes the proof of \eqref{eq: dyadic dim-free for fi} and hence the proof of Theorem \ref{dyadic dim-free}.
\end{proof}
\section{Full maximal function in small-scale regime}
In this section we will establish Theorem \ref{small scale dim-free}.
We begin with the following lemma, which allows us to discretize the supremum, so that we may employ a square function argument later.
\begin{Lem} \label{discretized supremum}
Fix $\varepsilon \in (0,1/2)$. For every $d \in \N$ and $G \in \mathcal{F}_d$ there exists a possibly empty nondecreasing sequence of parameters 
\[
0<t_0 \le t_1 \le \ldots \le t_d
\]
such that
\[d^{-1/2} \le \widetilde{\kappa}(t_0 G) \le \widetilde{\kappa}(t_d G) \le d^{-\varepsilon}
\]
and for every $f \in \ell^2(\Z^d)$ we have
\[
\Big\|\sup_{\substack{t>0, \\
 \widetilde{\kappa}(tG) \le d^{-\varepsilon}}} \big|\mathcal{M}_{t}^Gf\big| \Big\|_{\ell^2(\Z^d)} \lesssim \Big\|\sup_{0 \le i \le d} \mathcal{M}_{t_i}^G\big|f\big| \Big\|_{\ell^2(\Z^d)} + \|
 f\|_{\ell^2(\Z^d)},
\]
with the convention that the first term on the right-hand side equals $)$ if the sequence is empty.
\end{Lem}
\begin{proof}
    By triangle inequality we have 
    \begin{equation} \label{eq: k(tG) tr ineq}
        \Big\|\sup_{\substack{t>0, \\
 \widetilde{\kappa}(tG) \le d^{-\varepsilon}}} |\mathcal{M}_{t}^Gf| \Big\|_{\ell^2(\Z^d)} \le \Big\|\sup_{\substack{t>0, \\
 \widetilde{\kappa}(tG) \le d^{-1/2}}} |\mathcal{M}_{t}^Gf| \Big\|_{\ell^2(\Z^d)}+\Big\|\sup_{\substack{t>0, \\
 d^{-1/2} \le\widetilde{\kappa}(tG) \le d^{-\varepsilon}}} |\mathcal{M}_{t}^Gf| \Big\|_{\ell^2(\Z^d)}
    \end{equation}
we will treat two terms of the above separately.
Note that for any $t>0$ by \eqref{disc_incl} we have
\[
d \widetilde{\kappa}(tG) B^{1,(d)} \cap \{-1,0,1 \}^d \subseteq tG \cap \Z^d \subseteq d \widetilde{\kappa}(tG) B^{1,(d)} \cap \Z^d,
\]
In \cite[Corollary 1.8]{KNW} it was shown that if $n \le d^{1/2},$ then 
\[
|n B^{1,(d)} \cap \Z^d| \approx \Big|n B^{1,(d)} \cap \{-1,0,1 \}^d \Big|.
\]
This way for any $f:\Z^d \to \C$ we obtain 
\[
\sup_{\substack{t>0, \\
 \widetilde{\kappa}(tG) \le d^{-1/2}}} |\mathcal{M}_{t}^Gf| \lesssim \sup_{\substack{n \le d^{1/2}}} \mathcal{M}_{n}^{B^{1,(d)}}|f|.
\]
Moreover by \cite[Theorem 1.15]{KNW} we see that 
\[
\Big\|\sup_{\substack{t>0, \\
 \widetilde{\kappa}(tG) \le d^{-1/2}}} |\mathcal{M}_{t}^Gf| \Big\|_{\ell^2(\Z^d)}\lesssim \Big\|\sup_{\substack{n \le d^{1/2}}} \mathcal{M}_{n}^{B^{1,(d)}}|f| \Big\|_{\ell^2(\Z^d)} \lesssim \|f \|_{\ell^2(\Z^d)}.
\]
Now we will deal with the second term of the right-hand side of \eqref{eq: k(tG) tr ineq}. For $i \in \N_0$ we define
\[
I_i= \{ t \in \R_{>0}: d^{-1/2} \le \widetilde{\kappa}(tG) \le d^{-\varepsilon}, 12^i \le |tG \cap \Z^d| < 12^{i+1} \Big\}.   
\]
If for every $i \in \N_0$ the set $I_i$ is empty, then there is nothing to prove, so we assume that is not the case. Let 
\[
i_0= \min\{ i \in \N_0 : I_i \neq \emptyset \},
\]
\[
m= \max\{i \in \N_0: I_i \neq \emptyset\}.
\]
Note that for any $t>0$ with $\widetilde{\kappa}(tG) \le d^{-\varepsilon} \le 1$ by Lemma \ref{Cor 2.2 from my lq paper} we have 
\begin{equation} \label{eq: 12d bound}
    |tG \cap \Z^d | \le | d \widetilde{\kappa}(tG)B^{1,(d)} \cap \Z^d | \le 2\Big( \widetilde{\kappa}(tG)+ \frac{1}{2} \Big)^d \cdot 8^d \le 2 \cdot 12^d.
\end{equation}
This way we see that $m \le d$. Define $t_{i_0} \in I_{i_0}$ to be any number satisfying
\[
|t_{i_0} G \cap \Z^d|= \max_{t \in I_{i_0}} |t G \cap \Z^d|.
\]
Next for any $j \in (i_0,m] \cap \N$ we define $t_j$ in the following way: 
\begin{itemize}
    \item If $I_{j}= \emptyset$ let $t_j=t_{j-1}$.
    \item If $I_{j} \neq \emptyset$ let $t_j \in I_j$ be any number satisfying
    \[
    |t_{j} G \cap \Z^d|= \max_{t \in I_j} |t G \cap \Z^d|.
    \]
\end{itemize}
It is not difficult to see that $t_j \ge t_{j-1}$.
We claim that for any $f: \Z^d \to \C$ we have
\begin{equation} \label{eq: discretization of sup}
    \sup_{\substack{t>0, \\
d^{-1/2} \le \widetilde{\kappa}(tG) \le d^{-\varepsilon}}} |\mathcal{M}_{t}^Gf| \le 12 \sup_{i_0 \le i \le m} \mathcal{M}_{t_i}^G|f|.
\end{equation}
Take any $t>0$ with $d^{-1/2} \le \widetilde{\kappa}(tG) \le d^{-\varepsilon}$. Then for some $j \in [i_0,m]$ we have $t \in I_j$, in particular we have
\[
\frac{1}{12}|t_jG \cap \Z^d| \le |tG \cap \Z^d| \le |t_jG \cap \Z^d|,
\]
since the sets $sG \cap \Z^d$ are increasing, the second inequality implies that $tG \cap \Z^d \subseteq t_j G \cap \Z^d$. This way for any $f:\Z^d \to \C$ we obtain
\[
\big|\mathcal{M}_t^G f \big| \le 12 \mathcal{M}_{t_j}^G |f| \le  12 \sup_{i_0 \le i \le m} \mathcal{M}_{t_i}^G|f|.
\]
Since $t>0$ was arbitrary this shows that inequality \eqref{eq: discretization of sup} holds. Finally to get exactly the same conclusion as in Lemma \ref{discretized supremum} we can simply reindex sequence $(t_i)_{i=i_0}^m$ so that it starts from $0$ and extend it by constant sequence so that last index is $d$, using the fact that $m \le d$.  
\end{proof}
From now on we will simply follow the strategy from \cite[Section 4]{NW}. Let us recall results and notation from above cited work.  Fix $K \in \N$ and take $\overline{n}=(n_1,n_2,...,n_K) \in (\N_0)^K.$ Let $D_{\overline{n}}$ denote the set of lattice points in $\{-K,\ldots,K\}^d$ such that exactly $n_1$ coordinates are equal to $\pm 1,$ $n_2$ coordinates are equal to $\pm 2,$ and so on; formally
\begin{equation}
\label{eq:def Dn}
D_{\overline{n}}= \bigcap_{j=1}^K\Big\{ x \in \{-K,...,K\}^d:  \#\{i \in [d]: |x_i|=j\}=n_j \Big\}.
\end{equation}
Consider the operator
\[
\mathcal{D}_{\overline{n}}f(x)= \se{y \in D_{\overline{n}}}f(x-y)
\]
and denote by
\begin{equation}
\label{eq:bon}
  \beta_{\overline{n}}(\xi)=\se{x \in D_{\overline{n}}} e(x \cdot \xi),\qquad \xi\in\T^d
\end{equation}
its corresponding multiplier symbol. Moreover we give generalization of $\beta_{\overline{n}}$, which is also of importance in what is about to follow. For a finite set $J \subseteq \N$, $\overline{n}=(n_1,n_2,...,n_K) \in \N_0^K$ satisfying \\ $n_1+...+n_K \le |J| $ we define
    \begin{equation} \label{generalised beta defn}
    \beta_{\overline{n}}^J(\xi)=\se{\substack{I_1,I_2,...,I_K \subseteq J, \\
    (\forall j \in [K]) |I_j|=n_j, \\
    (\forall i,j \in [K], i \neq j) I_i \cap I_j=\emptyset
    }} \prod_{k=1}^K \prod_{i \in I_k} \cos(2k \pi \xi_i),
    \end{equation}
    where $\se{}$ denotes normalized sum. Let us recall a few results from \cite{NW}.
    \begin{Lem}[{\cite[Lemma 2.2]{NW}}]
    \label{lem:2.2 from NW}
    For any $K \in \N$, $\overline{n}=(n_1,n_2,...,n_K) \in \N_0^K$ satisfying \\ $n_1+...+n_K \le d $ and any $\xi \in \T^d$
    we have
    \[
    \beta_{\overline{n}}(\xi)=\beta_{\overline{n}}^{[d]}(\xi)=\se{\substack{I_1,I_2,...,I_K \subseteq [d], \\
    (\forall j \in [K]) |I_j|=n_j, \\
    (\forall i,j \in [K], i \neq j) I_i \cap I_j=\emptyset
    }} \prod_{k=1}^K \prod_{i \in I_k} \cos(2k \pi \xi_i).
    \]
\end{Lem}
\begin{Lem}[{\cite[Lemma 2.8]{NW}}] \label{lem:2.8 from NW}
For any $K \in \N$, finite set  $J \subseteq \N$, $K$-tuple  $\overline{n} \in \N_0^K$ satisfying \\ $ \sum_{i=1}^K n_i + \max_{1 \leq i \leq K} n_i \leq |J|$ and any $\xi \in \T^d$ we have
    \[
    |\beta^J_{\overline{n}}(\xi)| \leq 6 \prod_{j=1}^K \exp\Big(- \frac{cn_j}{80 K|J|} \min\Big( \sum_{i \in J} \sin^2(j \pi \xi_i), \sum_{i \in J} \cos^2(j \pi \xi_i)\Big) \Big),
    \]
    where $c>0$ is a universal constant.
\end{Lem}
\begin{Thm}[{\cite[Theorem 1.5]{NW}}] \label{thm:1.5 from NW} 
For all $K \in \N$ there is a constant $C_K>0$ depending only on $K$ and such that for all dimension $d\in\N$ and all exponents $p\in[2,\infty]$ we have
\[
 \Big\| \sup_{ n_1,n_2,...,n_K \leq \frac{d}{2K}} |\mathcal{D}_{n_1,...,n_K}f| \Big\|_{\ell^p(\Z^d)} \leq C_K \big\| f \big\|_{\ell^p(\Z^d)}, \qquad f\in\ell^p(\Z^d),
\]
where the supremum above is taken only over non-negative integers.
\end{Thm}
We shall prove the following result.
\begin{Prop} \label{prop MG bounded by D}
    Consider any $K \in \N, \varepsilon>0$, $d \ge \max(800,2K)^{K+1}, G \in \mathcal{F}_d$ and $f \in \ell^2(\Z^d)$. Let $0<t_0 \le t_1 \le \ldots \le t_d$ be any sequence of numbers
    satisfying $d^{-1/2} \le \widetilde{\kappa}(t_0G) \le \widetilde{\kappa}(t_dG) \le d^{- \frac{1+\varepsilon}{K+1}}$.
    Then for every $f \in \ell^2(\Z^d)$ we have
   \[
   \Big\|\sup_{0 \le j \le d} |\mathcal{M}_{t_j}^Gf| \Big\|_{\ell^2(\Z^d)} \lesssim_{K,\varepsilon} \sum_{i=0}^1 \Big\| \sup_{n_1,n_2,...,n_K \leq \frac{d}{2K}} |\mathcal{D}_{n_1,...,n_K}f_i| \Big\|_{\ell^2(\Z)} + \|f\|_{\ell^2(\Z)},
   \]
   where $f=f_0+f_1$ and
\[
\supp(\widehat{f_0}) \subseteq \{\xi \in \T^d: \|\xi \| \leq \| \xi + 1/2 \| \},
\]
\[
\supp(\widehat{f_1}) \subseteq \{\xi \in \T^d: \|\xi \| > \| \xi + 1/2 \| \}.
\]
\end{Prop}
Note that in order to prove Theorem \ref{small scale dim-free} it suffices to consider $p=2$ due to interpolation and establish
\[
\bigg\|\sup_{\substack{t>0, \\
 \widetilde{\kappa}(tG) \le d^{-\varepsilon}}} \big|\mathcal{M}_{t}^G f\big| \bigg\|_{\ell^2(\Z^d)} \lesssim \|f \|_{\ell^2(\Z^d)}
\]
only for non-negative functions.
With these remarks in mind Theorem \ref{small scale dim-free} follows from Proposition \ref{prop MG bounded by D} combined with Theorem \ref{thm:1.5 from NW} and Lemma \ref{discretized supremum}, since in the case $d \lesssim_K 1$ we can simply use trivial bound for the right-hand side of expression in Lemma \ref{discretized supremum}. 
\begin{proof}[Proof of Proposition \ref{prop MG bounded by D}] Fix $K \in \N,$  $\varepsilon>0$, and take $a$ from Lemma \ref{fun fact for sup sym conv body}. In the first part of the proof we take any $t>0$ such that $\widetilde{\kappa}(tG) \in [d^{-1/2}, d^{-\frac{1+\varepsilon}{K+1}}]$ and we denote $n_t=d \widetilde{\kappa}(tG) \in \N \cap [d^{1/2}, d^{1-\frac{1+\varepsilon}{K+1}}]$. We decompose every $x \in \Z^d$ into $x=y(x)+z(x)$, where
\[
z(x)_j= \begin{cases}
    x_j, \ \ \text{if $|x_j| \geq K+1$} \\
    0, \ \ \text{otherwise}
\end{cases}
\]
and denote $Z(tG)=\{z(x) \in \Z^d: x \in tG \cap \Z^d\}.$ Then we split
\[
\mathfrak{m}_{tG}(\xi)=\widetilde{\mathfrak{m}}_{tG}(\xi)+r_{tG}(\xi),
\]
where
\begin{align*}
\widetilde{\mathfrak{m}}_{tG}(\xi)&=\frac{1}{|tG \cap \Z^d|} \sum_{\substack{x \in tG \cap \Z^d, \\ |z(x)|_1 \le a, \\ |\{i: |x_i|=1\} | \geq n_t/2}} e(x \cdot \xi),
\\
r_{tG}(\xi)&=\frac{1}{|tG \cap \Z^d|} \sum_{\substack{x \in tG \cap \Z^d, \\ |z(x)|_1>a, \\ |\{i: |x_i|=1\} | \ge n_t/2}} e(x \cdot \xi) + \frac{1}{|tG \cap \Z^d|} \sum_{\substack{x \in tG \cap \Z^d \\ |\{i: |x_i|=1\} | < n_t/2}} e(x \cdot \xi).
\end{align*}

By Lemmas \ref{BMSW a lot of 1} and \ref{fun fact for sup sym conv body}, since $\widetilde{tG} \le d^{-\frac{1+\varepsilon}{K+1}}$ and $d \geq \max(800,2K)^{K+1}$, we have
\begin{equation}
\label{eq:rsnbound}
|r_{tG}(\xi)| \lesssim \frac{1}{d} + 2^{-n_t/2} \le \frac{1}{d}+2^{-\sqrt{d}/2} \lesssim \frac{1}{d} ,
\end{equation}
uniformly in $\xi \in \T^d$. 
For $z \in Z(tG)$ we define
\begin{align*}
&I(z)=\Big\{ (i_1, \ldots, i_K) \in \big([d]\cup\{0\}\big)^K: (\exists y \in \{-K,\ldots, K \}^d) \big( \supp(y) \cap \supp(z)= \emptyset \\
& (\forall j \in [K]) \ |\{i \in [d]: |y_i|=j \}|=i_j, \ y+z \in tG \ 
\Big\}.
\end{align*}
Note that since $G$ is 1-symmetric, if $(i_1,i_2,...,i_K) \in I(z)$ then for all $y \in \{-K,\ldots, K \}^d$ satisfying
\[
 \supp(y) \cap \supp(z)= \emptyset \quad  \text{and} \quad  (\forall j \in [K]) |\{i \in [d]: |y_i|=j \}|=i_j
\]
we have $y+z \in tG$.
Due to the above, to treat the term $\widetilde{\mathfrak{m}}_{tG}(\xi)$ we rewrite
\begin{align*}
&\widetilde{\mathfrak{m}}_{tG}(\xi)= \frac{1}{|tG \cap \Z^d|} \sum_{\substack{z \in Z(tG), \\ |z|_1 \leq a}} e(z \cdot \xi) \hspace{-0.5cm}\sum_{\substack{y \in \{-K,...,K\}^d, \\ 
\\ \supp(y) \cap \supp(z)= \emptyset, 
\\ |\{i: |y_i|=1 \}| \geq n_t/2,
\\ z+y \in tG
}} \hspace{-0.5cm}e(y \cdot \xi)
\\
&=\frac{1}{|tG \cap \Z^d|} \sum_{l=0}^a \sum_{\substack{z \in Z(tG),
\\ |\supp(z)|=l, \\ |z|_1 \le a
}} e(z \cdot \xi) \hspace{-0.5cm} \sum_{\substack{(i_1,...,i_K) \in I(z), \\
i_1 \geq n_t/2
}} 
\hspace{-0.4cm} \sum_{\substack{y \in \{-K,...,K\}^d, 
\\ \supp(y) \cap \supp(z)= \emptyset, 
\\
(\forall j \in [K]) |\{i \in [d]: |y_i|=j\}|=i_j
}} \hspace{-0.5cm} e(y \cdot \xi).
\end{align*}
Now we define two new multipliers
\begin{align*}
\phi_{tG,0}(\xi)&=\frac{1}{|tG \cap \Z^d|} \sum_{l=0}^a \sum_{\substack{z \in Z(tG),
\\ |\supp(z)|=l, \\ |z|_1 \le a
}}  \sum_{\substack{(i_1,...,i_K) \in I(z), \\
i_1 \geq n_t/2
}} 
\hspace{-0.4cm}\sum_{\substack{y \in \{-K,...,K\}^d, 
\\ \supp(y) \cap \supp(z)= \emptyset, 
\\
(\forall j \in [K]) |\{i \in [d]: |y_i|=j\}|=i_j
}} \hspace{-1cm}e(y \cdot \xi),
\\
\phi_{tG,1}(\xi)&=\frac{1}{|tG \cap \Z^d|} \sum_{l=0}^a \sum_{\substack{z \in Z(tG),
\\ |\supp(z)|=l, \\ |z|_1 \le a
}} (-1)^{\sum_{i=1}^d z_i} \hspace{-0.5cm}\sum_{\substack{(i_1,...,i_K) \in I(z), \\
i_1 \geq n_t/2
}} 
\hspace{-0.4cm}\sum_{\substack{y \in \{-K,...,K\}^d, 
\\ \supp(y) \cap \supp(z)= \emptyset, 
\\
(\forall j \in [K]) |\{i \in [d]: |y_i|=j\}|=i_j
}} \hspace{-1cm} e(y \cdot \xi).
\end{align*}
We claim that two crucial inequalities hold:
\begin{enumerate} \item If $\| \xi \| \leq \| \xi+ 1/2 \|$ then
\begin{equation} \label{eq:4.1}
|\widetilde{\mathfrak{m}}_{tG}(\xi)-\phi_{tG,0}(\xi)| \lesssim_{K,\varepsilon} \frac{1}{n_t}\lesssim \frac{1}{\sqrt{d}}.
\end{equation}
\item  If $\| \xi \| \geq \| \xi+ 1/2 \|$ then
\begin{equation} \label{eq:4.2}
|\widetilde{\mathfrak{m}}_{tG}(\xi)-\phi_{tG,1}(\xi)| \lesssim_{K,\varepsilon} \frac{1}{n_t} \lesssim \frac{1}{\sqrt{d}}.
\end{equation}
\end{enumerate}

We start with the proof of  \eqref{eq:4.1} and assume that $\|\xi\| \leq \| \xi+1/2 \|$. First, notice that for any $z \in Z(tG)$ and $(i_1,...,i_K) \in I(z)$ we have
\begin{equation}
\label{eq:syebd}
\begin{split}
&\sum_{\substack{y \in \{-K,...,K\}^d,  
\\ \supp(y) \cap \supp(z)= \emptyset, 
\\
(\forall j \in [K]) |\{i \in [d]: |y_i|=j\}|=i_j
}}\hspace{-1cm} e(y \cdot \xi)= 2^{\sum_{j=1}^K i_j} \hspace{-1cm}\sum_{\substack{I_1,...,I_K \subseteq [d] \setminus \supp(z), \\
 (\forall j \in [K]) |I_j|=i_j, \\
    (\forall i,j \in [K], i \neq j) I_i \cap I_j=\emptyset}} \prod_{k=1}^K \prod_{i \in I_k} \cos(2k \pi \xi_i)   
\\
&=  2^{\sum_{j=1}^K i_j} \frac{(d-|\supp(z)|)!}{\prod_{j \in [K]}i_j! \cdot (d-|\supp(z)|-\sum_{j \in [K]}i_j)!} \beta_{\overline{i}}^{[d] \setminus \supp(z)} (\xi);
\end{split}
\end{equation}
recall Definition \ref{generalised beta defn} for the meaning of $\beta_{\overline{i}}^{[d] \setminus \supp(z)}.$ 
Let 
\[
F(l,\overline{i})=  2^{\sum_{j=1}^K i_j} \frac{(d-l)!}{\prod_{j \in [K]}i_j! (d-l-\sum_{j \in [K]}i_j)!}.
\]
Exploiting invariance of $Z(tG)$ under sign changes of coordinates we get
\begin{align*}
&\widetilde{\mathfrak{m}}_{tG}(\xi)-\phi_{tG,0}(\xi)\\
&= \frac{1}{|tG \cap \Z^d|} \sum_{l=0}^a \sum_{\substack{z \in Z(tG),
\\ |\supp(z)|=l, \\ |z|_1 \le a
}} \Big( \prod_{j=1}^d \cos(2\pi z_j \xi_j)-1 \Big)  \hspace{-0.2cm}\sum_{\substack{(i_1,...,i_K) \in I(z), 
\\ 
i_1 \geq n_t/2
}} \hspace{-0.4cm} F(l, \overline{i}) \beta_{\overline{i}}^{[d] \setminus  \supp(z)}(\xi).
\end{align*}
From the inequality $|\prod_{j \in [d]} a_j - 1| \leq \sum_{j \in [d]} |a_j-1|$, valid if $\sup_j |a_j| \leq 1$, and Lemma \ref{lem:2.8 from NW} we obtain 
\begin{align*} 
&|\widetilde{\mathfrak{m}}_{tG}(\xi)-\phi_{tG,0}(\xi)| \lesssim 
\frac{1}{|tG \cap \Z^d|} \sum_{l=0}^a \sum_{\substack{z \in Z(tG),
\\ |\supp(z)|=l, \\ |z|_1 \le a
}}  \sum_{j=1}^d \sin^2( \pi z_j \xi_j)  \hspace{-0.5cm}\sum_{\substack{(i_1,...,i_K) \in I(z), 
\\ 
i_1 \geq n_t/2
}} \hspace{-0.4cm} F(l, \overline{i}) 
\\
 &\cdot\exp\Bigg( -\frac{ci_1}{80K(d-|\supp(z)|)} \min\bigg( \hspace{-0cm}\sum_{k \in [d] \setminus \supp(z)} \hspace{-0.5cm} \sin^2( \pi  \xi_k),  \hspace{-0.5cm}\sum_{k \in [d] \setminus \supp(z)} \hspace{-0.5cm}\cos^2( \pi  \xi_k)  \bigg) \Bigg)
 \\
 &\lesssim_{K,\varepsilon}
 \frac{1}{|tG \cap \Z^d|} \sum_{l=0}^a \sum_{\substack{z \in Z(tG),
\\ |\supp(z)|=l, \\ |z|_1 \le a
}}  \sum_{j=1}^d \sin^2( \pi z_j \xi_j)  \hspace{-0.5cm}\sum_{\substack{(i_1,...,i_K) \in I(z), 
\\ 
i_1 \geq n_t/2
}} \hspace{-0.4cm} F(l, \overline{i}) \exp\Big( -\frac{cn_t}{160Kd} \|\xi\|^2 \Big),
\end{align*}
where in the last estimate above we utilized our assumption $\|\xi\|\le \|\xi+1/2\|$ and the inequality $|\supp(z)| \le a \lesssim_{K,\varepsilon} 1.$ Using the inequality $\sin^2(\pi xt) \leq |x|^2 \sin^2(\pi t),$ which holds for all $x \in \Z, t \in \mathbb{R},$ we get
\begin{align*}
&|\widetilde{\mathfrak{m}}_{tG}(\xi)-\phi_{tG,0}(\xi)| 
\\
& \lesssim \exp\Big( -\frac{cn_t}{160Kd} \| \xi \|^2 \Big)  \sum_{j=1}^d \sin^2( \pi \xi_j)\frac{1}{|tG \cap \Z^d|} \sum_{l=0}^a \hspace{-0cm}\sum_{\substack{i_1,...,i_K \in [d] \cup \{0\}, 
\\ 
i_1 \geq n_t/2
}} \hspace{-0.2cm} F(l, \overline{i}) \hspace{-0.5cm}\sum_{\substack{z \in Z(tG),
\\ |\supp(z)|=l, \\ |z|_1 \le a, \\ (i_1,\ldots,i_K) \in I(z)
}}  \hspace{-0.4cm} z_j^2.
\end{align*}
By symmetry invariance the innermost sum above is
\[
\sum_{\substack{z \in Z(tG),
\\ |\supp(z)|=l, \\ |z|_1 \le a, \\ (i_1,\ldots,i_K) \in I(z)
}}  \hspace{-0.4cm} z_j^2=  \frac{1}{d} \hspace{-0.5cm}\sum_{\substack{z \in Z(tG),
\\ |\supp(z)|=l, \\ |z|_1 \le a, \\ (i_1,\ldots,i_K) \in I(z)
}}  \hspace{-0.4cm} |z|_2^2 \le  \frac{a^2}{d} \hspace{-0.5cm}\sum_{\substack{z \in Z(tG),
\\ |\supp(z)|=l, \\ |z|_1 \le a, \\ (i_1,\ldots,i_K) \in I(z)
}}  \hspace{-0.4cm} 1
\]
and thus we obtain 
\begin{align*}
&|\widetilde{\mathfrak{m}}_{tG}(\xi)-\phi_{tG,0}(\xi)| \\
&\lesssim_{K,\varepsilon} 
\exp\Big( -\frac{cn_t}{160Kd} \| \xi \|^2 \Big)  \frac{ \| \xi \|^2}{d} \frac{1}{|tG \cap \Z^d|} \sum_{l=0}^a \hspace{-0cm}\sum_{\substack{i_1,...,i_K \in [d] \cup \{0\}, 
\\ 
i_1 \geq n_t/2
}} \hspace{-0.2cm} F(l, \overline{i}) \hspace{-0.5cm}\sum_{\substack{z \in Z(tG),
\\ |\supp(z)|=l, \\ |z|_1 \le a, \\ (i_1,\ldots,i_K) \in I(z)
}}  \hspace{-0.4cm} 1.
\end{align*}
However, putting $\xi=0$ in \eqref{eq:syebd} we notice that 

\begin{align*}
&\frac{1}{|tG \cap \Z^d|} \sum_{l=0}^a \hspace{-0cm}\sum_{\substack{i_1,...,i_K \in [d] \cup \{0\}, 
\\ 
i_1 \geq n_t/2
}} \hspace{-0.2cm} \sum_{\substack{z \in Z(tG),
\\ |\supp(z)|=l, \\ |z|_1 \le a, \\ (i_1,\ldots,i_K) \in I(z)
}}  \hspace{-0.4cm} F(l, \overline{i})
\\
&=\frac{1}{|tG \cap \Z^d|} \sum_{l=0}^a \hspace{-0cm}\sum_{\substack{i_1,...,i_K \in [d] \cup \{0\}, 
\\ 
i_1 \geq n_t/2
}} \hspace{-0.2cm} \sum_{\substack{z \in Z(tG),
\\ |\supp(z)|=l, \\ |z|_1 \le a, \\ (i_1,\ldots,i_K) \in I(z)
}} 
\sum_{\substack{y \in \{-K,...,K\}^d, 
\\ \supp(y) \cap \supp(z)= \emptyset, 
\\
(\forall j \in [K]) |\{i \in [d]: |y_i|=j\}|=i_j
}} \hspace{-0.5cm} 1 \\
&= \widetilde{\mathfrak{m}}_{tG}(0) \leq \mathfrak{m}_{tG}(0)=1.
\end{align*}
Thus we have
\[
|\widetilde{\mathfrak{m}}_{tG}(\xi)-\phi_{tG,0}(\xi)| \lesssim_{K,\varepsilon} \exp\Big( -\frac{cn_t}{160Kd} \| \xi \|^2 \Big)  \frac{ \| \xi \|^2}{d} \lesssim_{K,\varepsilon} \frac{1}{n_t}
\]
and this finishes the proof of \eqref{eq:4.1}.

\par The proof of  \eqref{eq:4.2} is similar and thus we only sketch it. If $\| \xi+1/2 \| \leq \| \xi \|$ then one has
\begin{align*}
&\widetilde{\mathfrak{m}}_{tG}(\xi)=\frac{1}{|tG \cap \Z^d|} \sum_{l=0}^a \sum_{\substack{z \in Z(tG),
\\ |\supp(z)|=l, \\ |z|_1 \le a
}}(-1)^{\sum_{j=1}^d z_j} e\big(z \cdot (\xi+1/2)\big) 
\\ 
&\cdot \sum_{\substack{(i_1,...,i_K) \in I(z), \\ 
i_1 \geq n_t/2
}} 
(-1)^{\sum_{j \in [K]} ji_j} \hspace{-1cm}
\sum_{\substack{y \in \{-K,...,K\}^d, \\ 
\\ \supp(y) \cap \supp(z)= \emptyset, 
\\
(\forall j \in [K]) |\{i \in [d]: |y_i|=j\}|=i_j
}} e\big(y \cdot (\xi+1/2) \big),
\\
&\phi_{tG,1}(\xi)=\frac{1}{|tG \cap \Z^d|} \sum_{l=0}^a \sum_{\substack{z \in Z(tG),
\\ |\supp(z)|=l, \\ |z|_1 \le a
}} (-1)^{\sum_{i=1}^d z_i} \\
&\cdot\sum_{\substack{(i_1,...,i_K) \in I(z), \\ 
i_1 \geq n_t/2
}} 
(-1)^{\sum_{j \in [K]} ji_j} \hspace{-1cm}
\sum_{\substack{y \in \{-K,...,K\}^d, \\ 
\\ \supp(y) \cap \supp(z)= \emptyset, 
\\
(\forall j \in [K]) |\{i \in [d]: |y_i|=j\}|=i_j
}} \hspace{-0.5cm} e\big(y \cdot (\xi+1/2) \big).
\end{align*}
Hence using similar arguments as before and denoting $\eta=\xi+1/2$ we get
\begin{align*}
&|\widetilde{\mathfrak{m}}_{tG}(\xi)-\phi_{tG,1}(\xi)| \\
&\lesssim \frac{1}{|tG \cap \Z^d|} \sum_{l=0}^a \sum_{\substack{z \in Z(tG),
\\ |\supp(z)|=l, \\ |z|_1 \le a
}} \Big| \prod_{j=1}^d \cos\big(2\pi z_j \eta_j\big)-1 \Big|
\hspace{-0.2cm}\sum_{\substack{(i_1,...,i_K) \in I(z), \\
i_1 \geq n_t/2
}} \hspace{-0.4cm}F(l, \overline{i}) |\beta_{\overline{i}}^{[d] \setminus  \supp(z)}(\eta)|.
\end{align*}
Since $\|\eta\|\le \|\eta+1/2\|$ we see that the quantity above was estimated by $O_K(1/n_t)$ during the proof of \eqref{eq:4.1}. This justifies \eqref{eq:4.2}.

Having proven \eqref{eq:4.1} and \eqref{eq:4.2} we take a closer look at $\phi_{tG,0}(\xi)$ and $\phi_{tG,1}(\xi).$ Notice that 
\begin{align*}
&\phi_{tG,0}(\xi)= \frac{1}{|tG \cap \Z^d|} \sum_{l=0}^a \sum_{\substack{J \subseteq [d], \\ |J|=l}}\sum_{\substack{z \in Z(tG),
\\ \supp(z)=J, \\ |z|_1 \le a
}}  \sum_{\substack{(i_1,...,i_K) \in I(z), 
\\ 
i_1 \geq n_t/2, 
}} 
\hspace{-0.5cm} \sum_{\substack{y \in \{-K,...,K\}^d,  
\\ \supp(y) \cap J= \emptyset, 
\\
(\forall j \in [K]) |\{i \in [d]: |y_i|=j\}|=i_j
}} \hspace{-1.2cm} e(y \cdot \xi) 
\\
&=\frac{1}{|tG \cap \Z^d|} \sum_{l=0}^a \sum_{\substack{J \subseteq [d], \\ |J|=l}} \sum_{\substack{i_1,...,i_K \in [d] \cup\{0\}, 
\\
i_1 \geq n_t/2
}} \hspace{-0.3cm} \sum_{\substack{z \in Z(tG),
\\ \supp(z)=J, \\ |z|_1 \le a, \\ (i_1,\ldots, i_K) \in I(z)
}} 1 \cdot
\hspace{-0.4cm}
\sum_{\substack{y \in \{-K,...,K\}^d, 
\\ \supp(y) \cap J= \emptyset, 
\\
(\forall j \in [K]) |\{i \in [d]: |y_i|=j\}|=i_j
}} \hspace{-0.5cm}e(y \cdot \xi).
\end{align*}
Due to symmetry invariance of $Z(tG)$, for any $J \subseteq [d]$, $|J|=l$ and $i_1,...,i_K \in [d]$ we have
\[
\sum_{\substack{z \in Z(tG),
\\ \supp(z)=J, \\ |z|_1 \le a, \\ (i_1,\ldots, i_K) \in I(z)
}} 1= \frac{1}{\binom{d}{l}}\sum_{\substack{z \in Z(tG),
\\ |\supp(z)|=l, \\ |z|_1 \le a, \\ (i_1,\ldots, i_K) \in I(z)
}} 1,
\]
using the above and changing order of sums we obtain
\[
\phi_{tG,0}(\xi)
=\frac{1}{|tG \cap \Z^d|} \sum_{l=0}^a \sum_{\substack{z \in Z(tG),
\\ |\supp(z)|=l, \\ |z|_1 \le a
}}  
\sum_{\substack{(i_1,\ldots, i_K) \in I(z), \\ 
i_1 \geq n_t/2
}} 
\se{\substack{J \subseteq [d], \\ |J|=l}}\hspace{-0.7cm}
\sum_{\substack{y \in \{-K,...,K\}^d,  
\\ \supp(y) \cap J= \emptyset, 
\\
(\forall j \in [K]) |\{i \in [d]: |y_i|=j\}|=i_j
}} \hspace{-1.2cm} e(y \cdot \xi).
\]
Notice that 
\begin{align*}
&\se{\substack{J \subseteq [d], \\ |J|=l}}
\sum_{\substack{y \in \{-K,...,K\}^d,  
\\ \supp(y) \cap J= \emptyset, 
\\
(\forall j \in [K]) |\{i \in [d]: |y_i|=j\}|=i_j
}} e(y \cdot \xi)\\
&= \frac{1}{\binom{d}{l}}\sum_{\substack{y \in \{-K,...,K\}^d,  
\\
(\forall j \in [K]) |\{i \in [d]: |y_i|=j\}|=i_j
}} e(y \cdot \xi) | \{ J \subseteq [d]: |J|=l, J \cap \supp(y)= \emptyset \}|
\\
&=\frac{\binom{d-\sum_{j \in[K]}i_j}{l}}{\binom{d}{l}} \sum_{\substack{y \in \{-K,...,K\}^d,  
\\
(\forall j \in [K]) |\{i \in [d]: |y_i|=j\}|=i_j
}} e(y \cdot \xi)= \frac{\binom{d-\sum_{j \in[K]}i_j}{l}}{\binom{d}{l}} |D_{\overline{i}}| \beta_{\overline{i}}(\xi), 
\end{align*}
where $\overline{i}=(i_1,\ldots,i_K),$
and this implies 
\[
\phi_{tG,0}(\xi)
=\frac{1}{|tG \cap \Z^d|} \sum_{l=0}^a \sum_{\substack{z \in Z(tG),
\\ |\supp(z)|=l, \\ |z|_1 \le a
}}  
\sum_{\substack{(i_1,\ldots, i_K) \in I(z), \\ 
i_1 \geq n_t/2
}} 
\frac{\binom{d-\sum_{j \in[K]}i_j}{l}}{\binom{d}{l}} |D_{\overline{i}}| \beta_{\overline{i}}(\xi).
\]
Using exactly the same argument one can obtain
\[\phi_{tG,1}(\xi)
=\frac{1}{|tG \cap \Z^d|} \sum_{l=0}^a \sum_{\substack{z \in Z(tG),
\\ |\supp(z)|=l, \\ |z|_1 \le a
}}  (-1)^{\sum_{j=1}^dz_j}
\hspace{-0.7cm}\sum_{\substack{(i_1,\ldots, i_K) \in I(z), \\ 
i_1 \geq n_t/2
}} 
\hspace{-0.8cm}\frac{\binom{d-\sum_{j \in[K]}i_j}{l}}{\binom{d}{l}} |D_{\overline{i}}| \beta_{\overline{i}}(\xi).
\]
Since $\phi_{tG,0}(0)=\widetilde{\mathfrak{m}}_{tG}(0)=\mathfrak{m}_{tG}(0)-r_{tG}(0) \leq 1$, we see that
\begin{equation*} 
\frac{1}{|tG \cap \Z^d|} \sum_{l=0}^a \sum_{\substack{z \in Z(tG),
\\ |\supp(z)|=l, \\ |z|_1 \le a
}}  
\sum_{\substack{(i_1,\ldots, i_K) \in I(z), \\ 
i_1 \geq n_t/2
}} 
\frac{\binom{d-\sum_{j \in[K]}i_j}{l}}{\binom{d}{l}} |D_{\overline{i}}| \leq 1.
\end{equation*}
This implies that for $\epsilon \in \{0,1\}$ and $f\in \ell^2$ we have the pointwise bound, valid for $x\in \Z^d,$
\begin{equation}
\label{eq:4.4}
\begin{split}
\Big|\mathcal{F}^{-1} \Big( \phi_{tG,\epsilon}\widehat{f} \Big)(x)\Big| &\leq \frac{1}{|tG \cap \Z^d|} \sum_{l=0}^a \sum_{\substack{z \in Z(tG),
\\ |\supp(z)|=l, \\ |z|_1 \le a
}}  
\sum_{\substack{(i_1,\ldots, i_K) \in I(z), \\ 
i_1 \geq n_t/2
}} 
\\
&\frac{\binom{d-\sum_{j \in[K]}i_j}{l}}{\binom{d}{l}} |D_{\overline{i}}| \Big|\mathcal{F}^{-1} \Big( \beta_{\overline{i}}\widehat{f} \Big)(x)\Big|
\leq \sup_{i_1,...,i_K \leq \frac{d}{2K}} |\mathcal{D}_{\overline{i}}f|(x).
\end{split}
\end{equation}
Above we used observation that
\[
z \in Z(tG), (i_1,...,i_K) \in I(z) \implies (j \in [K]) \ i_j \le \sum_{j=1}^K ji_j \le d \widetilde{\kappa}(tG)=n_t
\]
and the fact that $n_t \leq \frac{d}{2K}$, which follows from assumptions \\ $d \geq (2K)^{K+1}$ and $\widetilde{\kappa}(tG) \leq d^{-\frac{1+\varepsilon}{K+1}}$.

Having established \eqref{eq:rsnbound}, \eqref{eq:4.1}, \eqref{eq:4.2} and \eqref{eq:4.4} we can now finish the proof of Proposition \ref{prop MG bounded by D}. Take any $f \in \ell^2(\Z^d)$ and decompose $f=f_0+f_1$, where
\[
\supp(\widehat{f_0}) \subseteq \{\xi \in \T^d: \|\xi \| \leq \| \xi + 1/2 \| \},
\]
\[
\supp(\widehat{f_1}) \subseteq \{\xi \in \T^d: \|\xi \| > \| \xi + 1/2 \| \}.
\]
Then for the sequence $(t_j)_{j=0}^d$ satisfying assumptions from the statement of Proposition \ref{prop MG bounded by D} we have
\begin{equation}
\label{eq:sf0f1}
\Big\|\sup_{0 \le j \le d} |\mathcal{M}_{t_j}^Gf| \Big\|_{\ell^2(\Z^d)}\lesssim \Big\|\sup_{0 \le j \le d} |\mathcal{M}_{t_j}^Gf_0| \Big\|_{\ell^2(\Z^d)}  +\Big\|\sup_{0 \le j \le d} |\mathcal{M}_{t_j}^Gf_1| \Big\|_{\ell^2(\Z^d)}. 
\end{equation}

The  term on the right hand side of \eqref{eq:sf0f1}  containing $f_0$ is estimated with the aid of \eqref{eq:4.4} 
\begin{align*}
&\Big\|\sup_{0 \le j \le d} |\mathcal{M}_{t_j}^Gf_0| \Big\|_{\ell^2(\Z^d)} \\
&\leq \Big\|  \sup_{0 \le j \le d} \Big|\mathcal{F}^{-1} \Big( \phi_{t_jG,0}\widehat{f_0} \Big)\Big| \Big\|_{\ell^2(\Z^d)} +\Big\|  \sup_{ 0 \le j \le d} \Big|\mathcal{F}^{-1} \Big( (\mathfrak{m}_{t_jG}-\phi_{t_jG,0})\widehat{f_0} \Big)\Big| \Big\|_{\ell^2(\Z^d)}\\
&\le \Big\|\sup_{i_1,...,i_K \leq \frac{d}{2K}} |\mathcal{D}_{\overline{i}}f_0|\Big\|_{\ell^2(\Z^d)}+\Big\|  \sup_{ 0 \le j \le d} \Big|\mathcal{F}^{-1} \Big( (\mathfrak{m}_{t_jG}-\phi_{t_jG,0})\widehat{f_0} \Big)\Big| \Big\|_{\ell^2(\Z^d)}.
\end{align*} 
For the second term of the last inequality we use Parseval's formula, \eqref{eq:rsnbound}, and \eqref{eq:4.1} to get
\begin{align*}
&\Big\|  \sup_{ 0 \le j \le d} \Big|\mathcal{F}^{-1} \Big( (\mathfrak{m}_{t_jG}-\phi_{t_jG,0})\widehat{f_0} \Big)\Big| \Big\|_{\ell^2(\Z^d)}^2 \\
&\leq \Big\| \Big(\sum_{j=0}^d \Big|\mathcal{F}^{-1} \Big( (\mathfrak{m}_{t_jG}-\phi_{t_jG,0})\widehat{f_0} \Big)\Big|^2 \Big)^{1/2} \Big\|_{2}^2 =\sum_{j=0}^d \Big\|\mathcal{F}^{-1} \Big( (\mathfrak{m}_{t_jG}-\phi_{t_jG,0})\widehat{f_0} \Big)\Big\|_{2}^2  \\
&=  \int_{\T^d} | \widehat{f_0}(\xi)|^2 \sum_{j=0}^d  |\mathfrak{m}_{t_jG}(\xi)-\phi_{t_jG,0}(\xi)|^2 \,d\xi 
\\
&\lesssim_{K, \varepsilon}  \int_{\T^d} | \widehat{f_0}(\xi)|^2\,d\xi \cdot \sum_{j=0}^d\big( \frac{1}{d^2}+ \frac{1}{d}\big)  \lesssim_{K, \varepsilon} \| f_0 \|_2^2.
\end{align*}
In summary, we justified that
\[
\Big\|\sup_{0 \le j \le d} |\mathcal{M}_{t_j}^Gf_0| \Big\|_{\ell^2(\Z^d)} \lesssim_{K,\varepsilon}  \Big\| \sup_{n_1,...,n_K \leq \frac{d}{2K}} |\mathcal{D}_{\overline{n}}f_0| \Big\|_{\ell^2(\Z^d)}+ \|f_0 \|_{\ell^2(\Z^d)}.
\]

Using similar arguments, but with $\phi_{tG,0}$ replaced by $\phi_{tG,1}$ and \eqref{eq:4.2} used in place of \eqref{eq:4.1}, we can also estimate the term on the right hand side of \eqref{eq:sf0f1} containing $f_1,$ namely
\[
\Big\|\sup_{0 \le j \le d} |\mathcal{M}_{t_j}^Gf_1| \Big\|_{\ell^2(\Z^d)} \lesssim_{K,\varepsilon}  \Big\| \sup_{n_1,...,n_K \leq \frac{d}{2K}} |\mathcal{D}_{\overline{n}}f_1| \Big\|_{\ell^2(\Z^d)}+ \|f_1 \|_{\ell^2(\Z^d)}.
\]
Thus if $d$ is sufficiently large we obtain
\[
   \Big\|\sup_{0 \le j \le d} |\mathcal{M}_{t_j}^Gf| \Big\|_{\ell^2(\Z^d)} \lesssim_{K,\varepsilon} \sum_{i=0}^1 \Big\| \sup_{n_1,n_2,...,n_K \leq \frac{d}{2K}} |\mathcal{D}_{n_1,...,n_K}f_i| \Big\|_{\ell^2(\Z^d)} + \|f\|_{\ell^2(\Z^d)},
   \]
and this completes the proof of Proposition \ref{prop MG bounded by D}.

\end{proof}


\begin{thebibliography}{1}


\normalsize
\baselineskip=17pt


\bibitem{bob-naz}
S.G. Bobkov, F.L. Nazarov (2003). \textit{On Convex Bodies and Log-Concave Probability Measures with Unconditional Basis.} In: Milman, V.D., Schechtman, G. (eds) Geometric Aspects of Functional Analysis. Lecture Notes in Mathematics, vol 1807. Springer, Berlin, Heidelberg. \url{https://doi.org/10.1007/978-3-540-36428-3_6}
    \bibitem{B1} 
J. Bourgain, 
\textit{On high dimensional maximal
  functions associated to convex bodies}, 
	Amer. J.  Math. {\bf 108}, (1986), 1467--1476.
        
\bibitem{B2} 
J. Bourgain, 
\textit{On $L^p$ bounds for maximal
  functions associated to convex bodies in $\R^n$}, 
	Israel J. Math. {\bf 54}, (1986), 257--265.

\bibitem{B3} 
J. Bourgain, 
\textit{On the Hardy-Littlewood maximal
  function for the cube}, 
	Israel J. Math.  {\bf 203}, (2014), 275--293.

    \bibitem{BMSW1} 
J. Bourgain, M. Mirek, E.M. Stein, B. Wr\'obel,
\textit{Dimension-free variational estimates on $L^p(\R^d)$ for
  symmetric convex bodies}, 
	Geom. Funct. Anal.
  {\bf 28}, (2018),  58--99.


\bibitem{balls} J. Bourgain, M. Mirek, E. Stein and B. Wróbel,
\emph{On Discrete Hardy--Littlewood Maximal Functions over the Balls in {$\mathbb{Z}^d$}: dimension-free estimates},
Geometric Aspects of Functional Analysis. Lecture Notes in Mathematics 2256, Springer, 2020, pages 127--169.



  \bibitem{cubes} J. Bourgain, M. Mirek, E.M. Stein, B. Wr\'obel,
  \textit{Dimension-free estimates for discrete Hardy--Littlewood
  	averaging operators over the cubes in
  	$\mathbb Z^d$}, 
		Amer. J. Math. {\bf 141}, (2019), 857--905.


  \bibitem{BMSW4} 
	J. Bourgain, M. Mirek, E.M. Stein, B. Wr\'obel,
  \textit{On the Hardy--Littlewood maximal functions in high
  dimensions: Continuous and discrete perspective}.
P.\ Ciatti, A.\ Martini (eds.), Geometric Aspects of
Harmonic Analysis. A conference proceedings  on the occasion of Fulvio Ricci's 70th birthday,
Cortona, Italy, 25--29.06.2018. Springer INdAM Series 45, (2021), pp.\ 456.

\bibitem{conv book}
S. Brazitikos, A. Giannopoulos, P. Valettas and B.-H. Vritsiou,\textit{ Geometry of Isotropic Convex Bodies, Mathematical Surveys and Monographs}, Vol. 196 (American Mathematical Society, Providence, RI, 2014).

\bibitem{Car1} 
A. Carbery, 
\textit{An almost-orthogonality principle
  with applications to maximal functions associated to convex bodies},
Bull. Amer. Math. Soc.  {\bf 14}, (1986), 269--274.

\bibitem{Cow1}
M. Cowling,
\textit{Harmonic analysis on semigroups}, Annals of Mathematics (2) {\bf 117} 1983, pp.\ 267-283.


\bibitem{HL-cont} L. Deleaval, O. Guédon and B. Maurey,
\emph{Dimension free bounds for the Hardy–Littlewood maximal operator associated to convex sets.},  
Ann. Fac. Sci. Toulouse Math. (6) 27, no.1, 2018,1--198.

\bibitem{HKS}

D. Halikias, B. Klartag, B. A. Slomka. \textit{Discrete variants of Brunn–Minkowski type inequalities.} Annales de la Faculté des sciences de Toulouse : Mathématiques, Serie 6, Volume spécial à l’occasion du semestre thématique “Calculus of Variations and Probability”, Volume 30 (2021) no. 2, pp. 267-279. doi: 10.5802/afst.1674

\bibitem{HSX}
M. Henk, M. Schymura, and F. Xue,
\textit{Packing minima and lattice points in convex bodies},
Moscow J. Comb. Number Theory 10 (2021), no. 1, 25--48.

\bibitem{He}
D. Hensley (1980). \textit{ Slicing Convex Bodies-Bounds for Slice Area in Terms of the Body’s Covariance.} Proceedings of the American Mathematical Society, 79(4), 619–625. \url{https://doi.org/10.2307/2042510}

\bibitem{INZ}
D. Iglesias, J. Yepes Nicholás, A. Zvavitch, \textit{Brunn–Minkowski type inequalities for the lattice point enumerator}, Adv. Math. 370 (2020), article no. 107193.

\bibitem{iso conj}
B. Klartag, J. Lehec, \textit{Affirmative Resolution of Bourgain’s Slicing Problem Using Guan’s Bound.} Geom. Funct. Anal. 35, 1147–1168 (2025). \url{https://doi.org/10.1007/s00039-025-00718-w}
\bibitem{KNW}
D. Kosz, J. Niksiński, B. Wróbel, \textit{Uniform estimates for Delannoy numbers \\ and dimension-free estimates for discrete maximal functions over cross-polytopes}, preprint, arXiv:2604.15844
\bibitem{KMPW} 
D. Kosz, M. Mirek, P. Plewa and B. Wróbel,
\emph{Some remarks on dimension-free estimates for the discrete Hardy—Littlewood maximal functions.},
Israel J. Math. 254, 2023, 1--38. 

\bibitem{Lev}V. I. Levenshtein,
\emph{Krawtchouk polynomials and universal bounds for codes and designs in Hamming spaces},
IEEE Trans. Inform. Theory 41, no. 5, 1995, 1303--1321.
\bibitem{matusek}
 J. Matoušek, \textit{Lectures on Discrete Geometry}, Graduate Texts in Mathematics, vol. 212, Springer-Verlag, New York, 2002.
\bibitem{MP} 
V. D. Milman, A. Pajor. \textit{Isotropic position and inertia ellipsoids and zonoids of the unit ball of a normed n-dimensional space.} In: Lindenstrauss, J., Milman, V.D. (eds) Geometric Aspects of Functional Analysis. Lecture Notes in Mathematics, vol 1376, (1989). Springer, Berlin, Heidelberg.
 \bibitem{MiSzWr} M.\ Mirek, T.Z.\ Szarek, B.\ Wr\'obel,
\emph{Dimension-Free Estimates for the Discrete Spherical Maximal Functions}, International Mathematics Research Notices, Volume 2024, Issue 2, January 2024, Pages 901–963. doi: 10.1093/imrn/rnac329
 \bibitem{Mul1} 
D. M\"uller, 
\textit{A geometric bound for maximal
  functions associated to convex bodies}, 
	Pacific J. Math. {\bf 142}, (1990), 297--312.
\bibitem{Ni1}
J. Niksiński,
\textit{Dimension-free estimates on $\ell^2(\Z^d)$ for discrete dyadic maximal function over $\ell^1$ balls: small scales},
Coll. Math. 175 (2024), 37--54.

\bibitem{Ni2}
J. Niksiński, \textit{Remark on dimension-free estimates for discrete maximal functions over $\ell^q$ balls: Small dyadic scales}. Bull. London Math. Soc., 58, (2026): e70220. \url{https://doi.org/10.1112/blms.70220}
\bibitem{NW} J.~Niksi{\'n}ski, B. Wr\'obel,
	\textit{Dimension-free estimates for discrete maximal functions and lattice points in high-dimensional spheres and balls with small radii}, J. Math. Pures Appl. 214, 2026.
\bibitem{Pa}
G. Paouris, \textit{Concentration of mass on convex bodies.} GAFA, Geom. funct. anal. 16, 1021–1049 (2006). https://doi.org/10.1007/s00039-006-0584-5
\bibitem{SZ}
G. Schechtman, J. Zinn. On the Volume of the Intersection of Two $L_n^p$ Balls. Proceedings of the American Mathematical Society, 110(1), 217–224, (1990). \url{https://doi.org/10.2307/2048262}
\bibitem{SZ2}
Schechtman, G., Zinn, J. Concentration on the $\ell^n_p$  ball. In: Milman, V.D., Schechtman, G. (eds) Geometric Aspects of Functional Analysis. Lecture Notes in Mathematics, vol 1745, (2000). Springer, Berlin, Heidelberg. \url{https://doi.org/10.1007/BFb0107218}
\bibitem{Ste1} E. M. Stein, \emph{Topics in harmonic analysis
  related to the Littlewood-Paley theory}, Annals of Mathematics
Studies, Princeton University Press 1970, pp. 1-157.

\bibitem{SteinMax} 
E.M. Stein, 
\textit{The development of square
  functions in the work of A. Zygmund}, 
	Bull.  Amer. Math. Soc. {\bf 7}, (1982), 359--376.
	

\bibitem{StStr} 
E.M. Stein, J.O. Str\"omberg, 
\textit{Behavior of maximal functions in $\mathbb{R}^n$ for large $n$}, 
Ark. Mat. {\bf 21}, (1983), 259--269.

\end{thebibliography}
\end{document}